\documentclass[11pt,a4paper]{article}

\usepackage[round]{natbib}
\usepackage{amsmath, amsthm, amssymb}
\usepackage[pdftex]{graphicx}
\usepackage{bm}
\usepackage{setspace}
\usepackage{fancyvrb}
\usepackage{hyperref}
\usepackage{color}
\usepackage[utf8]{inputenc}
\usepackage[a4paper, total={5.8in, 9.25in}]{geometry}
\usepackage[margin=44pt,labelfont=bf, labelsep=period]{caption}
\usepackage[section]{placeins}

\newtheorem{theorem}{Theorem}[section]
\newtheorem{corollary}{Corollary}[theorem]
\newtheorem{lemma}[theorem]{Lemma}
\newtheorem{proposition}[theorem]{Proposition}
\newtheorem{conjecture}[theorem]{Conjecture}

\numberwithin{equation}{section}
\numberwithin{table}{section}
\numberwithin{figure}{section}

\newcommand{\add}[1]{\textcolor{black}{#1}}
\newcommand{\define}[1]{\textit{#1}}

\newenvironment{myquote}{%
  \list{}{\leftmargin=2.6em\rightmargin=\leftmargin}%
  \item[]\ignorespaces%
}{%
  \endlist\ignorespacesafterend%
}

\begin{document}

\begin{center}
\Large{\textbf{Professor Preece's tredoku tilings}}\\[2ex]

\large{Martin Ridout}\footnote{University of Kent, UK. \texttt{msr@kent.ac.uk}}\\[3.5ex]

\small
\textbf{Abstract}\\[0.75ex]
\end{center}

\begin{myquote}
Shortly before he died in 2014, Donald Preece gave two talks about what he called tredoku tilings, inspired by the puzzle of the same name.
In these talks he presented a conjecture about the existence of these tilings that has been proved recently by Simon Blackburn. This
paper provides an overview of Donald's work in this area, including his work on a natural generalisation of a tredoku tiling that he
called a quadridoku tiling. Additionally, the paper gives alternative proofs of some parts of the existence theorem for tredoku tilings, presents
a computer enumeration of the isomorphism classes of tredoku tilings with up to 16 tiles and provides a brief introduction to tilings with holes.
\end{myquote}

\normalsize
\section{Introduction}

My friend Donald Preece, who died in 2014, was a British statistician and combinatorialist. Bailey (2014) provides an overview of his life and work.

Donald gave talks in July and October 2013 about his work on \emph{tredoku tilings}, inspired by the puzzle of the same name.\footnote{Tredoku is a
`three-dimensional' generalisation of sudoku. It is a trademark of Mindome Ltd.} We define tredoku tilings in Section~\ref{tredtilings} but for now note
that two important parameters of a tredoku tiling are the number of tiles, $\tau$, and the number of `runs', $\rho$; a run is a sequence of three tiles
that are edge-connected in a specific way.
In his talks, Donald proposed an answer to the question `For what values of $\tau$ and $\rho$ does a tredoku tiling exist?'. He showed that necessary
conditions are $\tau \ge 5$ and
\[
 \frac{3}{2} \rho \le \tau \le 2 \rho + 1,
\]
and was able to construct tilings for most combinations of $\tau$ and $\rho$ within this range. He conjectured that tilings do not exist for the
remaining combinations.

Recently, Blackburn (2024) provided a complete proof of this conjecture. He developed constructions similar to Donald's for combinations
of $\tau$ and $\rho$ for which tilings exist and gave proofs of non-existence for each of the remaining combinations. Some of these proofs are quite
intricate, particularly that for $\tau=12, \rho=8$.

Donald's \textit{Nachlass} included his work on tredoku tilings, consisting mostly of pictures of tilings and material prepared for the talks.
One objective of this article is to provide an archive of Donald's work. Appendix A contains the various constructions that he devised to establish
the existence of tilings and Appendix B shows the numerous tredoku tilings that he discovered.

Many of Donald's tilings are \define{reducible}, in the sense that they are constructed by merging two smaller tilings so that some tiles overlap.
Considering reducible and irreducible tilings separately is beneficial in two ways. First it facilitates computer enumeration of the isomorphism
classes of tilings. My interest in enumeration was partly to provide a computer check of Donald's claim that tilings with $\tau=12$ and $\rho=8$ and
$\tau=15$ and $\rho=7$ do not exist. Such a check is redundant following Blackburn's proof, but the enumeration results remain of interest and are
included later in the paper. The second benefit of studying reducible and irreducible tilings separately is that it leads to an alternative proof of
the non-existence parts of Donald's conjecture that differs considerably from Blackburn's proof.

We also introduce a graph associated with a tredoku tiling, termed the \emph{run graph},  and provide another proof of the non-existence of certain
tredoku tilings by demonstrating that the corresponding run graphs do not exist.

Donald also explored what he termed quadridoku and quindoku tilings. These are generalisations of tredoku tilings in which the runs of three tiles that
characterise tredoku tilings are replaced by runs of four or five tiles respectively. Appendices C and D show the tilings that Donald discovered. While
there are only a handful of quindoku tilings, Donald's work on quadridoku tilings was sufficiently extensive to allow him to formulate a conjecture about
the combinations of $\tau$ and $\rho$ for which they exist. Many aspects of this conjecture remain to be proved or disproved.

The abstract for the first of Donald's talks states that a tredoku tiling must not contain holes. However, this restriction is omitted from the abstract
of the second talk and his notes include a few tilings with a single large hole. We provide a brief introduction to tredoku tilings with holes, but this
topic remains largely unexplored.

The paper is organised as follows. Section~\ref{tredtilings} provides a leisurely introduction to tredoku tilings, including their definition and various
notions of equivalence of tilings. Section~\ref{sect:exist} provides a brief account of the existence theorem for tredoku tilings, while Section~\ref{sect:altprf}
gives a new proof of the non-existence aspects. Section~\ref{sect:enum} covers enumeration of tilings with 16 or fewer tiles and Section~\ref{sect:rungraphs}
introduces run graphs. Section~\ref{sect:quadridoku} discusses Donald's work on quadridoku and quindoku tilings and Section~\ref{sect:kdoku} gives a few results
about what we term $\kappa$-doku tilings, where runs consist of $\kappa$ tiles. Tilings with holes are introduced in Section~\ref{sect:holes} and we conclude in
Section~\ref{sect:discuss} with some suggestions for future work. Donald's tilings are shown in Appendices A--D; because they contain many figures, these
Appendices are provided as separate documents.

\section{Tredoku tilings}\label{tredtilings}

Donald's definition of a tredoku tiling appears in the abstracts of his 2013 talks. The version that is preserved in his notes and reproduced below, appears
to be the abstract for the second talk. Blackburn (2024) reproduces the abstract for the first talk, which is identical except for the additional requirement
that the tiling `must not have any holes in it'. Here, tilings will be assumed to have no holes, except in Section~\ref{sect:holes}.

This section provides examples of tredoku tilings and introduces some terminology from the mathematical tiling literature, generally following
Gr{\"u}nbaum \& Shephard (2016).

\begin{center}
\noindent\fbox{
    \parbox{0.9\textwidth}{
\begin{center}
\LARGE{Tredoku tilings}\\[0.5ex]
\Large{Donald A. Preece}\\[1ex]
\large{Queen Mary, University of London; University of Kent}\\[1.2ex]
\normalsize
\end{center}

We have a supply of identical tiles, each in the shape of a rhombus with alternate angles $60^{\circ}$ and $120^{\circ}$. We use $\tau$ of these tiles
($\tau > 4$) to form a tiling $\mathcal{T}$ of part of a plane. For convenience of description, we may regard $\mathcal{T}$ as a graph having sides and
vertices in the positions of the outlines of the positioned tiles.\\

To be a \textit{tredoku tiling}, $\mathcal{T}$ must have these properties:
\begin{enumerate}
\item [P1] $\mathcal{T}$ is connected, with no connection consisting merely of a single vertex;
\item [P2] If two tiles in $\mathcal{T}$ touch, either they do so only at a vertex or they have a side in common;
\item [P3] $\mathcal{T}$ cannot be disconnected by removing just one tile (but connection by just a vertex is allowed after the removal);
\item [P4]If tiles A and B share a side $p$ of $\mathcal{T}$, then there is a third tile C that shares a side $q$ of either A or B, where $q$ is parallel
to $p$. This gives a ``run'' of three tiles.
\item [P5] Runs of more than three tiles are not allowed.
\end{enumerate}
(There is no requirement for any part of the tiling to repeat.) If a tredoku tiling has $\rho$ runs of three tiles, then
\[
  \frac{3}{2} \rho \le \tau \le 2 \rho + 1.
\]
Many existence results and constructions are available for tredoku tilings. Motivation: If each tile in $\mathcal{T}$ is subdivided into a $3 \times 3 $ grid,
we have an overall grid for a tredoku$^{\copyright}$ puzzle such as appears daily in \textit{The Times}.
    }
}
\end{center}

The left panel of Figure~\ref{fig:fig2pt1} shows a tredoku tiling with seven tiles. Clearly, this tiling satisfies the defining properties P1 and P2.
It also satisfies property P3, but illustrates the necessity of allowing connection only via vertices after removing a tile (e.g.\ tile 2). Property P4
is satisfied, as the tiling has four runs of length three, with every tile belonging to at least one of these runs (Figure~\ref{fig:fig2pt2}). Finally,
there are no runs of length greater than three (P5).

\begin{figure}[htb!]
   \centering
   \makebox[\textwidth][c]{\includegraphics[trim=0in 6.8in 0in 0.0in, clip, width=0.8\textwidth]{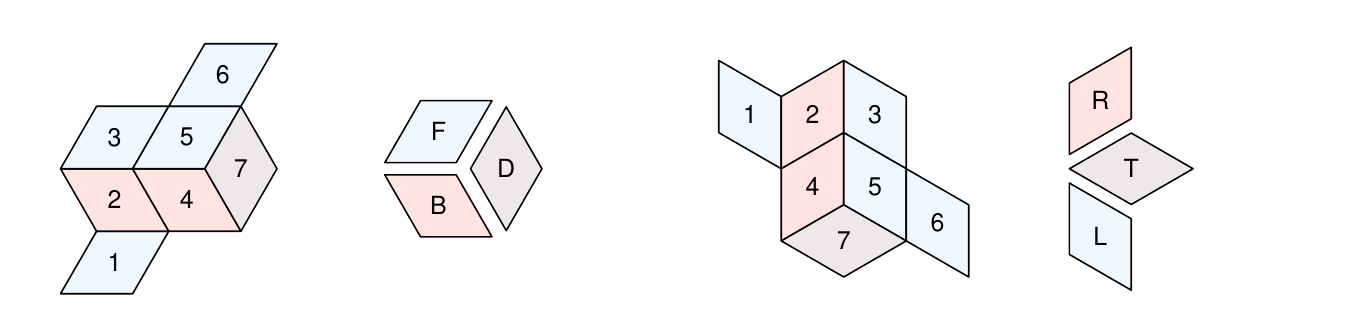}}
   \caption{The tredoku tiling dap7.4a, which has $\tau=7$ tiles and \mbox{$\rho=4$} runs, drawn in vertical format (left panels) and horizontal format
            (right panels). The three possible tile types are shown in each case.}
   \label{fig:fig2pt1}
\end{figure}

The colouring of the tiles in Figure~\ref{fig:fig2pt1} distinguishes their types and emphasises the three-dimensional appearance of the structure.
Similarly, the numbering helps to identify the individual tiles. However, neither the colouring nor the numbering is an inherent part of the tiling.

I use the notation 7.4 to denote a tiling with $\tau=7$ tiles and $\rho=4$ runs and prefix any tiling that appears in Donald's notes with his initials, `dap'.
Lower case letters are appended to distinguish tilings with the same values of $\tau$ and $\rho$.

\begin{figure}[htb!]
   \centering
   \makebox[\textwidth][c]{\includegraphics[trim=0in 7.7in 0in 0.01in, clip, width=0.8\textwidth]{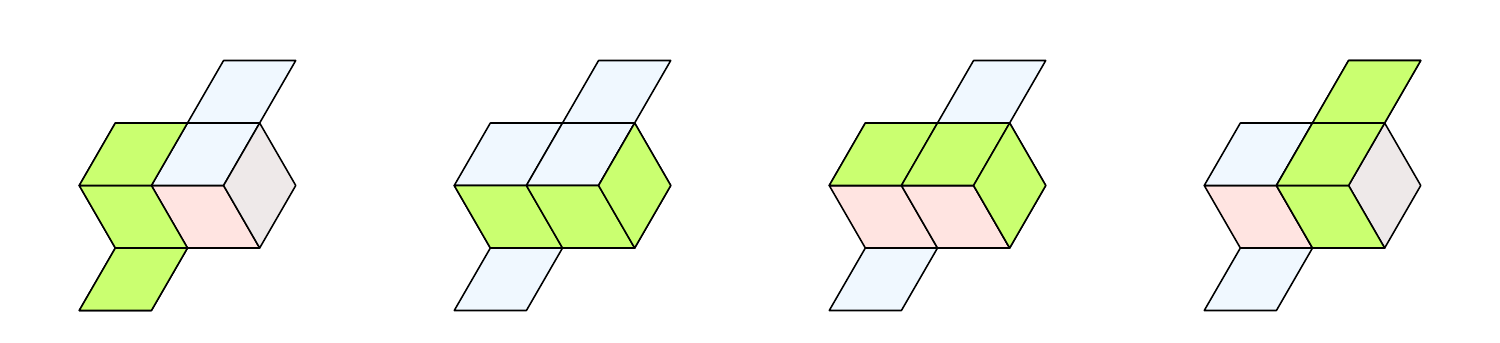}}
   \caption{The four runs of the tiling dap7.4a, highlighted in green.}
   \label{fig:fig2pt2}
\end{figure}

Tredoku tilings are \define{monohedral}, that is based on a single tile shape, sometimes called the \define{prototile}, but this can appear in up to three
different orientations, as indicated in the second panel of of Figure~\ref{fig:fig2pt1}. I shall refer to these as tile \define{types} and label them
B, D and F, for backward slanting rhombus, diamond and forward slanting rhombus respectively. Any tredoku tiling can be rotated so that it is made up solely
of tiles of these three types. I refer to this as the \define{vertical format} for drawing a tiling, based on the orientation of the D tile.

Alternatively, by rotating through 90$^\circ$ we may draw the tiling in \define{horizontal} format, as shown in the right two panels of Figure~\ref{fig:fig2pt1}.
Blackburn (2024) uses horizontal format and refers to the corresponding tile types as R(ight), T(op) and L(eft). I have generally used vertical format, but have
sometimes used horizontal format  when this fits more conveniently on the page, for example in Appendix A.

The prototile used in tredoku tilings consists of two equilateral triangles joined along a common edge and is often called a \define{lozenge}. Lozenge tilings
have been widely studied by both mathematicians and physicists; see Gorin (2021) for an overview.

\subsection{Connectedness} \label{connected}

We now consider properties P1--P3
in more detail. Two tiles in a tiling are said to be \define{adjacent} if they share an edge and are \define{neighbours} if their intersection is non-empty.
Property P2 states that neighbours in a tredoku tiling are either adjacent or have a single vertex in common. This property, together with the shape of the
prototile, implies that tredoku tilings are \define{normal} tilings (Gr{\"u}nbaum \& Shephard, 2016, Section 3.2).

The \define{dual graph} of a tiling $\mathcal{T}$ is the graph that has a vertex corresponding to each tile and an edge joining vertices $i$ and $j$ if and
only if the corresponding tiles are adjacent. It is also useful to introduce the \define{neighbour dual graph} or \define{n-dual graph}, which is a weighted
graph defined similarly but with an edge joining vertices $i$ and $j$ if and only if the corresponding tiles are neighbours. Edges that represent tiles that
have only a vertex in common are assigned weight 1, while edges that represent adjacent tiles are assigned weight 2.\footnote{These weights are arbitrary;
any two distinct values could have been chosen.} Figure~\ref{fig:dualgraph} shows the dual and n-dual graphs for the tiling shown in Figure~\ref{fig:fig2pt1}.

\begin{figure}[h!]
   \centering
   \makebox[\textwidth][c]{\includegraphics[trim=0in 2.7in 0in 2.7in, clip, width=0.8\textwidth]{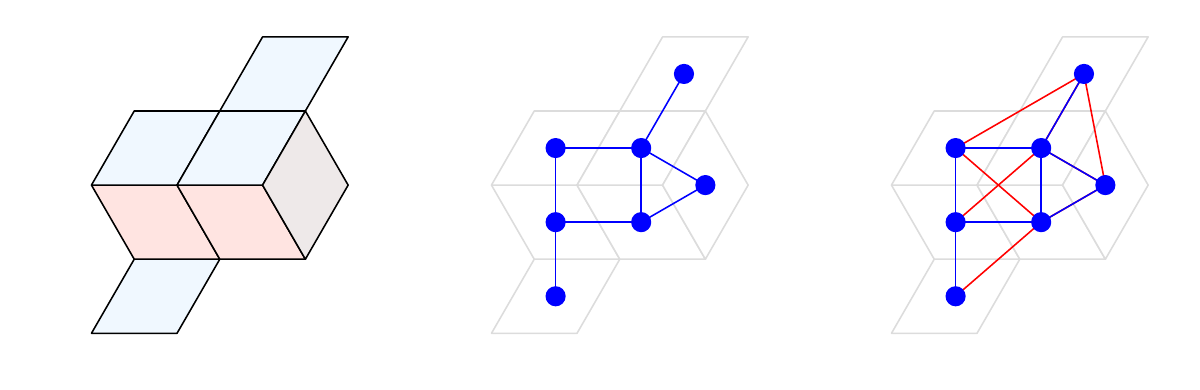}}
   \caption{The dual and n-dual graphs of the tiling dap7.4a. The left panel shows the tiling, the middle panel the dual graph and the right panel the n-dual
            graph; for the n-dual graph, the colours of the edges indicate their weights (red=1, blue=2).}
   \label{fig:dualgraph}
\end{figure}

Property P1 requires that for any pair of tiles $i$ and $j$ in $\mathcal{T}$, there is a sequence of tiles beginning with $i$ and ending with $j$ wuch that every pair
of consecutive tiles is adjacent. An equivalent condition is that the dual graph is connected in the standard graph-theory sense. A tiling that satisfies property P1
is \define{edge-connected}. Similarly, property P3 requires that the \mbox{n-dual} graph is 2-connected, i.e.\ it is connected and remains connected if any single
vertex is removed. Reframing properties P1 and P3 in terms of connectivity properties of graphs is useful for checking these properties by computer, because efficient
algorithms are readily available for checking graph connectivity.

Computer checking of property P2 can exploit the fact that if the property is \emph{not} satisfied then at least one tile contains in its interior a vertex of another
tile. Checking this is a well-studied problem in computational geometry, known as the point-in-polygon problem, for which several efficient algorithms are known (see,
e.g., Hormann \& Agathos, 2001).

\subsection{Runs} \label{sect:runs}

A minor difference between this paper and Blackburn (2024), is that I have required a run to consist of at least two tiles, rather than allowing a single tile to constitute a run of length one.

What Donald refers to as a \define{run}, has a variety of other names in the tiling literature, including \define{worm}, \define{de Bruijn line} and \define{ribbon}; see Frettl{\"o}h \& Harriss (2013) for references.

Following Blackburn (2024), we note that runs lie in one of three directions:
\begin{description}
\item \emph{BD runs} consist of B and/or D tiles, where the common edges are the E and W edges of the B tiles and the NE and SW edges of the D tiles.
\item \emph{BF runs} consist of B and/or F tiles, where the common edges are the N and S edges of the B and F tiles.
\item \emph{DF runs} consist of D and/or F tiles, where the common edges are the E and W edges of the F tiles and the NW and SE edges of the D tiles.
\end{description}
Figure~\ref{fig:directions} gives examples of each type. Note that a horizontal run of three B tiles lies in a different direction to a horizontal run of three F tiles.\\[-1ex]
\begin{figure}[htb!]
   \centering
   \makebox[\textwidth][c]{\includegraphics[trim=0in 7.88in 0in 0.0in, clip, width=1\textwidth]{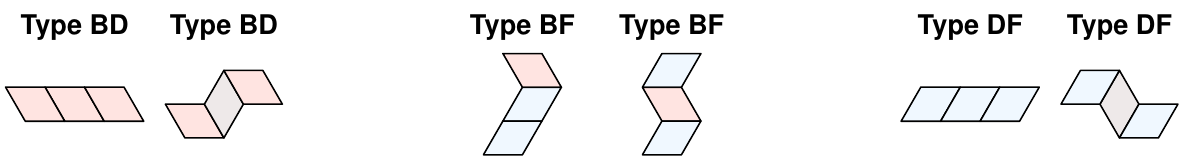}}
   \caption{Examples of runs in each of the three possible directions.}
   \label{fig:directions}
\end{figure}

Tredoku tilings are part of broader class of tilings in which all runs have the same length, $\kappa$, which we refer to as \define{$\kappa$-doku} tilings.\footnote{Like
tredoku, and Donald's terms quadridoku and quindoku, the name $\kappa$-doku makes no sense in relation to the etymology of the word sudoku.} Donald's work on
$\kappa$-doku tilings, mostly involving $\kappa=4$ or $\kappa=5$,  is discussed in Sections~\ref{sect:quadridoku} and~\ref{sect:kdoku}.

\subsection{Tiling annotations} \label{annotations}

Donald annotated many of the tilings that he discovered in various ways that are illustrated in Figure~\ref{fig:fig2pt3}, which shows dap7.4a and the three other
7.4 tilings that Donald found.

\begin{figure}[h!]
   \centering
   \makebox[\textwidth][c]{\includegraphics[trim=0in 5.9in 0in 0.0in, clip, width=0.8\textwidth]{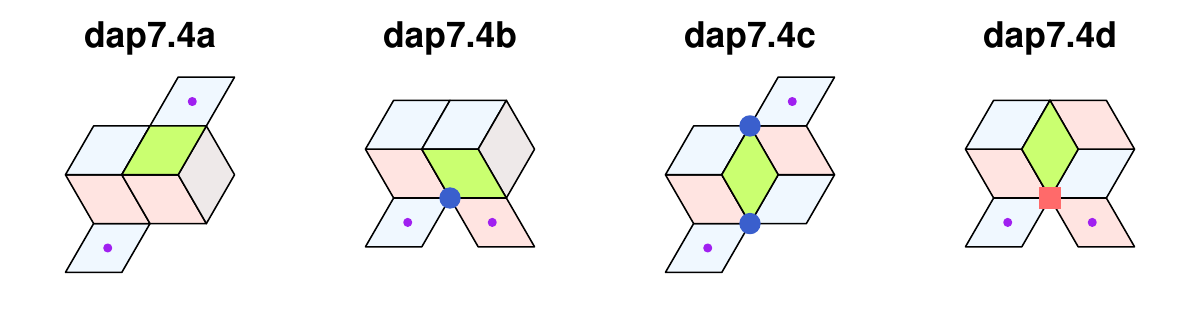}}
   \caption{The four 7.4 tilings with Donald's annotations. Green tiles are 4--tiles and blue dots and red filled squares indicate vertices where five or six tiling
            edges meet, respectively. A small dot indicates a leaf tile. }
   \label{fig:fig2pt3}
\end{figure}

The first annotation highlights tiles that appear in the central position of two distinct runs, which I refer to as \define{4--tiles}, because they are adjacent
to four other tiles; \mbox{3--, 2-- and 1--tiles} are defined similarly. As discussed in Section~\ref{sect:altproof}, 4--tiles are relatively uncommon. Each of the 7.4
tilings has exactly one 4--tile.

The second annotation highlights vertices that are the meeting point of either five edges (highlighted with a blue filled circle) or six edges (highlighted with a red
filled square). The number of edges that meet at a vertex is the \define{valency} of that vertex,  so these are \define{5-valent} and \define{6-valent} vertices
respectively. Again they are relatively uncommon and each of the four 7.4 tilings has a different pattern of 5- and 6-valent vertices.

The final annotation indicates the \define{leaves} of the tiling with a small dot. Donald introduced the term leaf to describe a tile that is adjacent to only one other tile.
Therefore, a leaf is the same as a 1--tile. Blackburn (2024) instead defines a leaf to be a tile that appears in exactly one run of length 3. The two definitions
are equivalent for tilings without holes, but for tilings with holes, Blackburn's definition turns out to be more useful (Section~\ref{sect:holes}).

The valency of a tredoku tiling vertex can range from 2 to 6 whereas the number of tiles that share a vertex can range from 1 to 6. The smallest example of a tiling
where six tiles meet at a vertex is dap10.6j (Figure~\ref{fig:dap106j}).

\begin{figure}[h!]
   \centering
   \makebox[\textwidth][c]{\includegraphics[trim=0in 5.2in 0in 0.01in, clip, width=0.55\textwidth]{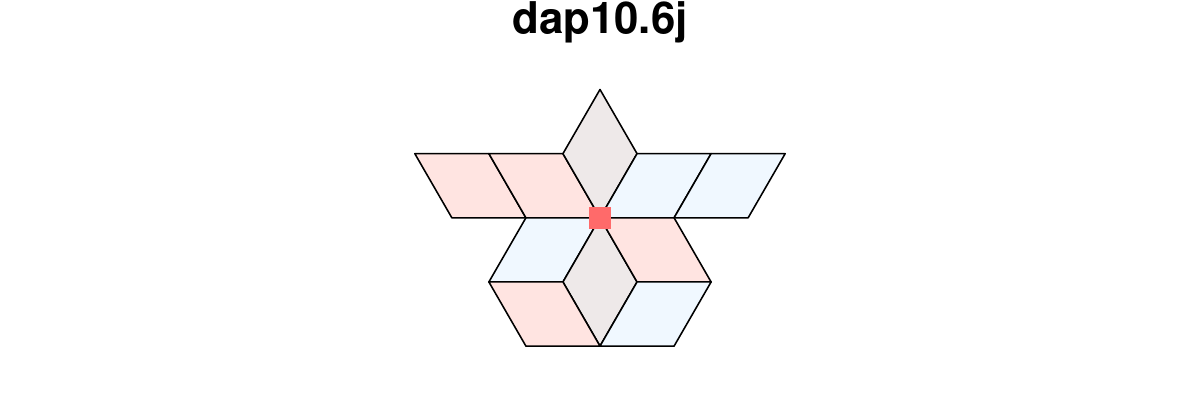}}
   \caption{The smallest tredoku tiling that has a vertex that is shared by 6 tiles (highlighted).}
   \label{fig:dap106j}
\end{figure}

\subsection{Equivalent tilings}

Donald presented only two tredoku tilings with 6 tiles and 3 runs, dap6.3a and dap6.3b (Figure~\ref{fig:transform2}). Many other tilings exist, but he evidently regarded
these as equivalent to either dap6.3a or dap6.3b. This raises the question of what it means for two tilings to be equivalent. We shall consider three types of equivalence,
congruence, isomorphism and flip-equivalence. Although Donald's notes do not discuss this issue, it is clear that he considered two tilings to be equivalent
if are isomorphic. It is sometimes useful to regard a designation such as dap7.4a as referring to an isomorphism class of tilings rather than a particular tiling.

\subsubsection{Congruence}

One simple type of equivalence is congruence; two tilings are \define{congruent} if one can be transformed to the other by an \define{isometry}, i.e.\ a transformation
that preserves distances and angles. Plane isometries involve combinations of translations, rotations and reflections, but because tredoku tilings are determined by the
relative rather than absolute locations of their tiles, we need only consider rotations and reflections.

If we require the tiling to be oriented so that all tiles are of type B, D or F then the only isometries that we need to consider are those that map these tile types onto
one another. There are just twelve such transformations. They are defined in terms of the \define{centroid} of the tiling, $C$, which is the point whose $(x,y)$
coordinates are the arithmetic means of the $x$- and $y$-coordinates of the tile vertices. The twelve transformations comprise, rotations about $C$ through an angle
\mbox{$\theta = j \pi / 3$} \mbox{($j = 0, \ldots, 5)$}, and reflections in the line through $C$ that makes an angle \mbox{$\phi = j \pi / 6$} \mbox{$ (j = 0, \ldots, 5)$}
with the horizontal axis. These are the symmetries of the regular hexagon and form the \define{dihedral group} $D_6$ under composition.

The tilings dap6.3a and dap6.3b are not congruent. However, there are numerous other 6.3 tilings that are not congruent to either dap6.3a or dap6.3b. For example, none
of the tilings in Figure~\ref{fig:transform2} are congruent. On the other hand, variants 1--3 are only minor variations on dap6.3a, they simply modify
the type of one or more of the leaf tiles. There is no change to the basic structure of the tiling in terms of which tiles are adjacent and which are neighbours.

\begin{figure}[h!]
   \centering
   \makebox[\textwidth][c]{\includegraphics[trim=0in 6.1in 0in 0.01in, clip, width=0.9\textwidth]{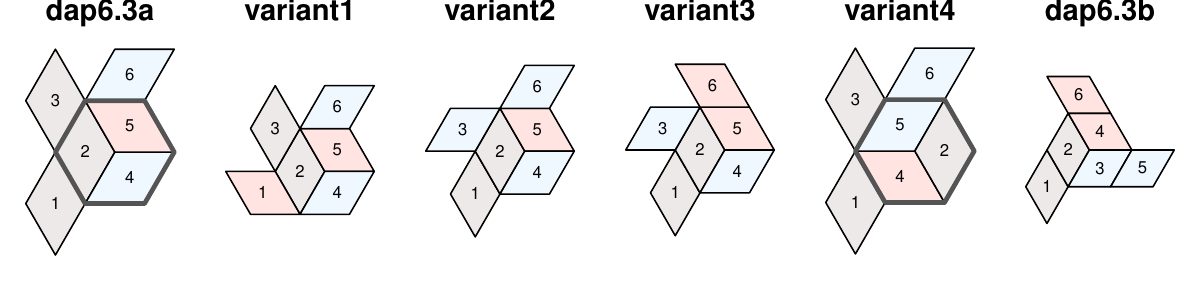}}
   \caption{The tilings dap6.3a on the far left and dap6.3b on the far right. In between are four variants of dap6.3a that are not rotations or reflections.}
   \label{fig:transform2}
\end{figure}

\subsubsection{Isomorphism}

For variants 1–3 of Figure ~\ref{fig:transform2} to be considered equivalent to dap6.3a, a weaker notion of equivalence than congruence is evidently needed.
This is provided by combinatorial equivalence; two tilings are \define{combinatorially equivalent} or (combinatorially) \define{isomorphic} if it is possible
to label the tiles of each tiling with the symbols $1, 2, \ldots, \tau$ in such a way that the \define{incidence relations} of the two tilings are identical.
Two elements of a tiling are \define{incident} if their intersection is non-empty (Gr{\"u}nbaum \& Shephard, 2016, p.18). For example, a tile is incident with its
edges and its vertices, and vice versa, and an edge is incident with the vertices at each end, and vice versa.

Gr{\"u}nbaum \& Shephard (2016, p.\,168) give a more formal definition of combinatorial equivalence and show that for normal tilings, which include tredoku tilings,
it is equivalent to \define{topological equivalence}. With this definition, all four variants shown in Figure~\ref{fig:transform2} are isomorphic to dap6.3a.

It can be difficult to assess visually whether two tilings are isomorphic and it seems likely that Donald used the various annotations discussed in
Section~\ref{annotations} to help with this, since the features that are highlighted (leaves, 4--tiles and 5- and 6-valent vertices) are preserved under isomorphism.
To determine whether two tilings are isomorphic using a computer, we can check whether their 1-skeletons are isomorphic, where the 1-skeleton of a tiling is the graph
comprising its vertices and edges. I used the bliss algorithm (Junttila, \& Kaski, 2007, 2011) to determine whether two skeletons are isomorphic.

\subsubsection{Flip-equivalence} \label{sect:flip}

Figure~\ref{fig:kast1} shows the two ways of forming a regular hexagon using a single tile of each type. The operation of switching from one form to the other is
known as a \define{flip} and is used widely in the study of  lozenge tilings.

\begin{figure}[h!]
   \centering
   \makebox[\textwidth][c]{\includegraphics[trim=0in 6.2in 0in 0.01in, clip, width=0.9\textwidth]{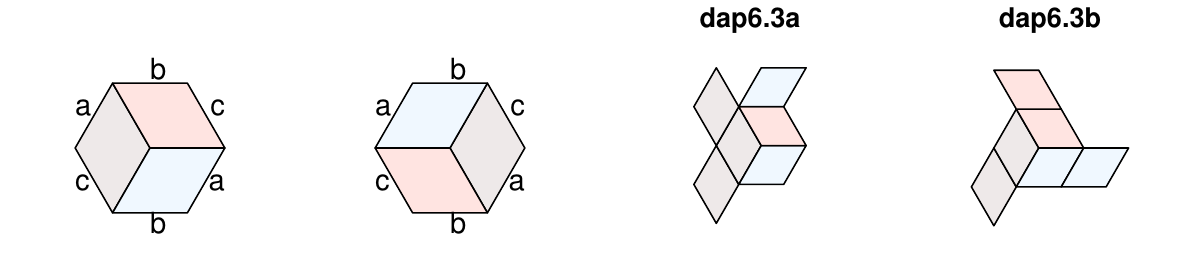}}
   \caption{The two possible ways of forming a hexagon using a single tile of each type. Tilings dap6.3a and 6.3b are the smallest tredoku tilings that
            include such a hexagon.}
   \label{fig:kast1}
\end{figure}

If a tredoku tiling contains such a hexagon, it remains a tredoku tiling even after the hexagon is flipped. This follows from the fact that, in order to have runs
of length three, there must be a tile attached to exactly one of each pair of opposite sides labelled a, b and c in Figure~\ref{fig:kast1} and, however the
tiles are attached, this will still give three runs of length three after flipping the hexagon.

Flipping may or may not preserve isomorphism. For example, flipping the hexagon in either dap6.3a or dap6.3b (Figure~\ref{fig:kast1}) gives a tiling that is isomorphic
to the original tiling. On the other hand, flipping the hexagon in dap7.4a gives the non-isomorphic tiling dap7.4c (Figure~\ref{fig:fig2pt3}).
Flipping the newly formed hexagon gives a tiling that is isomorphic to dap7.4a. We shall say that tilings dap7.4a and dap7.4c are \define{flip-equivalent}. More
generally, two tilings are \define{flip-equivalent} if one can be transformed to the other by a sequence of flips. Tilings dap7.4b and
dap7.4d are also flip-equivalent. Thus, although there are four isomorphism classes of 7.4 tilings, there are only two flip-equivalence classes.

More generally we might consider the simply connected region of the plane, $\mathcal{R}$, covered by a tredoku tiling. Are there alternative lozenge tilings
of $\mathcal{R}$, and if so how many? Will any alternative tilings still be tredoku tilings?

The first question can be addressed using `Kasteleyn-Temperley-Fisher theory', as described in Gorin (2021 Chapter 2). The starting point is the \define{Kasteleyn
matrix} corresponding to the original tiling. To generate this matrix,  first divide each tile into two equilateral triangles and mark them with black and white discs as
illustrated in the right three panels of  Figure~\ref{fig:kast2}. This creates a chequerboard pattern of alternating black and white triangles covering
$\mathcal{R}$. Triangles that share a common edge are said to be \define{adjacent}.

\begin{figure}[h!]
   \centering
   \makebox[\textwidth][c]{\includegraphics[trim=0in 6.2in 0in 0.0in, clip, width=0.9\textwidth]{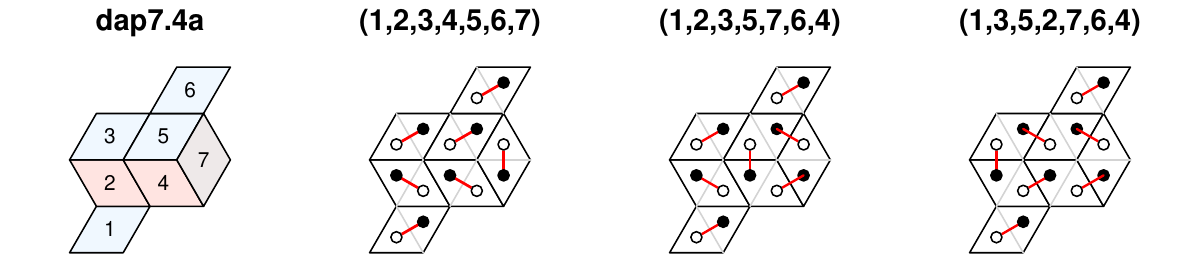}}
   \caption{The left panel shows a possible numbering of dap7.4a. The right three panels show the triangular grid of black and white dots and one of the three possible
            perfect matchings of the two colours. Each of these perfect matchings corresponds to a possible tiling of the region of the plane covered by dap7.4a.}
   \label{fig:kast2}
\end{figure}

Now number the tiles from 1 to $\tau$ and assign these numbers to both triangles. The left panel of Figure~\ref{fig:kast2} shows one possible numbering for the
tiling dap7.4a. The Kasteleyn matrix is the $\tau \times \tau$ matrix $K = \left( k_{ij} \right)$, where
\[
       k_{ij} = \begin{cases}
        1 & \quad \mathrm{if\ black\ triangle\ } i \mathrm{\ is\ adjacent\ to\ white\ triangle\ } j \\
        0 & \quad \mathrm{otherwise}.
        \end{cases}
\]
Thus, for dap7.4a
\[
   K = \left( \begin{array}{ccccccc}
   1  &  1  &  0 &  0 &  0 &  0  &  0 \\
   0  &  1  &  1 &  0 &  0 &  0  &  0 \\
   0  &  0  &  1 &  0 &  1 &  0  &  0 \\
   0  &  1  &  0 &  1 &  1 &  0  &  0 \\
   0  &  0  &  0 &  0 &  1 &  1  &  1 \\
   0  &  0  &  0 &  0 &  0 &  1  &  0 \\
   0  &  0  &  0 &  1 &  0 &  0  &  1
       \end{array} \right).
\]

The original tiling pairs each black triangle with the white triangle of the same number, as indicated by the red lines in the second panel of Figure~\ref{fig:kast2}.
This is an example of a \define{perfect matching} of the black and white triangles; each black triangle is matched to a distinct adjacent white triangle. Equally,
since any two adjacent triangles can be combined to give a lozenge tile, any perfect matching of the black and white triangles corresponds to a lozenge tiling of
$\mathcal{R}$. For dap7.4a, there are two other perfect matchings which are shown in the right two panels of Figure~\ref{fig:kast2}. Perfect matchings can be identified
by a permutation $\pi$, where black triangle $i$ is matched to white triangle $\pi(i)$.

Thus, to count the number of distinct lozenge tilings of $\mathcal{R}$, we can count perfect matchings, but it turns out that this is equivalent to simply
evaluating $|\det(K)|$; for a proof, see Gorin (2021, Section 2.1).

The second question is answered by a result of Kenyon (1993), that if $\mathcal{R}$ is simply connected then any two tilings of $\mathcal{R}$ are flip-connected.
It follows that all lozenge tilings of $\mathcal{R}$ are tredoku tilings.

\section{The existence theorem for tredoku tilings} \label{sect:exist}

Donald's existence theorem for tredoku tilings may be stated as follows:

\begin{theorem}[Preece and Blackburn]
A tredoku tiling with $\tau$ tiles and $\rho$ runs must satisfy
\begin{equation}
 \max \left(5, \frac{3 \rho}{2} \right)\le \tau \le 2 \rho + 1.
\label{eq:inequality1}
\end{equation}
Subject to these constraints:
\begin{itemize}
\item[(i)]
For $\rho \ge 3$, tredoku tilings exist for all $\tau \le 2 \rho$, except for the combinations 5.3, 6.4 and 12.8.
\item[(ii)]
The only tredoku tilings that exist with $\tau = 2\rho + 1$ are 5.2, 7.3, 9.4, 11.5, 13.6 and 17.8.
\end{itemize}
\label{thm:thm1}
\end{theorem}

Donald's proof of the inequalities~(\ref{eq:inequality1}) was as follows. First, a tiling that consists of a single run of three tiles is not a tredoku tiling
because it becomes disconnected if the central tile is removed. Therefore a tredoku tiling must contain at least two runs, and since distinct runs can have at
most one tile in common, we must have $\tau \ge 5$.

Next, every tile in a tredoku tiling occurs in either one or two runs of length three. Let $\lambda$ denote the number of tiles that occur in a single run (these
tiles are the leaves of the tiling). Then, counting the number of tiles in all runs in two ways, $\lambda + 2(\tau-\lambda) = 3 \rho$. Hence
\begin{equation}
    \lambda = 2 \tau - 3 \rho,
\label{eq:leaves}
\end{equation}
and since $\lambda \ge 0$ we have
  $ \tau \ge {3 \rho} \slash {2}$.

Finally, because of connectivity, every run except the first must have at least one tile in common with an earlier run. Therefore,
$\tau \le 3 + 2 (\rho-1) = 2 \rho + 1$.\footnote{Blackburn (2024) gives a different argument in the proof of his Lemma 3.}

Inequalities~(\ref{eq:inequality1}) may also be expressed in terms of the number of leaves as
\begin{equation}
   0 \le \lambda \le \rho+2.
\label{eq:lambda_inequality2}
\end{equation}
Tilings with $\lambda=0$ are called \define{leafless} and Blackburn (2024) defines \define{verdant} tilings to be those with $\tau = 2 \rho + 1$ and hence
$\lambda=\rho+2$. It follows from equation~(\ref{eq:leaves}) that leafless tilings can exist only for even $\rho$ and according to part (i) of the theorem,
they exist for all $\rho \ge 6$ except for $\rho=8$. On the other hand, part (ii) of the theorem says that verdant tredoku tilings exist only for a few values
of $\rho$.

As we have remarked earlier, Donald gave constructions for all combinations of $\tau$ and $\rho$ satisfying inequalities~(\ref{eq:inequality1}) for which tilings
exist and we give details of these constructions below. In his proof of Donald's theorem, Blackburn (2024) developed constructions that are similar in spirit,
but generally different in detail.

Figure~\ref{fig:exist1} gives examples of 5.2, 7.3, 8.5, 9.4, 9.6, 11.5, 11.7, 13.6, 15.10 and 17.8 tilings.
Apart from the 8.5 and 15.10 tilings, these tilings are  unique up to isomorphism and they include all of the tilings mentioned in part (ii) of the theorem.
\begin{figure}[h!]
   \centering
   \makebox[\textwidth][c]{\includegraphics[trim=0in 2.2in 0in 0.01in, clip, width=0.9\textwidth]{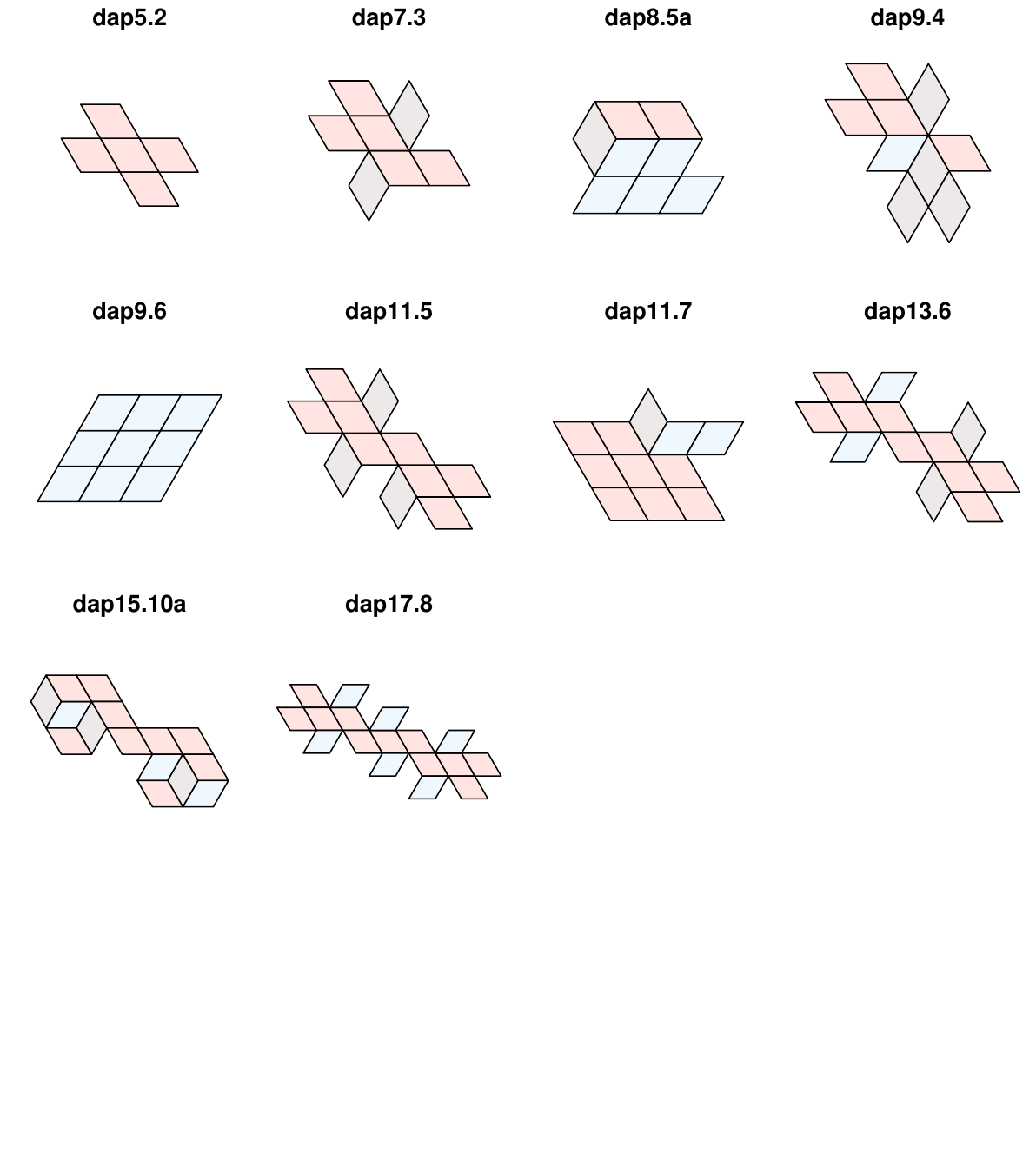}}
   \caption{Tredoku tilings for specific values of $\tau$ and $\rho$ that are not generated by the general constructions described in the text.}
   \label{fig:exist1}
 \end{figure}

For all other combinations of $\tau$ and $\rho$, there is a construction that yields an infinite sequence of tilings, including a tiling with the specified
values of $\tau$ and $\rho$. Donald's notes include 28 such constructions, which are shown in Appendix A. Twenty five of these are what he called \define{internal
extensions} because they modify the interior of a tiling. There are also three \define{external extensions} that add tiles to the external boundary of a tredoku tiling.

Figure~\ref{fig:exist2} provides an example of an internal extension. It shows a 17.9 tredoku tiling within which two small segments have been identified, labelled A
and B. New tredoku tilings can be generated by repeating these segments as many times as desired. In addition, segment A can be removed entirely. If there are $a$
copies of segment A ($a \ge 0$) and $b$ copies of segment B ($b \ge 1$) then the resulting tiling has $\tau = 10 + 4a + 3b$ tiles with $\rho=5+2a+2b$ runs, implying
that $\rho$ is odd and $\rho \ge 7$, and that $\tau=2\rho-b$.

 \begin{figure}[h!]
   \centering
   \makebox[\textwidth][c]{\includegraphics[trim=0in 1.5in 0in 1.5in, clip, width=0.52\textwidth]{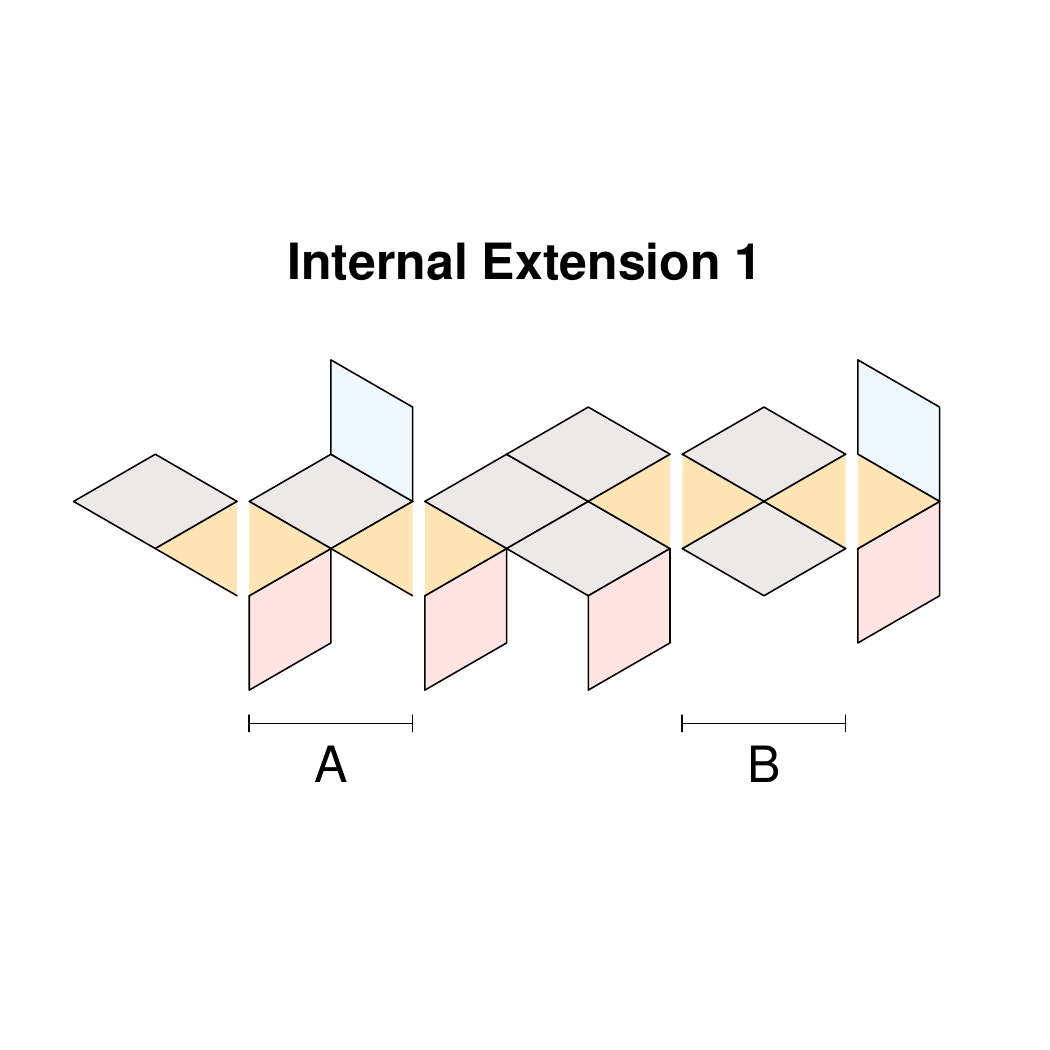}}
   \caption{In this construction, which is drawn in horizontal format, segment A is repeated $a$ times \mbox{($a \ge 0$)} and segment B is repeated $b$ times
            ($b \ge 1$), giving {$\tau = 10 + 4a + 3b$} and $\rho = 5+2a+2b$, so that $\tau = 2\rho - b$ for odd $\rho \ge 7$.}
   \label{fig:exist2}
 \end{figure}

Many of the internal extensions have either an A segment or a B segment, but not both. The segment can be repeated as many times as desired to generate an infinite
sequence of tredoku tilings. The distinction, as above, is that an A segment can be removed, but a B segment cannot.

Eight of the internal extensions suffice to generate tredoku tilings for all combinations of $\tau$ and $\rho$ for which tilings exists, apart from the combinations
shown in Figure~\ref{fig:exist1}. Table~\ref{tab:exist1} lists these extensions and the range of values of $\tau$ and $\rho$ that they generate.

\begin{table}[h!]
\centering
\caption{Internal extensions used to confirm the existence of tredoku tilings with $\tau$ tiles and $\rho$ runs. For a given value of $\rho$, there is an internal
         extension that yields a tiling with $\tau$ tiles for all $\tau$ satisfying $ 3 \rho / 2  \le \tau \le 2\rho$.}
\label{tab:exist1}
\begin{tabular}{crrcc}
\hline
Internal & \multicolumn{2}{c}{Number of runs, $\rho$}  & Number of  & Number of\\
Extension & Parity & Range & tiles, $\tau$ & leaves, $\lambda$\\
\hline
 18 &   odd & $\rho \ge 5$  & $\tau = 2\rho$ & $\lambda = \rho$ \\
  1 &   odd & $\rho \ge 7$  & $( 3 \rho + 5) / 2 \le \tau \le 2\rho - 1$ & odd, $5 \le \lambda \le \rho-2$\\
 10 &   odd & $\rho \ge 3$  & $\tau = (3\rho+3)/2$ & $ \lambda = 3$\\
 14 &   odd & $\rho \ge 9$  & $\tau = (3\rho+1)/2$ & $ \lambda = 1$ \\[1.8ex]

 16 &  even & $\rho \ge 4$  & $\tau = 2\rho$ & $\lambda = \rho$ \\
  5 &  even & $\rho \ge 6$  & $ (3 \rho + 4) / 2  \le \tau \le 2\rho - 1$ & even, $4 \le \lambda \le \rho-2$\\
 11 &  even & $\rho \ge 4$  & $\tau = (3\rho+2)/2$ & $\lambda = 2$\\
 24 &  even & $\rho \ge 12$ & $\tau = 3\rho/2$ & $\lambda = 0$\\
\hline
\end{tabular}
\end{table}

Figure~\ref{fig:proof} shows how the different extensions combine to cover the range of allowable combinations of $\tau$ and $\rho$. The boxes indicate the combinations
that are potentially allowable according to inequalities~(\ref{eq:inequality1}). Boxes marked with an asterisk are the combinations of $\tau$ and $\rho$ for which example
tredoku tilings are given in Figure~\ref{fig:exist1}. Shaded boxes are combinations for which no tiling exists. The numbers in the remaining boxes indicate the internal
extension that may be used to construct a $\tau.\rho$ tiling.

 \begin{figure}[h!]
   \centering
   \makebox[\textwidth][c]{\includegraphics[trim=0in 0.2in 0in 0.5in, clip, width=0.95\textwidth]{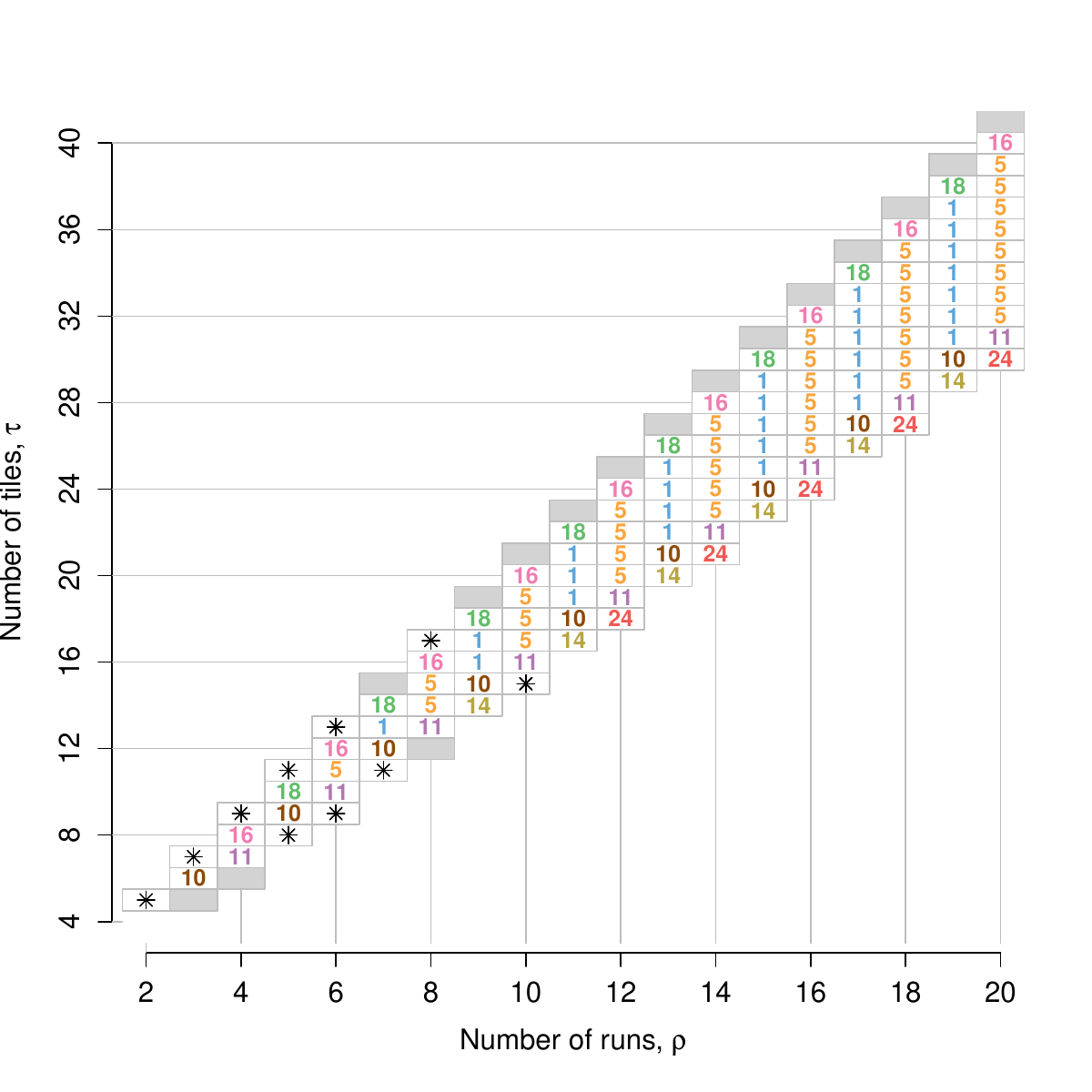}}
   \caption{This graphic shows how the constructions listed in \mbox{Table 1} combine together. Numbers in cells indicate the construction that is used to create a tiling
            with the given values of $\tau$ and $\rho$. Asterisks indicate that a tredoku tiling is shown in Figure~\ref{fig:exist1}. Shaded cells indicate combinations of
            $\tau$ and $\rho$ for which no tredoku tiling exists.}
   \label{fig:proof}
 \end{figure}

\section{An alternative proof of the existence theorem for tredoku tilings}\label{sect:altprf}

Donald noted that many of the tilings that he found could be generated by merging two or more smaller tredoku tilings, indeed it seems likely that this is how
he obtained many of these tilings. We call such tilings \define{reducible}. Tilings that cannot be generated in this way are called \define{irreducible}.

After merging, the two component tilings of a reducible tiling share either one or two tiles, as illustrated in Figure~\ref{fig:combine1}. We use the terms \define{single merging}
and \define{double merging} to distinguish these.\footnote{The term \define{merging} originates with Blackburn (2024), who uses it for what we call a single merging.}

\begin{figure}[h!]
   \centering
   \makebox[\textwidth][c]{\includegraphics[trim=0in 5.5in 0in 0.001in, clip, width=0.9\textwidth]{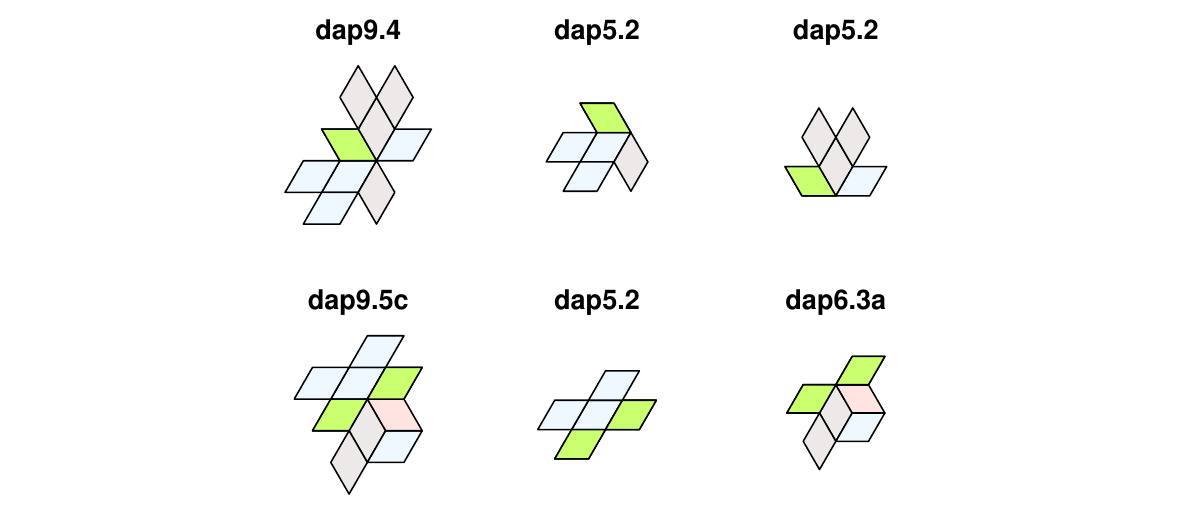}}
   \caption{Two examples of merging tredoku tilings to produce a larger tiling. The top row shows a single merging and the bottom row a double merging.}
   \label{fig:combine1}
\end{figure}

The top row of Figure~\ref{fig:combine1} shows, on the left, the tiling dap9.4, which may be obtained by merging the two variants of dap5.2 on the right so that the
highlighted tile overlaps. On the bottom row, tiling dap9.5c on the left may be obtained similarly by merging variants of dap5.2 and dap6.3a, but this time by overlapping
two tiles.

Considering irreducible and reducible tilings separately provides an alternative approach to establishing some aspects of the existence theorem for tredoku tilings that
has little in common with Blackburn's proof. Specifically, we show in Theorem~\ref{thm:thm2} that irreducible tredoku tilings exist only for a subset of the parameter
values allowed by Theorem~\ref{thm:thm1}. In particular, there is no irreducible 12.8 tiling and no irreducible verdant tilings with $\rho > 3$. Then we show, by examining
a small number of possibilities, that no reducible 12.8 tiling exists and that reducible verdant tilings exist only for $\rho = 4, 5, 6$ and $8$.

\subsection{Reducible and irreducible tilings}

To begin with, recall that an $i$-tile is a tile that is adjacent to  $i$ other tiles. 1--tiles are leaves, 3--tiles lie at the end of one run and the centre of another
and 4--tiles lie at the centre of two runs. The following lemma characterises 2--tiles.
\begin{lemma}
A 2--tile lies at the end of two runs.
\label{lem:type2}
\end{lemma}
\begin{proof}
Let $X$ be a 2--tile. Because $X$ is adjacent to two other tiles, it lies either at the end of two runs or at the centre of a run. In the latter case, there is a tile
$A$ that shares an edge $a$ with $X$ and a tile $B$ that shares the parallel edge $b$ of $X$, whilst the other two edges of $X$, $c$ and $d$, are not shared with another tile.
Suppose $\mathcal{P}$ is a path from $A$ to $B$ that does not include $X$. Then the cycle $\mathcal{P}, X, A$ contains one of the edges $c$ or $d$ in its interior and the
tiling therefore contains a hole with this edge on its boundary. Therefore, the only path from $A$ to $B$ is $A$, $X$, $B$. But then if tile $X$ is removed there is no path
from $A$ to $B$ and the tiling is no longer connected. Thus $X$ cannot lie at the centre of a run.
\end{proof}

The reducibility of a tiling may be determined by examining the  \define{spine} of the tiling, defined  to be the set of its 3- and 4--tiles. If the spine is edge-connected,
we say that the tiling is \define{irreducible}. Otherwise, the spine consists of two or more components, each of which is connected and is the spine of a smaller irreducible
tiling. In this case the original tiling is \define{reducible} and results from merging the tilings corresponding to each spine component. These claims are formalised in the
following lemmas.

\begin{lemma}
Every spine component shares a vertex with at least one other spine component, but no two components share more than a single vertex.
\label{lem:lem3}
\end{lemma}

\begin{proof}
Choose a tile $A_1$ in a spine component $A$ and another tile $B_1$ in a different spine component $B$. Because the tiling is connected, but the spine components are not,
any path from $A_1$ to $B_1$ must include at least one non-spine tile. The path cannot include any 1--tiles, since these are adjacent to only one other tile, and so the
non-spine tiles are 2--tiles. Clearly, from Lemma~\ref{lem:type2}, two 2--tiles cannot be adjacent, and so the path must have the form  $A_1, \ldots, A_r, X, C_1, \ldots B_1$,
where $A_r$ is the last tile that belongs to spine component $A$, $X$ is a 2--tile, and $C_1$ belongs to a spine component $C$, which may or may not coincide with component
$B$. The edges of tile $X$ that are shared with tiles $A_r$ and $C_1$ are not parallel, because tile $X$ does not lie in the middle of a run,
and therefore there is a vertex that is shared by tiles $A_r$, $X$, and $C_1$. Thus spine components $A$ and $C$ share a vertex.

Suppose $A$ and $C$ share two vertices. Then we could create a cycle starting and ending at $C_1$ and passing through $A_r$. Because the tiling does not contain holes,
the interior of this cycle would be filled with tiles which would necessarily be 4--tiles. But then $A$ and $C$ would be connected, contradicting the premise that they
are separate components. Therefore, two components cannot share more than a single vertex.
\end{proof}

\begin{lemma}
The sub-tiling consisting of the tiles of a spine component, together with any tiles that are adjacent to these tiles, is a tredoku tiling.
\label{lem:lem4}
\end{lemma}

\begin{proof}
Clearly the sub-tiling is connected and remains connected if any tile that is not part of the spine is removed. It also remains connected if any tile of the spine is
removed, because otherwise removing this tile would disconnect the full tiling. Every tile in the spine lies at the centre of one or two runs and the inclusion of all
adjacent tiles ensures that the complete run is included in the sub-tiling. Finally, because it is a sub-tiling of a tredoku tiling, there are no runs of length
greater than three.
\end{proof}

\begin{lemma}
The tredoku tilings associated with two spine components that share a common vertex have either one or two tiles in common.
\label{lem:lem5}
\end{lemma}

\begin{proof}
The 2--tile denoted by $X$ in the proof of Lemma~\ref{lem:lem3} is part of the tredoku tilings associated with the spine components $A$ and $C$ that share a common
vertex, so there is always at least one such tile. Figure~\ref{fig:combine1} shows that the tredoku tilings may also have two tiles in common.

Suppose the tilings had three tiles in common. Each such tile would have an edge in common with a tile from the first spine component, and another edge in common with
a tile from the second spine component. Moreover, all of these tiles would have as one of their vertices the vertex that is shared by the two spine components. But
this is not possible, because the maximum number of tiles that can meet at the vertex of a lozenge tiling is six.
\end{proof}

\begin{lemma}
If a tredoku tiling is a double merging of two tredoku tilings, $T_1$ and $T_2$, then the two merged tiles do not lie in the same run in either $T_1$ or $T_2$.
\label{lem:lem5b}
\end{lemma}
\begin{proof}
This follows from the fact that the two tiles are leaf tiles that share a vertex and therefore cannot lie at opposite ends of the same run.
\end{proof}

\subsection{The alternative proof}\label{sect:altproof}

Let $n_i$ denote the number of $i$--tiles in a tiling. Counting the number of tiles that lie at the ends of runs and in the centre of runs gives
\begin{align}
    n_1 + 2 n_2 + n_3 &=  2 \rho,  \label{eq:n2} \\
    n_3  + 2 n_4 &=  \rho. \label{eq:n3}
\end{align}
Together with the constraint  $n_1 + n_2 + n_3 + n_4 = \tau$, these imply Donald's formula for the number of leaves, $n_1 = 2 \tau - 3 \rho$ (equation~\ref{eq:leaves}).
The following lemma gives an additional constraint on the $n_i$'s.

\begin{lemma}
For a tredoku tiling
\begin{equation}
n_1 + n_2 \ge n_3.
\label{lem:n123}
\end{equation}
\end{lemma}

\begin{proof}
If we move around the boundary of the tiling, we define a sequence of tiles whose edges lie on the boundary. These tiles will be 1--, 2-- and 3--tiles, all of which have
at least one edge on the boundary of the tiling. We claim that the sequence of tiles cannot contain two consecutive 3--tiles. A 3--tile has one edge on the perimeter of
the tiling and lies at the end of a run in the direction of that edge. It therefore lies at the centre of a run in the direction of the two edges that are not parallel to
the boundary edge. Therefore, a tile attached to either of these edges cannot also be in the centre of a run in this direction and therefore cannot be a 3--tile. Therefore,
the sequence of tiles must contain a 1--tile or a 2--tile between each pair of 3--tiles, and hence $n_1 + n_2 \ge n_3$.
\end{proof}

Our existence result for irreducible tilings depends upon the following Lemma.
\begin{lemma}
For an irreducible tredoku tiling, $n_4=0$ or $n_4=1$.
\label{lem:n4}
\end{lemma}

\begin{proof}
Suppose $n_4 > 1$ and choose two 4--tiles, $A$ and $B$. Since the tiling is irreducible, the spine of the tiling is connected and there must be a path between the two tiles.
Clearly, $A$ and $B$ cannot share an edge, because they cannot then both be at the centre of a run in that direction, so consider the penultimate tile, $C$ in the path from
$A$ to $B$. Suppose it is attached by edge $e$ to the preceding tile in the path. Tile $C$ must be at the end of a run in the direction of $e$ and therefore $B$ cannot be
attached to the edge parallel to $e$. But $C$ is also at the centre of a run in the direction of the two edges of $C$ that are not parallel to $e$, so $B$ cannot be attached
to either of these edges, because it is also be at the centre of a run in this direction. Thus a connected spine cannot contain more than one 4--tile.
\end{proof}

We can now give an existence result for irreducible tredoku tilings.
\begin{theorem}
Suppose $\mathcal{T}$ is a tredoku tiling with $\tau$ tiles and $\rho$ runs. Irreducible tilings exist for $\tau=5, \rho=2$ and for $\tau=7, \rho=3$. For any other
irreducible tiling
\begin{equation}
                   2 \rho - 3 \le \tau \le 2 \rho.
\end{equation}
Conversely, for $k = 0, 1, 2, 3$, irreducible tilings exists with $\tau=2 \rho- k$ for $\rho \ge 3+k$.
\label{thm:thm2}
\end{theorem}
\begin{proof}
We first show that there are no irreducible verdant tilings with $\rho \ge 4$. Recall that a verdant tiling has $\tau = 2 \rho + 1$ and $n_1 = \rho + 2$.
From Lemma~\ref{lem:n4}, an irreducible tiling has $n_4 = 0$ or $n_4=1$. However, for a verdant tiling, we cannot set
$n_4=0$, because this leads to $n_2 = -1$, so an irreducible verdant tiling has $n_4=1$, implying $n_2=0$ and $n_3=\rho-2$.

The verdant tilings 5.2 and 7.3 do indeed have $n_4=1$ and $n_2=0$, but we show that this is not possible for $\rho \ge 4$. The spine of such a verdant tiling would contain
at least two 3--tiles as well as the single 4--tile. Up to congruence, there are only two possible ways of connecting a 4--tile to a 3--tile, shown in the left two panels of
Figure~\ref{fig:proof2}. However, the runs of the second of these cannot be completed without adding the highlighted  2--tile. For the same reason, the only possible way of
connecting a 4--tile and two 3--tiles is shown in the third panel of Figure~\ref{fig:proof2}. But again, the runs of this tiling cannot be completed without the use of at
least one 2--tile, as shown in the right three panels of Figure~\ref{fig:proof2}. So it is impossible to construct an irreducible verdant tiling with $\rho \ge 4$ runs
without the use of 2--tiles, and therefore no such tiling exists.

\begin{figure}[h!]
   \centering
   \makebox[\textwidth][c]{\includegraphics[trim=0in 6.6in 0in 0.01in, clip, width=0.99\textwidth]{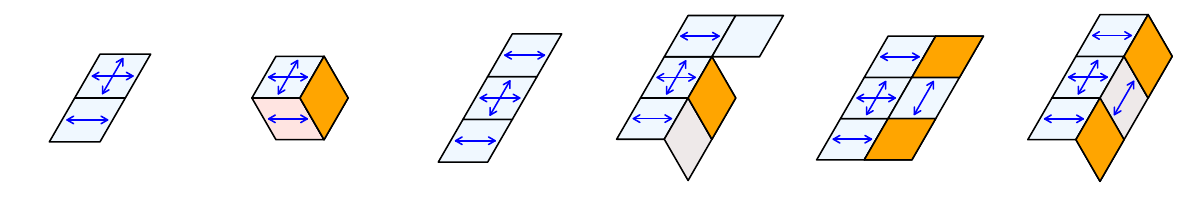}}
   \caption{Arrangements of 3--tiles and 4--tiles. Arrows indicate directions of runs. Tiles that must be 2--tiles are highlighted in orange.}
   \label{fig:proof2}
\end{figure}

Non-verdant tilings have $\tau \le 2 \rho$. From equation~(\ref{eq:n3}), $n_3 = \rho - 2 n_4$ and substituting this and $n_1 = 2 \tau - 3 \rho$ into equation~(\ref{eq:n2}),
we can solve for $n_2$. This leads to $n_1+n_2 = \tau - \rho + n_4$. From Lemma~\ref{lem:n123} we have
\[
     \tau - \rho + n_4 \ge \rho - 2 n_4.
\]
It follows that if $n_4=0$, $\tau=2 \rho$, and if $n_4=1$,
\[
    2 \rho - 3 \le \tau \le 2 \rho.
\]

The second part of the theorem follows directly from Donald's work. It is easily checked that the tilings generated by his constructions 16-23 are all irreducible.
These, together with the particular tilings dap6.3a, dap7.4a, dap8.5a and dap9.6 cover all the cases specified in the Theorem.
\end{proof}

We can now prove several key parts of Theorem~\ref{thm:thm1}.

\begin{corollary}
There is no tredoku tiling with $\tau=12$ and $\rho=8$.
\end{corollary}

\begin{proof}
From Theorem~\ref{thm:thm2}, any such tiling would be reducible. A merging of a $\tau_1.\rho_1$ tiling and a $\tau_2.\rho_2$ tiling has $\rho_1+\rho_2$ runs and either
$\tau_1+\tau_2-1$ tiles (single merging) or $\tau_1+\tau_2-2$ tiles (double merging). Since $\tau \ge 5$, the largest tiling that could be involved in a merging has
$\tau=9$. It is easily checked that amongst the possible mergings of tilings with $\tau \le 9$ that yield a tiling with $\tau=12$, including mergings of three 5.2
tilings or two 5.2 tilings and a 6.3 tiling, most have fewer than 8 runs. Some other combinations, such as a double merging of a 6.3 tiling and an 8.5 tiling are not
possible because the 8.5 tiling has only a single leaf. The only possibility that remains is a double merging of two 7.4 tilings. However, it is easily seen by inspection
that it is not possible to double merge tilings from any of the four isomorphism classes of 7.4 tilings, because of the positioning of their leaves. Therefore, no 12.8
tredoku tiling exists.
\end{proof}

\begin{corollary}
The only verdant tredoku tilings that exist are 5.2, 7.3, 9.4, 11.5, 13.6 and 17.8.
\end{corollary}

\begin{proof}
From Theorem~\ref{thm:thm2}, any verdant tiling other than 5.2 and 7.3 must result from merging two smaller tilings. Moreover, because of the constraint
$\tau \le 2 \rho+1$, this must be a single merging of two tredoku tilings that are themselves both verdant. Thus, all verdant tilings with $\rho \ge 4$ results
from merging 5.2 and/or 7.3 tilings. It is then easily checked that the only possible mergings are 5.2$-$5.2, 5.2$-$7.3, 7.3$-$7.3 and 7.3$-$5.2$-$7.3, which
yield the tilings stated in the corollary.
\end{proof}

Blackburn (2024) also shows that a verdant tiling with $\rho \ge 4$ must be the result of merging two smaller tilings, but by a completely different argument.

\section{Enumerating tredoku tilings}\label{sect:enum}

This section describes some computer enumerations of isomorphism classes of tredoku tilings. We use different approaches to enumerate very small tilings
($\tau \le 11$) and slightly larger tilings \mbox{($12 \le \tau \le 16$)}.

\subsection{Enumerating small tredoku tilings $(\tau \le 11)$}

The simplest way to enumerate tredoku tilings with $\tau$ tiles is to generate all possible arrangements of $\tau$ tiles and count how many of these satisfy the
tredoku properties \mbox{P1--P5}. Although this exhaustive approach is feasible only for small values of $\tau$, it is a useful starting point for more sophisticated approaches.

We build tilings sequentially, at each step adding a tile to the outer boundary of an existing tiling. A tiling constructed in this way automatically satisfies property P1.
We also require the tilings that we construct to have no overlapping tiles (property P2) and not to have any runs of length greater than 3 (property P5), otherwise the
tiling cannot be a tredoku tiling and nor can it be extended to a tredoku tiling be adding further tiles. We call a tiling of $\tau$ tiles that satisfies properties P1,
P2 and P5 a \define{fragment} of size $\tau$.

Figure~\ref{fig:count1} shows two incongruent fragments of size 2 and nine incongruent fragments of size 3. Every fragment of size 2 or 3 is congruent to one of these.
\begin{figure}[h!]
   \centering
   \makebox[\textwidth][c]{\includegraphics[trim=0in 5.3in 0in 0.0in, clip, width=0.7\textwidth]{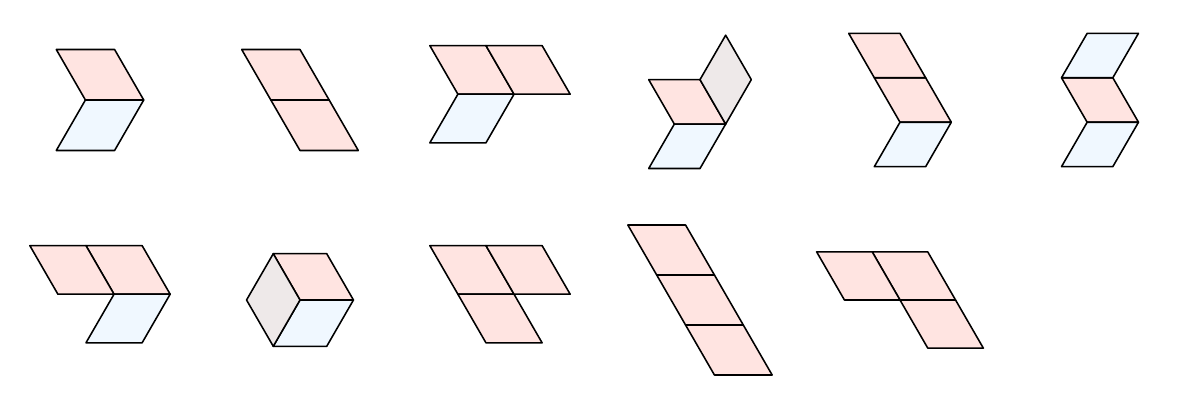}}
   \caption{The two incongruent fragments of size 2 (top row, left) and the nine incongruent fragments of size 3.}
   \label{fig:count1}
   % See enumerate1.R for source code
\end{figure}

Although we are ultimately interested in isomorphism classes of tilings, we need to retain all distinct incongruent to be able to generate all larger fragments.
For example, consider adding an extra tile to each of the 9 fragments of size 3 shown in Figure~\ref{fig:count1}. This generates 35 incongruent fragments of
size 4, which fall into 6 different isomorphism classes. The top row of Figure~\ref{fig:count2} shows a representative of each class. However, there are tredoku
tilings that cannot be generated from any of these six fragments by adding further tiles. An example is the tiling dap7.4d shown at the bottom right of
Figure~\ref{fig:count2}. This tiling can be generated from exactly three of the 35 incongruent fragments of size 4. These three fragments, which belong to three
different isomorphism classes, are shown in the second row of Figure~\ref{fig:count2}.

\begin{figure}[h!]
   \centering
   \makebox[\textwidth][c]{\includegraphics[trim=0in 5.2in 0in 0.0in, clip, width=0.8\textwidth]{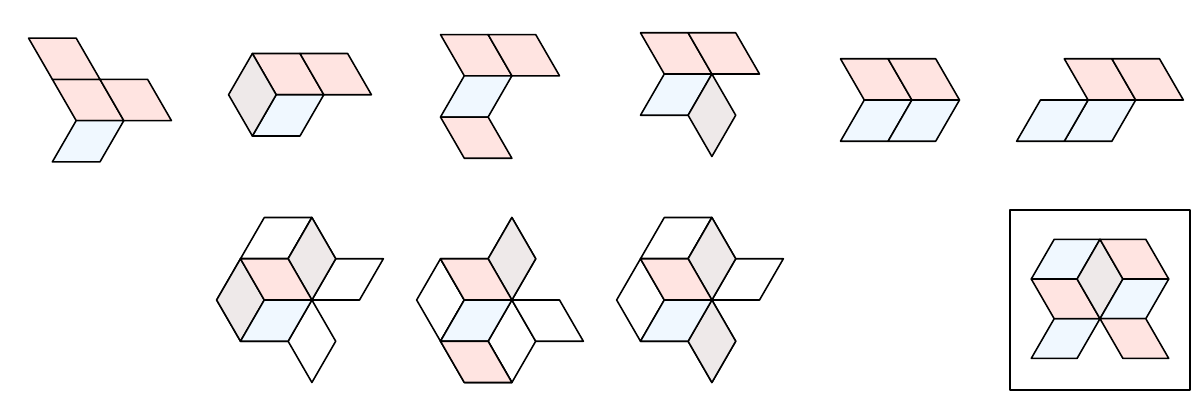}}
   \caption{The upper row shows a representative of each of the 6 isomorphism classes of fragments of size 4. The lower row shows the three fragments of size 4
            that can be extended to create the tredoku tiling dap7.4d shown in the box at the bottom right; they are each isomorphic to the fragment immediately above.}
   \label{fig:count2}
\end{figure}

We can attempt to generate a fragment of size $\tau+1$ from a fragment of size $\tau$ by adding a single tile to an edge of its outer boundary. There will always be
two distinct types of tile that can be added. For each type, we need to check whether the resulting tiling is a valid fragment.

Repeating this process systematically for all edges on the outer boundary of all fragments of size $\tau$ generates all possible fragments of size $\tau+1$. We then
reduce this set of fragments to a single representative of each congruence class. For example, the nine fragments of size 3 in Figure~\ref{fig:count2} were generated
in this way from the two fragments of size 2.

Table~\ref{tab:count1} shows the resulting number of incongruent fragments of size $\tau$ for $1 \le \tau \le 11$. For most of this range, the number of fragments of
size $\tau+1$ is very approximately five times the number of size $\tau$.

\begin{table}[h!]
\begin{center}
\caption{Results of sequential generation of incongruent fragments. Counts of tredoku tilings are the number of isomorphism classes.}
\label{tab:count1}
\begin{tabular}{rrrrr}
\hline
  & \multicolumn{1}{c}{Number of}  & \multicolumn{1}{c}{Number with}  & \multicolumn{2}{c}{Number of tredoku tilings} \\
$\tau$  & \multicolumn{1}{c}{fragments}  & \multicolumn{1}{c}{holes}  & without holes & with holes\\
\hline
 1 &  1      \ \ \ &      0 $\quad$ \ &    0 $\qquad$ \ &    0 $\qquad$ \ \\
 2 &  2      \ \ \ &      0 $\quad$ \ &    0 $\qquad$ \ &    0 $\qquad$ \ \\
 3 &  9      \ \ \ &      0 $\quad$ \ &    0 $\qquad$ \ &    0 $\qquad$ \ \\
 4 &  35     \ \ \ &      0 $\quad$ \ &    0 $\qquad$ \ &    0 $\qquad$ \ \\
 5 &  179    \ \ \ &      2 $\quad$ \ &    1 $\qquad$ \ &    0 $\qquad$ \ \\
 6 &  863    \ \ \ &     37 $\quad$ \ &    2 $\qquad$ \ &    1 $\qquad$ \ \\
 7 &  4291   \ \ \ &    368 $\quad$ \ &    5 $\qquad$ \ &    8 $\qquad$ \ \\
 8 &  21441  \ \ \ &   2869 $\quad$ \ &   10 $\qquad$ \ &    4 $\qquad$ \ \\
 9 &  109195 \ \ \ &  19839 $\quad$ \ &   16 $\qquad$ \ &    39 $\qquad$ \ \\
10 &  561972 \ \ \ & 128252 $\quad$ \ &   37 $\qquad$ \ &    62 $\qquad$ \ \\
11 & 2917524 \ \ \ & 797082 $\quad$ \ &   98 $\qquad$ \ &    666 $\qquad$ \ \\
\hline
\end{tabular}
\end{center}
\end{table}

Of course, this exhaustive generation process is extremely inefficient. Typically, each congruence class is generated many times and this incurs not only the direct
cost of repeated generation but also the subsequent cost of determining that the fragment is congruent to a fragment that has already been found and should therefore
be excluded. For enumerating polyominoes, Redelmeier (1981) avoided this problem by defining a canonical representation of polyominoes and developing an efficient
algorithm to generate each canonical polyomino exactly once, but I have been unable to find a similar canonical representation for tredoku fragments.

For $\tau \ge 5$, an increasing proportion of fragments generated in this way contains holes as shown in Table~\ref{tab:count1}. For a fragment of size $\tau$ that
contains holes, it may be possible to generate a fragment of size $\tau+1$ by adding a tile to fill or partially fill a hole. However, it is always possible to generate
such a fragment by adding a tile to the outer boundary of a (different) fragment of size $\tau$. So to generate all possible fragments, we need only consider additions
to the outer boundary.

Table~\ref{tab:count1} also shows the number of non-isomorphic fragments that are tredoku tilings, with or without holes. We provide more detailed results in
Section~\ref{sect:taule16}, where we describe how the enumeration may be extended beyond $\tau=11$, using ideas from Donald's work.

\subsection{Evaluating larger tredoku tilings ($\tau \le 16)$}\label{sect:taule16}

We can extend the enumeration of tredoku tilings somewhat by generating reducible and irreducible tilings separately.

\subsubsection{Generating irreducible tilings}

The basic approach is illustrated in Figure~\ref{fig:genirr1}. We start with a fragment of size $\sigma$ that has no holes and has at most one 4--tile, and may therefore
be the spine of one or more irreducible tilings. We then generate all tredoku tilings which have this fragment as their spine, by adding tiles to the boundary of the
fragment in all possible configurations. Although we are ultimately interested only in non-isomorphic tilings, at this stage we retain all incongruent tilings to use
in constructing reducible tilings. In the example, this generates the four tredoku tilings shown in Figure~\ref{fig:genirr1}, two incongruent versions of dap8.5a,
dap9.5b and a 10.5 tiling that Donald did not discover.\\[1ex]
\begin{figure}[htb!]
   \centering
   \makebox[\textwidth][c]{\includegraphics[trim=0in 6.5in 0in 0.01in, clip, width=1\textwidth]{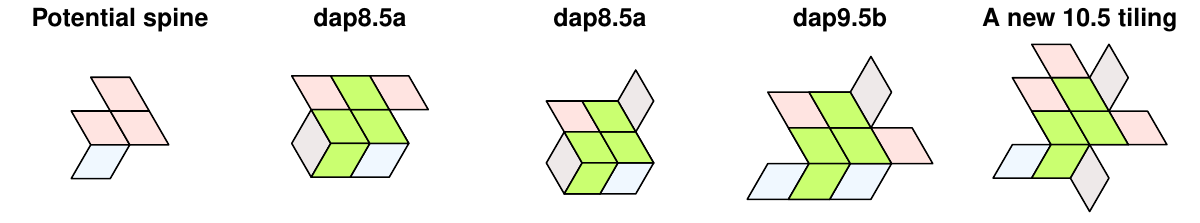}}
   \caption{The four tredoku tilings whose spine is the fragment shown in the left panel.}
   \label{fig:genirr1}
\end{figure}

The following Lemma enables us to determine the range of spine sizes that we need to consider to generate all tredoku tilings up to a given size.
\begin{lemma}
Suppose $\mathcal{T}$ is an irreducible tredoku tiling with $\tau$ tiles and $\rho$ runs that has a spine of size $\sigma \ge 5$ tiles. Then either $\rho=\sigma$ and
$\tau=2\sigma$ or $\rho=\sigma+1$ and $2\sigma -1 \le \tau \le 2\sigma + 2 $.
\end{lemma}

\begin{proof}
Since the tiling is irreducible, it has either $n_4=0$ or $n_4=1$. If $n_4=0$  then $n_3=\sigma$ and hence $\rho=n_3+2 n_4=\sigma$. Also, from equation~(\ref{eq:n2})
and Lemma~\ref{lem:n123},
\[
   n_2 = n_1+2 n_2 - (n_1+n_2) \le n_1 + 2 n_2 - n_3 = 0.
\]
Hence $n_2=0$ and $n_1=\sigma$, giving $\tau=2\sigma$. So if the spine has no 4--tiles, the only tiling that can be generated has $\tau=2 \sigma$ and $\rho=\sigma$.

Similarly, if $n_4=1$, we find that $n_3=\sigma-1$ and hence $\rho=\sigma+1$, $n_1=\sigma+3-2n_2$, $\tau=2\sigma+3-n_2$ and $n_2\le 4$. We also need $n_2 > 0$,
otherwise the tiling would have $2 \rho+1$ tiles, with $\rho \ge 6$, in contradiction of Theorem~\ref{thm:thm2}. Hence, $2\sigma -1 \le \tau \le 2\sigma + 2 $.
\end{proof}

In particular, it follows from this Lemma that the smallest irreducible tredoku tiling that can be generated from a spine of size 9 has 17 tiles. Therefore, to enumerate
irreducible tredoku tilings of up to 16 tiles, it is sufficient to consider spines consisting of 8 tiles or fewer.

\subsubsection{Generating reducible tilings}

We generate reducible tredoku tilings consisting of $\tau$ tiles by considering systematically all possible single and double mergings of smaller tilings.
We elaborate on the procedure for single merging;  the procedure for double merging is similar.

We first identify the range of possible tiling sizes that can be merged. If $\mathcal{T}_{1}$ and $\mathcal{T}_{2}$ are tredoku tilings consisting of $\tau_1$ and $\tau_2$
tiles respectively then single merging these tilings gives  a tiling with $\tau_1+\tau_2-1$ tiles. We may assume $\tau_1 \le \tau_2$ without loss of generality. We therefore
need to consider all pairs $(\tau_1, \tau_2)$ where, $5 \le \tau_1 \le \tau_2$ and $\tau_1+\tau_2-1 = \tau$. For each such pair we need to consider in turn all possible
tredoku tilings $\mathcal{T}_{1}$ and $\mathcal{T}_{2}$, and attempt to combine them by single mergeing.

To describe the procedure in detail, we need to introduce the idea of a \define{leaf-variant} tiling that is implicit in some of the discussion above. Whenever a tile is
attached to the edge of another tile, there is always an alternative tile type that could potentially be attached. For example, either a B or an F tile can be attached to
the lower edge of an F tile, whilst either a D or an F tile can be attached to its right hand edge. A leaf-variant tiling is created by replacing some of the leaf tiles by
their alternative type, in such a way that the variant remains a valid tredoku tiling. Therefore, if there are $\lambda$ leaf tiles, there are at most $2^\lambda$
leaf-variant tilings.

Figure~\ref{fig:genred1} provides two examples. The top row shows the tilings dap6.3a and dap8.4e. They have 3 and 4 leaves respectively. The second row shows the four
leaf-variant tilings of dap6.3a. Other modifications of the leaf tiles do not lead to valid tilings. All four variants are isomorphic.

\begin{figure}[h!]
   \centering
   \makebox[\textwidth][c]{\includegraphics[trim=0in 2.9in 0in 0.01in, clip, width=0.92\textwidth]{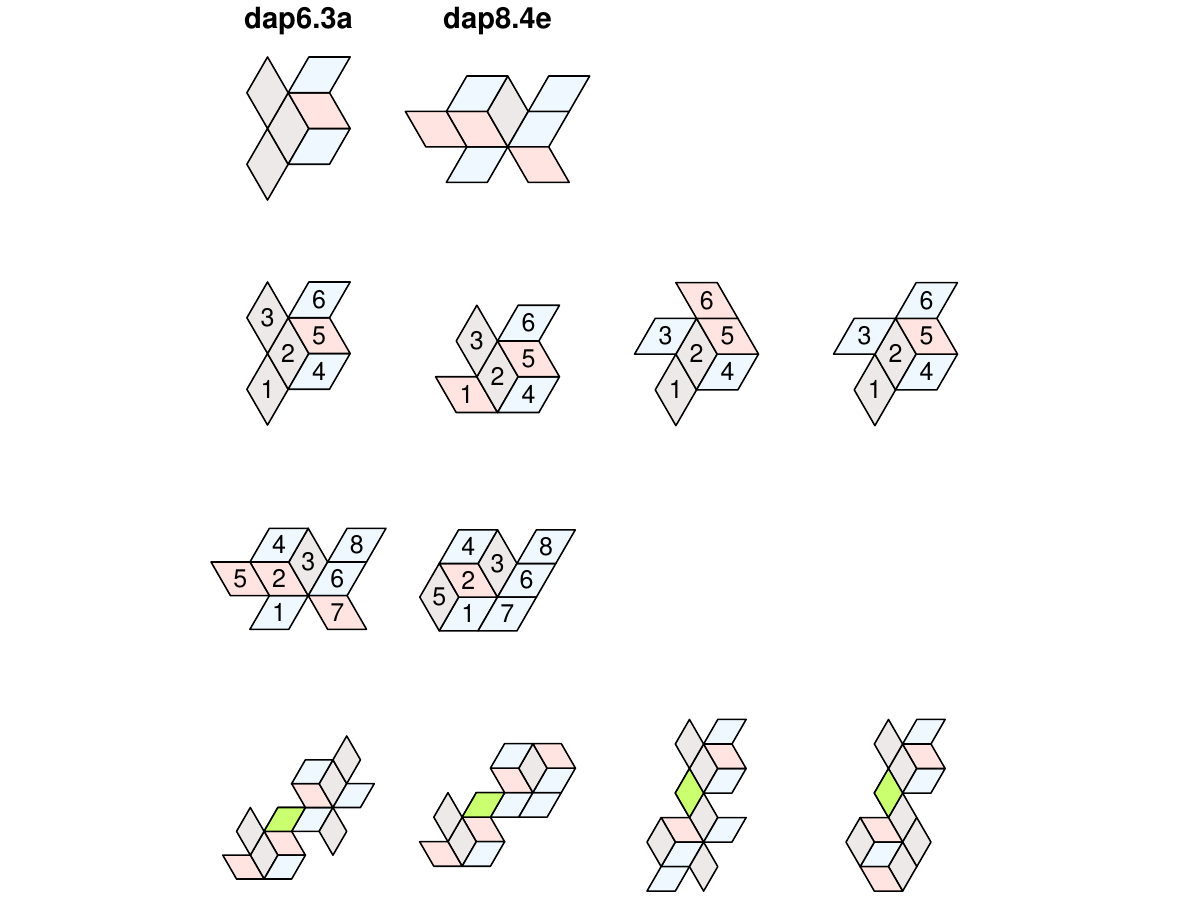}}
   \caption{The top row shows the tilings dap6.3a and dap8.4e. The next two rows show the leaf-variants of these tilings. The bottom row shows the four non-isomorphic
            13.7 tredoku tilings generated by
            single merging dap6.3a and dap8.4e.}
   \label{fig:genred1}
\end{figure}

Although it has more leaves, dap8.4e has only two leaf-variants. These are shown in the third row of Figure~\ref{fig:genred1} and comprise the original tiling along with
the tiling dap8.5c.

The bottom row of Figure~\ref{fig:genred1} shows the four non-isomorphic 13.7 tredoku tilings that can result from a single merging of dap6.3a and dap8.4e.
More specifically, for each leaf-variant of dap6.3a and each leaf-variant of dap8.4e, we  check whether it is possible to overlap a tile of the leaf-variant of
dap6.3a with a tile of the leaf-variant of dap8.4e to generate a valid 13.7 tredoku tiling. To fully explore all possibilities, it is necessary to consider
all twelve isometries of one of the leaf-variants.

The whole procedure yields twelve tredoku tilings. However, these belong to just four isomorphism classes and the bottom row of Figure~\ref{fig:genred1} shows one
tiling from each class. None of these tilings is isomorphic to any of the four 13.7 tilings that Donald found.

\subsection{Results of enumeration}

The complete enumeration of tredoku tilings with up to 16 tiles is shown in Table~\ref{tab:count2}. The smallest reducible tiling is dap8.4a (see Figure~\ref{fig:rungraph2}
later) and irreducible tilings predominate until $\tau=10$. For $\tau=11$ there are equal numbers of reducible and irreducible tilings, and thereafter irreducible tilings
predominate. The number of reducible tilings increases monotonically with $\tau$, but the number of irreducible tilings does not.

\begin{table}[h!]
\begin{center}
\caption{The number of isomorphism classes of irreducible (I) and reducible (R) tredoku tilings containing 16 or fewer tiles. An entry such as 4/6 indicates that there are 6 isomorphism
         classes, of which Donald found 4.}
\label{tab:count2}
\begin{tabular}{ccrrrrrrrrrr}
\hline
Number of        & & \multicolumn{9}{c}{Number of runs, $\rho$} &  \\
tiles, $\tau$ & Type & 2 & 3 & 4 & 5 & 6 & 7 & 8 & 9 & 10 & Total \\
\hline
 5 & I &  1/1 & \add{0} & & & & &\phantom{314} &\phantom{314} & \phantom{314} & 1 \\
 5 & R &  0 & \add{0} & & & & & & &  & 0 \\ [1ex]
 6 & I &     & 2/2 & \add{0} & & & & & &  & 2 \\
 6 & R &     & 0 & \add{0} & & & & & &  & 0 \\ [1ex]
 7 & I &     & 1/1 & 4/4 & & & & & &  & 5 \\
 7 & R &     & 0 & 0 & & & & & &  & 0 \\ [1ex]
 8 & I &     &  & 6/6 & 3/3 & & & & &  & 9 \\
 8 & R &     &  & 1/1 & 0 & & & & &  & 1 \\ [1ex]
 9 & I &     &  & 0 & 9/12 & 1/1 & & & &  & 13 \\
 9 & R &     &  & 1/1 & 2/2 & 0 & & & &  & 3 \\ [1ex]
10 & I &     &  & & 3/8 & 13/15 & & & &  & 23 \\
10 & R &     &  & & 4/6 & 8/8 & & & &  & 14 \\ [1ex]
11 & I &     &  & & 0 & 1/22 & 1/1 & & &  & 23 \\
11 & R &     &  & & 1/1 & 5/22 & 0 & & &  & 23 \\ [1ex]
12 & I &     &  & & & 14 & 30 & \add{0} & &  & 44 \\
12 & R &     &  & & & 23 & 31 & \add{0} & &  & 54 \\ [1ex]
13 & I &     &  & & & 0 & 20 & 8 & &  & 28 \\
13 & R &     &  & & & 1 & 115 & 31 & &  & 147 \\[1ex]
14 & I &     &  & & &   & 3 & 29 & 0 &  & 32 \\
14 & R &     &  & & &  & 52 & 263 & 18 &  & 333 \\ [1ex]
15 & I &     &  & & &  &  \add{0} & 17 & 4 & 0    & 21 \\
15 & R &     &  & & &  &  \add{0} & 445 & 355 & 6 & 806 \\ [1ex]
16 & I &    &  & & &  &  &  4 & 35  & 0 & 39 \\
16 & R &   \phantom{314}  & \phantom{314} & \phantom{314}&\phantom{314} &\phantom{314}  &\phantom{314}  & 84 & 1345 & 314 & 1743 \\ [0.8ex]
\hline
\end{tabular}
\end{center}
\end{table}
\ \\[-4ex]
Although it is not clear to what extent Donald attempted to find all small non-isomorphic tilings, he succeeded in identifying 59/71 isomorphism classes with $\tau \le 10$,
a remarkable achievement for someone armed only with a pencil and a pad of isometric graph paper. For $\tau > 10$, he seems to have been content to find a few representative
examples of $\tau.\rho$ tilings. For example, he identified the unique isomorphism classes for 11.5 and 11.7 tilings, but listed only six of the forty four 11.6 tilings.

\section{Run graphs}\label{sect:rungraphs}

In this section we introduce a new type of graph associated with a tredoku tiling or fragment, the \emph{run graph}. Later in this section, we use run graphs to give another
proof of the non-existence of 5.3, 6.4 and 12.8 tredoku tilings.

The run graph of a tiling is the graph in which the vertices represent the runs of the tiling and there is an edge joining any two \emph{distinct} vertices for which the
corresponding runs have a tile in common. Distinctness ensures that the run graph is a simple graph.

Because a tredoku tiling is connected, every run has a tile in common with at least one other run, so its run graph is connected. Equally, each tile of a run can lie in at most
one other run, so the maximum degree of any vertex is 3. Since two distinct runs can have at most one tile in common, each edge of the run graph corresponds to a unique
tile in the tiling. Leaf tiles lie in only a single run and are therefore not represented in the run graph. Consequently, removing the leaves from a tredoku tiling leaves a
fragment that has the same run graph as the original tiling.

As an example, Figure~\ref{fig:rungraph1} shows the tiling dap15.8e and its run graph. The figure also shows a \emph{coloured run graph}, in which the vertices are coloured to
indicate the direction of the run that they represent. Since distinct runs that share a tile necessarily lie in different directions, this is a proper vertex colouring of the
run graph. In addition, we may colour the edges of a run graph according to the colours of the vertices that they join; the edge colours indicate the different tiles types.

\begin{figure}[h!]
   \centering
   \makebox[\textwidth][c]{\includegraphics[trim=0in 6.6in 0in 0.01in, clip, width=0.8\textwidth]{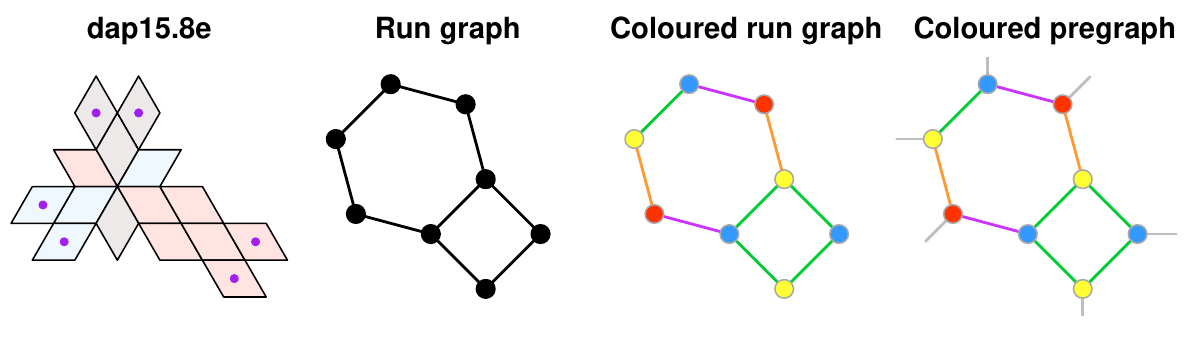}}
   \caption{The tiling dap15.8e and its associated run graph. The figure also shows a version with the vertices coloured to represent directions of runs. In the pregraph, semi-edges
            are added to represent the leaf tiles.}
   \label{fig:rungraph1}
\end{figure}

In the rungraph of dap15.8e, two vertices have degree 3 but the remainder have degree 2. We may add a semi-edge to each of the degree 2 vertices to give what is sometimes called a
\emph{pregraph} (Brinkmann, Van Cleemput \& Pisanski, 2013), as shown in the right hand panel of Figure~\ref{fig:rungraph1}. The semi-edges correspond to the leaf tiles in the tiling.
The semi-edges in Figure~\ref{fig:rungraph1} are coloured in grey because the type of a leaf tile is not always uniquely determined.

In the pregraph, all vertices have degree 3, so it is a \emph{cubic} (3-regular) pregraph. In particular, the run graph of a leafless tiling is a cubic graph. The number of
cubic pregraphs is the sequence A243393 in OEIS. However, not all cubic pregraphs are run graphs of tredoku tilings. Table~\ref{tab:npreg} shows the number of pregraphs with up to
8 vertices that are run graphs of tredoku tilings.
\begin{table}[h!]
\begin{center}
\caption{The number of cubic pregraphs with $n$ vertices (OEIS sequence A243393) and the number of these that are the run graph of at least one tredoku tiling, $R_n$.}
\label{tab:npreg}
\begin{tabular}{lrrrrrrr}
\hline
$n$  & 2 & 3 & 4 & 5 & 6 & 7 & 8 \\
\hline\\[-2ex]
A243393 & \phantom{11}1 & \phantom{11}2 &\phantom{11}6 & \phantom{1}10 & \phantom{1}29 & \phantom{1}64 & 194 \\
$R_n$ & 1 & 2 & 4 & 9 & 24 & 33 & 85 \\
\hline
\end{tabular}
\end{center}
\end{table}
\vspace*{-\baselineskip}
\ \\
The following simple lemma relates the parameters of a tredoku tiling to the parameters of its run graph and allows us to characterise the extreme tilings.
\begin{lemma}
Suppose that a run graph with $\rho$ vertices has $\epsilon$ edges. Then it represents a $\tau.\rho$ tiling, with $\tau = 3\rho-\epsilon$ and $\lambda = 3\rho-2\epsilon$ leaves.
\label{lem:rungraph1}
\end{lemma}
\begin{proof}
The sum of the degrees of the vertices is $2\epsilon$ and the number of semi-edges of the pregraph is therefore
$\lambda = 3\rho-2\epsilon$. It follows that the tiling has $\tau = 3\rho-\epsilon$ tiles.
\end{proof}
\begin{corollary}
Suppose $G$ is the run graph of a tiling $T$. Then
\begin{itemize}
\item[(a)] $T$ is a leafless tiling if and only if $G$ is a cubic graph.
\item[(b)] $T$ is a verdant tiling, with $\tau = 2 \rho + 1$, if and only if $G$ is a tree.
\end{itemize}
\end{corollary}

\subsection{Relationship to isomorphism}

The following lemma shows that, although isomorphic tilings have isomorphic run graphs, one cannot determine whether two tilings are isomorphic by inspecting their run graphs.
\begin{lemma}
Suppose $T_1$ and $T_2$ are tredoku tilings with run graphs $G_1$ and $G_2$ and coloured run graphs $H_1$ and $H_2$. Then
\begin{itemize}
\item[(a)] If $T_1$ and $T_2$ are isomorphic tilings then $G_1$ and $G_2$ are isomorphic graphs. However, $H_1$ and $H_2$ need not be isomorphic coloured graphs.
\item[(b)] If $T_1$ and $T_2$ are not isomorphic then $G_1$ and $G_2$ may nonetheless be isomorphic graphs and $H_1$ and $H_2$ may be isomorphic coloured graphs.
\end{itemize}
\end{lemma}
\begin{proof}
If $T_1$ and $T_2$ are isomorphic then by definition their tiles can be labelled in such a way that their incidence relations are identical, and hence they have the same runs
and therefore the same run graph. However, the tile types and run directions need not be the same and hence their coloured run graphs may differ. For part (b),
Figure~\ref{fig:rungraph7} shows several non-isomorphic tilings that have the same coloured run graph.
\end{proof}

Figure~\ref{fig:rungraph6} shows two examples of isomorphic tilings whose coloured run graphs are not isomorphic. The upper row shows two 9.6 tilings that are clearly
isomorphic. The fact that the two 15.9 tilings shown in the lower row are isomorphic is less obvious, but is indicated by the numbering of the tiles.

\begin{figure}[hbt!]
   \centering
   \makebox[\textwidth][c]{\includegraphics[trim=0in 4.3in 0in 0.01in, clip, width=0.9\textwidth]{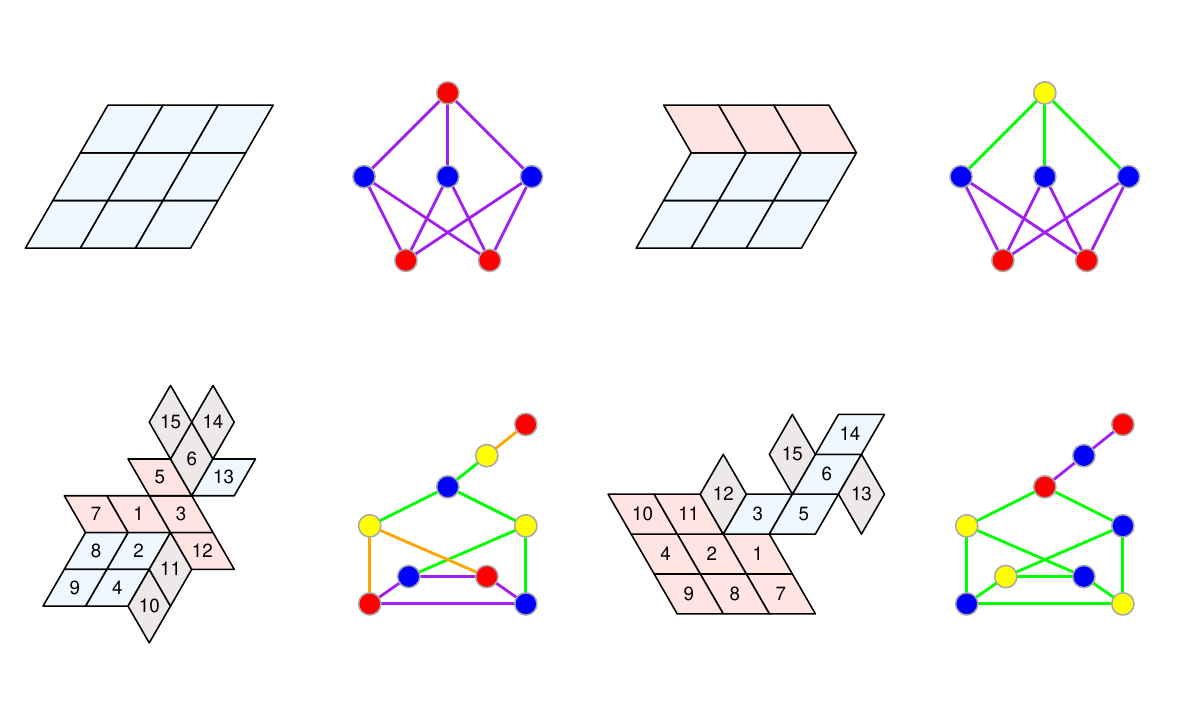}}
   \caption{Two examples of isomorphic tredoku tilings that have coloured run graphs that are not isomorphic. The top row shows two 9.6 tilings and the bottom row two 15.9 tilings.}
   \label{fig:rungraph6}
\end{figure}
In Figure~\ref{fig:rungraph7}, each tiling is obtained from the tiling on its left by moving the highlighted leaf tile to the opposite end of its run, an operation that
we term \emph{end swapping}. Clearly, end swapping does not affect either the direction of the run or its overlap with other runs and the coloured run graph of the tiling
is therefore unchanged. The resulting tiling may or may not be isomorphic to the original tiling. In his notes, Donald occasionally highlighted examples in which a new
non-isomorphic tiling could be generated by end swapping, although he doesn't seem to have made extensive use of this method to construct new tilings.

\begin{figure}[hbt!]
   \centering
   \makebox[\textwidth][c]{\includegraphics[trim=0in 7.3in 0in 0.01in, clip, width=0.99\textwidth]{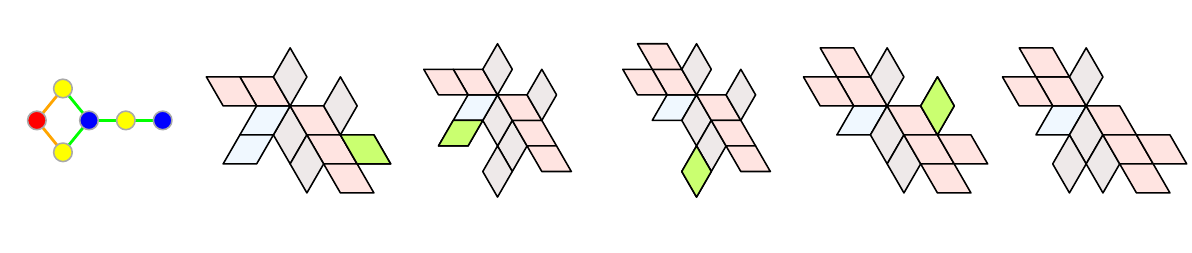}}
   \caption{Five non-isomorphic 12.6 tredoku tilings that have the same coloured run graph. Each tiling is obtained from the tiling on its left by moving the highlighted leaf tile to the
   opposite end of its run.}
   \label{fig:rungraph7}
\end{figure}

Another operation that does not affect the run graph of a tiling, but may lead to a new non-isomorphic tiling, is the flip operation discussed in Section~\ref{sect:flip}.
For example, each of the four 7.4 tilings shown in Figure~\ref{fig:fig2pt3} can be generated from any of the other tilings by a combination of flipping and/or end swapping,
and therefore all four tilings have the same run graph.

\subsection{Reducibility}

In this section we explore the connection between run graphs and reducibility.

We make use of the notion of a \emph{weak tredoku tiling} from Blackburn (2024)\footnote{Blackburn uses the term tredoku \emph{pattern} rather than tredoku tiling.};
a weak tredoku tiling is a tiling that satisfies all of the defining properties of a tredoku tiling except perhaps P3. Unlike a fragment, a weak tredoku tiling cannot
have incomplete runs of length 2.

Recall that a \define{bridge} in a connected graph is an edge whose removal separates the graph into two connected components. More generally, suppose there is a set of
$k$ edges in a connected graph, no two of which have a common endpoint, whose removal separates the graph into two connected components. We call such a
set of edges a \mbox{\emph{$k$-bridge}.} Thus, a \mbox{$k$-bridge} is an edge cut set in which none of the edges share an endpoint. This requirement is motivated by
the fact that in a double merging of two tilings, the two merged tiles do not lie in the same run in either of the component tilings (Lemma~\ref{lem:lem5b}). It implies
that for any $k$-bridge with $k>1$, each of the two separated components contains at least two vertices.

\begin{lemma}
If $T$ is a reducible tredoku tiling, resulting from a single (double) merging of two tredoku tilings $T_1$ and $T_2$, its run graph $G$ has a bridge ($2$-bridge). The
bridge corresponds to the merged leaves of $T_1$ and $T_2$, and the connected components $G_1$ and $G_2$ that result from removing the bridge from $G$ are the run graphs
of $T_1$ and $T_2$.
\label{lemm:runred1}
\end{lemma}
\begin{proof}
This follows from the fact that merging a leaf tile of $T_1$ with a leaf tile of $T_2$ corresponds to combining a semi-edge from each of their respective run graphs,
$G_1$ and $G_2$, to create an edge in the run graph, $G$, of the combined tiling.

\end{proof}
The following lemma gives a partial converse:
\begin{lemma}
If $G$ is the run graph of a tiling $T$ and has a bridge ($2$-bridge) that splits $G$ into connected components $G_1$ and $G_2$, then $T$ is a single (double) merging
of two \emph{weak} tredoku tilings, $T_1$ and $T_2$, whose run graphs are $G_1$ and $G_2$.
\label{lemm:runred2}
\end{lemma}
\begin{proof}
The edges that comprise the bridge in $G$ become semi-edges in the component graphs $G_1$ and $G_2$. Because $T$ is a tredoku tiling, the component tilings $T_1$ and
$T_2$ are each edge connected and have runs of length 3, and are therefore weak tredoku tilings.
\end{proof}
Recall that we defined $T$ to be a reducible tiling if $T_1$ and $T_2$ are themselves both tredoku tilings. It follows from Theorem 5 of Blackburn (2024) that this is
the case if and only if the tile(s) that correspond to the bridge ($2$-bridge) of $G$ are leaves of both $T_1$ and $T_2$.

Figures~\ref{fig:rungraph2} and~\ref{fig:rungraph3} show by example that a run graph may correspond to a mixture of reducible and irreducible tilings. The first row of
Figure~\ref{fig:rungraph2} shows four of the seven isomorphism classes of 8.4 tredoku tilings. These have a run graph that has a double bridge that splits the tilings
into two components, each of which is a 5.2 tiling. As shown in the next two rows, dap8.4a is a double merging of two 5.2 tredoku tilings, and is therefore reducible.
The remaining tilings are irreducible, dap8.4b and dap8.4c are double mergings of a tredoku tiling and a weak tredoku tiling and dap8.4d is a double merging of two
(congruent) weak tredoku tilings.

The last row of Figure~\ref{fig:rungraph2} shows the three remaining isomorphism classes of 8.4 tredoku tilings. These have a run graph that has a single bridge, so
that each 8.4 tiling is a single merging of a 6.3 tredoku tiling and a (necessarily) weak 3.1 tredoku tiling, the single run of three tiles that appears at the right
had end of each tiling. Hence, all of these tilings are irreducible.

Figure~\ref{fig:rungraph3} shows a slightly more complex example, a 14.8 tiling that has two possible bridges. The first of these splits the tiling into two tredoku
tilings, implying that the tiling is reducible. The second bridge splits the tiling into a tredoku tiling and a weak tredoku tiling.

\begin{figure}[h!]
   \centering
   \makebox[\textwidth][c]{\includegraphics[trim=0in 3in 0in 0.01in, clip, width=0.99\textwidth]{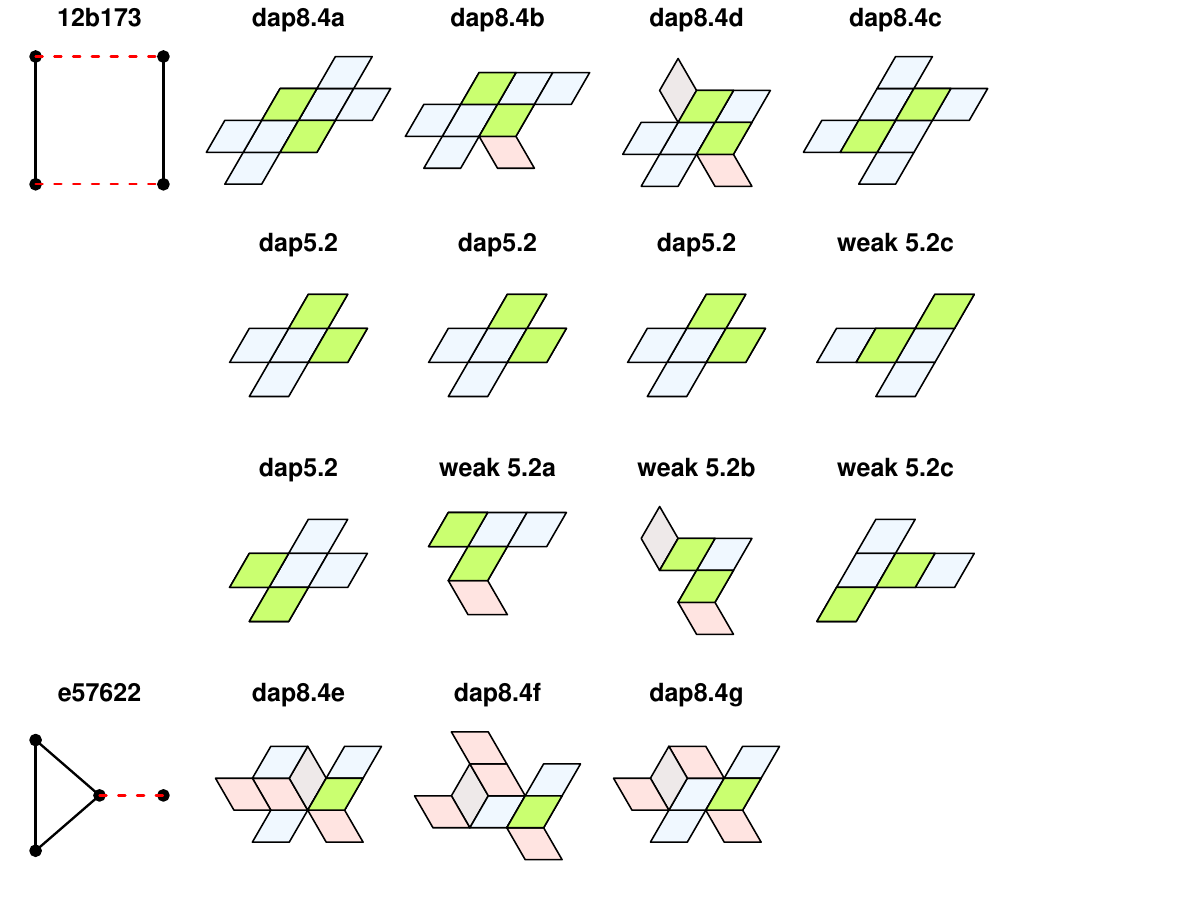}}
   \caption{Examples of single and double bridges in run graphs. Bridges are shown as dashed lines and correspond to the highlighted tiles in the tilings.}
   \label{fig:rungraph2}
\end{figure}

\begin{figure}[h!]
   \centering
   \makebox[\textwidth][c]{\includegraphics[trim=0in 4.5in 0in 0.01in, clip, width=0.99\textwidth]{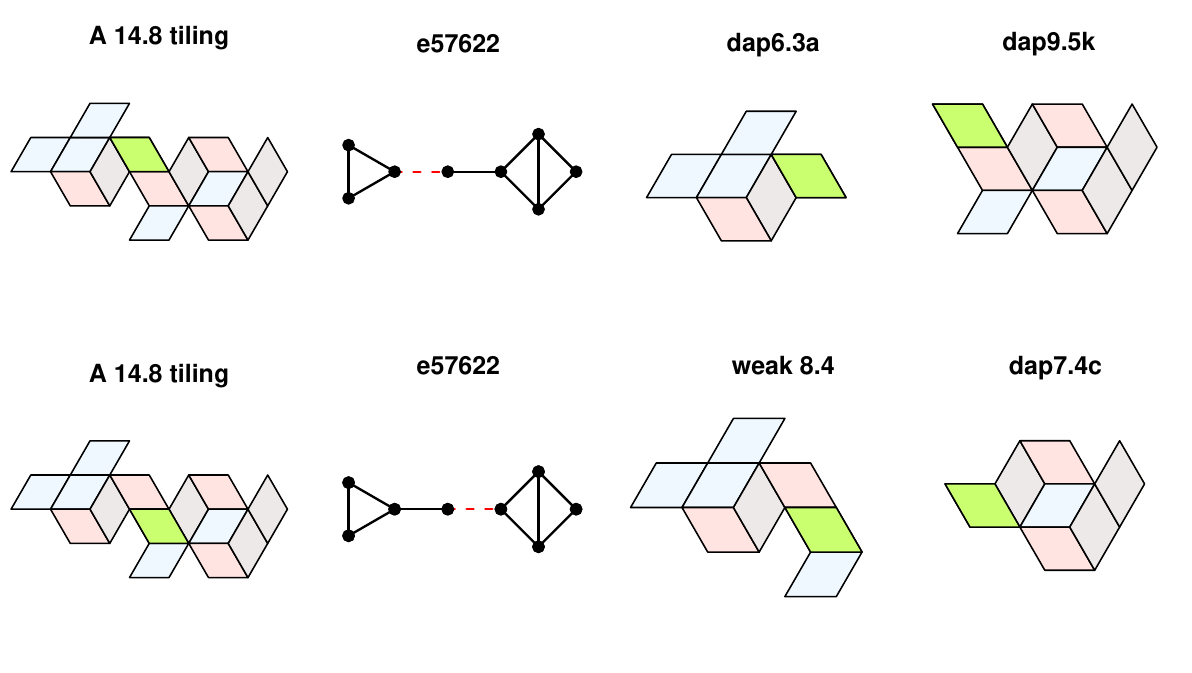}}
   \caption{Two alternative decompositions of a 14.8 tiling whose run graph has two possible bridges that correspond to the tiles that are highlighted. The tiling is reducible because one of these bridges indicates that the
   tiling is a single merging of a 6.3 tiling and a 9.5 tiling.}
   \label{fig:rungraph3}
\end{figure}

\subsection{Non-existence of certain tredoku tilings via run graphs}

One way of proving that a $\tau.\rho$ tredoku tiling does not exist for particular values of $\tau$ and $\rho$ is to show that the corresponding run graph does not exist. This leads
to relatively simple proofs in the case of 5.3, 6.4 and 12.8 tilings.

If a 5.3 tredoku tiling exists then, from Lemma~\ref{lem:rungraph1}, its run graph would have three vertices and four edges. But the maximum number of edges in a simple graph with three vertices is three, so
no 5.3 tiling exists.

Similarly, no 6.4 tredoku tiling exists because, from Lemma~\ref{lem:rungraph1}, its run graph would be the complete graph $K_4$, which has chromatic number four. But the chromatic number of a run graph cannot exceed three.

To prove that no 12.8 tiling exists, we first restate Lemma~\ref{lem:lem5} in the language of run graphs.
\begin{lemma}
If $G$ is the run graph of a tredoku tiling then $G$ does not have a $k$-bridge for $k \ge 3$.
\end{lemma}
The run graph of a 12.8 tiling would be a cubic graph with 8 vertices and 12 edges. There are only five such graphs, shown in Figure~\ref{fig:rungraph4}. As indicated in the Figure, two of
these have a 3-bridge and two have a 4-bridge. So only the first graph, which has a 2-bridge could be the run graph of a tredoku tiling. This would be a double merging of two 7.4 tredoku tilings
but, as noted in the earlier proof in Section~\ref{sect:altproof}, this is clearly not possible because of the position of the leaf tiles of 7.4 tilings (Figure~\ref{fig:fig2pt3}).

\begin{figure}[h!]
   \centering
   \makebox[\textwidth][c]{\includegraphics[trim=0in 7in 0in 0.01in, clip, width=0.99\textwidth]{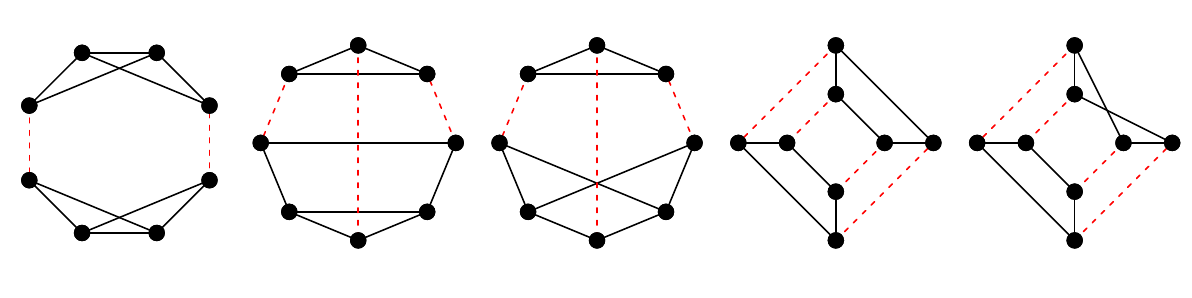}}
   \caption{The five cubic graphs with 8 vertices and 12 edges. Dashed lines are $k$-bridges.}
   \label{fig:rungraph4}
\end{figure}

\section{Quadridoku tilings}\label{sect:quadridoku}

Donald found numerous quadridoku tilings, which are shown in Appendix C. He also formulated the following conjecture about the existence of quadridoku tilings.

\begin{conjecture}
A quadridoku tiling with $\tau$ tiles and $\rho$ runs must satisfy
\begin{equation}
 \max \left ( 9, 2 \rho \right)  \le \tau \le \left \lfloor \frac{5 \rho + 3}{2} \right \rfloor.
\label{eq:quad_inequality}
\end{equation}
Within these constraints:
\begin{itemize}
\item[(i)]
For $\rho \ge 3 $, quadridoku tilings exist for all $\tau \le \left \lfloor \frac{(5 \rho + 1)}{2} \right \rfloor $, except for the combinations 9.4, 10.4, 10.5,
11.5, 18.9, and 22.11.
\item[(ii)]
The only quadridoku tilings that exist with $\tau = \left \lfloor \frac{(5 \rho + 3)} {2} \right \rfloor$ are 9.3, 11.4 and 16.6.
\end{itemize}
\label{conj:quadexist}
\end{conjecture}
The proof of the lower bound in (\ref{eq:quad_inequality}) parallels that of the corresponding result for tredoku tilings given in (\ref{eq:inequality1}). Each tile of a run, except perhaps the two end tiles, must also lie in another run, and these additional runs must be distinct. So there must be at least three runs. If a tiling has three runs then each pair of runs must share exactly one tile, giving a total of nine tiles.
A 9.3 tiling does indeed exist, as shown in Figure~\ref{fig:verdquad}, which also shows the tilings dap11.4 and dap16.6. These are all \define{verdant} tilings, that is the number of tiles
achieves the upper bound in (\ref{eq:quad_inequality}). According to part (ii) of the conjecture, these are the only verdant quadridoku tilings that exist.

To complete the proof of the lower bound in (\ref{eq:quad_inequality}), the formula for the number of leaves for quadridoku tiling is
\begin{equation}
 \lambda = 2 \tau - 4 \rho,
\label{eq:leaf_general}
\end{equation}
implying that the number of leaves must be even and that $\tau \ge 2 \rho$. Hence, \mbox{$\max \left ( 9, 2 \rho \right)  \le \tau$.}

Donald didn't have a proof of the upper bound on $\tau$ in (\ref{eq:quad_inequality}), but it follows from the inequality $\lambda \le \rho + 3$, which we establish for
all $\kappa$-doku tilings in Section~\ref{maxleaveskappa}.

\begin{figure}[h!]
   \centering
   \makebox[\textwidth][c]{\includegraphics[trim=0in 5.4in 0in 0.001in, clip, width=0.7\textwidth]{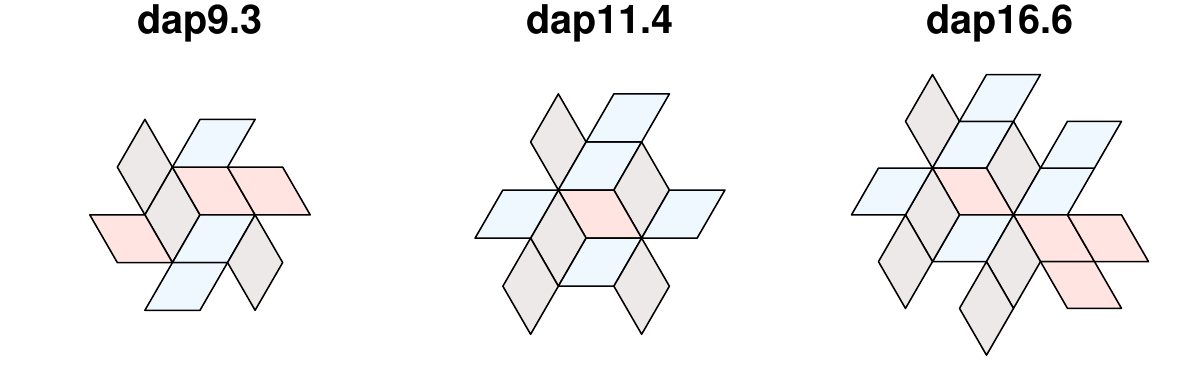}}
   \caption{According to Conjecture~\ref{conj:quadexist}, verdant quadridoku tilings exist only for $\rho=3$, $4$ and $6$.
   }
   \label{fig:verdquad}
 \end{figure}

It is possible to adapt Blackburn's proof that no 6.4 tredoku tiling exists (part of his Lemma 14) to show that there are no 9.4, 10.4, 10.5 or 11.5 quadridoku tilings.
For example, to show that no 11.5 quadridoku tiling exists, we show that any set of five runs requires at least 12 tiles. This is clearly true if three of the runs lie in the same direction, because distinct runs in the same direction have no tiles in common. The only alternative, since there are only three possible directions, is that there are two runs, $r_1$ and $r_2$ in one direction, two runs, $r_3$ and $r_4$ in a second direction, and a single run $r_5$ in the
third direction. But runs in different directions can have at most one tile in common and so $r_1$, $r_2$, $r_3$ and $r_4$ would together require at least 12 distinct tiles.

The rest of the conjecture remains open. For part (i), aside from establishing the non-existence of 18.9 and 22.11 tilings, it will be necessary to establish that
quadridoku tilings do exist for all other parameter values. Donald's collection of quadridoku tilings in Appendix C gives examples of tilings for $\tau \le 30$, but general constructions
analogous to those shown in Appendix A for tredoku tilings, and others given in Blackburn (2024), have yet to be developed for quadridoku tilings.

As an example of what is needed, the proof of the following Proposition  uses an external expansion to show that, although verdant tilings are conjectured not to exist for most values of $\rho$, tilings with one fewer tile generally do exist.

\begin{proposition}
A quadridoku tiling exists with $\tau = \left \lfloor \frac{(5 \rho + 1)}{2} \right \rfloor $ for all $\rho \ge 5$.
\label{lem:almost_verdant}
\end{proposition}

\begin{proof}
Appendix C gives several 13.5 tilings. For $\rho \ge 6$, we may use one of the external expansions shown in Figure~\ref{fig:almost_verdant}, depending on the parity of $\rho$.
\end{proof}

\begin{figure}[h!]
   \centering
   \makebox[\textwidth][c]{\includegraphics[trim=0cm 18.5cm 0cm 0.01cm, clip, width=0.999\textwidth]{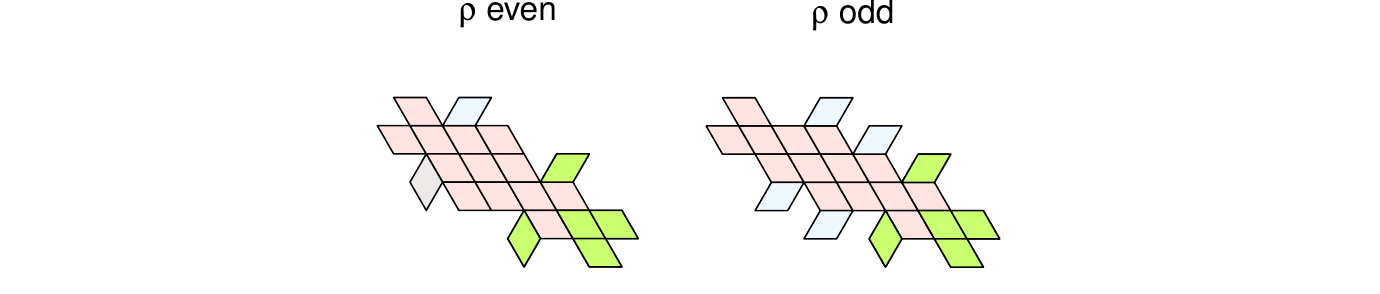}}
   \caption{Example of an external expansion. In each figure, the tiles shown in green can be added to the basic tiling $a$ times ($a \ge 0$). The left hand figure generates the sequence 15.6, 20.8, 25.10 \ldots and
   the right hand figure the sequence 18.7, 23.9, 28.11.}
   \label{fig:almost_verdant}
 \end{figure}

\subsection{Reducible quadridoku tilings} \label{sect:redquad}

Reducible quadridoku tilings can be obtained, as for tredoku tilings, by merging tilings so that either one or two tiles overlap. In Figure~\ref{fig:reduce_quad}, the reducible tiling dap24.10e
is obtained by double merging two copies of the tiling dap13.5b, and dap25.10c by single merging two copies of a tiling that is isomorphic to dap13.5b.
\begin{figure}[h!]
   \centering
   \makebox[\textwidth][c]{\includegraphics[trim=0cm 18cm 0cm 0.0cm, clip, width=0.999\textwidth]{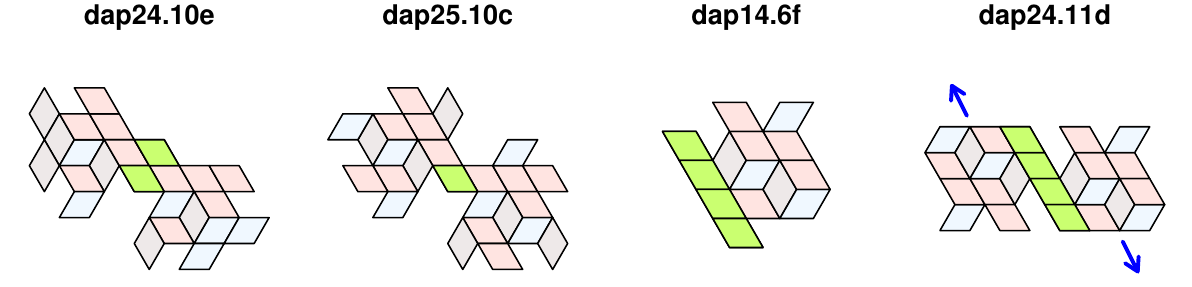}}
   \caption{Examples of reducible quadridoku tilings.
   The tiling dap24.10e is a double merging dap13.5b with itself and dap25.10c is a single merging of two copies of another 13.5 tiling that is isomorphic to dap13.5b.
   Overlapping tiles are shown in green. The tiling dap24.11d results from a new process called four merging that combines two copies of dap14.6f, as described in the text.}
   \label{fig:reduce_quad}
 \end{figure}

In addition, Donald discovered a new method of combination that we term \define{four merging}. The building block for four merging is a tiling that has  a run of four tiles of the same type, with the two end tiles being leaves and with all other tiles
of the tiling laying on one side of this run. The tiling dap14.6f in Figure~\ref{fig:reduce_quad} provides an example.

To generate a new tiling, the tiling is merged with a copy that has been rotated through 180$^\circ$, overlapping the run of four tiles of the same type, which then lies at the centre of the new
tiling. Then, as indicated in Figure~\ref{fig:reduce_quad}, the tiles that lie on either side of this run are shifted by one tile length in the direction of the arrows. This last step
ensures that the new tiling is a quadridoku tiling.

\section{$\kappa$-doku tilings} \label{sect:kdoku}

In this Section we investigate $\kappa$-doku tilings for general values of $\kappa$. We sometimes use the notation $\tau.\rho_{\kappa}$ to denote a $\kappa$-doku tiling with $\tau$ tiles and $\rho$ runs.

The inventors of tredoku also proposed a variant aimed at children, involving runs of length 2  (Mindome, 2013).\footnote{In the puzzle, each tile is replaced by a $2 \times 2$ grid
instead of the usual $3 \times3$ grid of sudoku or tredoku, and only four symbols are used rather than nine, generally resulting in a much simpler puzzle.}
A simple example of a $2$-doku tiling is a hexagon comprising one tile of each type and a more complex example with a hole is given in
the righthand panel of Figure~\ref{fig:hole2} in Section~\ref{sect:holes}. There are only limited possibilities for connecting tiles in a $2$-doku tiling, and
it may be shown that $\tau=\rho-1$ or $\tau=\rho$. This may make 2-doku tilings particularly amenable to study, but we do not pursue this here, since Donald focused his attention on quadridoku tilings ($\kappa=4$) and, to a lesser extent, quindoku tilings ($\kappa=5$).

Donald gave the following construction for a $\kappa$-doku tiling, depending on the parity of $\kappa$.
For even values of $\kappa$ of the form $\kappa=2a$ ($a \ge 1$), a regular hexagon of side $a$ filled with lozenge tiles is a $\kappa$-doku tiling. This is illustrated in the first panel of Figure~\ref{fig:kdoku} for $\kappa=6$. This tiling has $ 3\kappa^2\slash 4$  ($=27$) tiles and $3 \kappa\slash 2$ ($=9$) runs.

\begin{figure}[h!]
   \centering
   \makebox[\textwidth][c]{\includegraphics[trim=0cm 17.4cm 0cm 0.1cm, clip, width=0.95\textwidth]{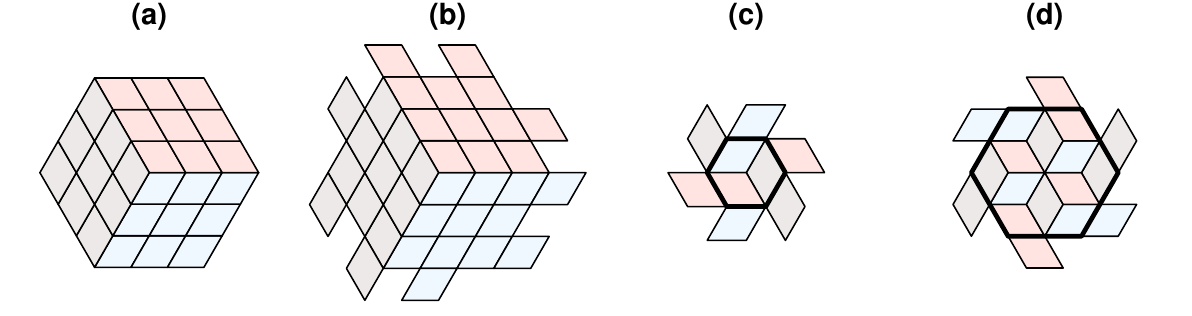}}
   \caption{Examples of $\kappa$-doku tilings: (a) a 6-doku tiling, (b) a 7-doku tiling, (c) the smallest possible $4$-doku tiling, (d) a $5$-doku tiling.}
   \label{fig:kdoku}
 \end{figure}

This hexagonal tiling also forms the basis for a $\kappa$-doku tiling for $\kappa = 2a + 1$, which is obtained by adding leaves to alternate tiles on the perimeter, as illustrated in the second
panel of Figure~\ref{fig:kdoku}. This $7$-doku tiling has  $ 3(\kappa^2-1)\slash 4$ ($= 36$) tiles and $3 (\kappa-1)\slash 2$ ($=9$) runs.

For $\kappa \ge 5$, these are the smallest possible $\kappa$-doku tilings.\footnote{We shall see in
Section~\ref{sect:mintaukappa} that for even values of $\kappa$ there is an alternative tiling with the same number of tiles but one less run.}
This is a consequence of the following Theorem, which generalises the inequalities (\ref{eq:inequality1})
and (\ref{eq:quad_inequality}) for tredoku and quadridoku tilings
respectively.

\begin{theorem}
For $\kappa \ge 5$, a  $\kappa$-doku tiling with $\tau$ tiles and $\rho$ runs must satisfy
\begin{equation}
 \max \left ( \left \lceil \frac{3 (\kappa^2-1)}{4} \right \rceil, \left \lceil \frac{\kappa \rho^{\vphantom{2}}}{2} \right \rceil \right)  \le \tau \le \left \lfloor \frac{(\kappa+1) \rho + 3}{2} \right \rfloor.
\label{eq:thm3}
\end{equation}
\label{thm:thm3}
\end{theorem}

The proof of Theorem~\ref{eq:thm3} is outlined in the following two Sections which derive the upper and lower bounds in turn.

\subsection{The maximum number of tiles of a $\kappa$-doku tiling} \label{maxleaveskappa}

In this Section, we show that the number of leaves, $\lambda$, of a $\kappa$-doku tiling satisfies $\lambda \le \rho + 3$.
Since $\lambda=2\tau - \kappa \rho$, this leads immediately to the upper bound for $\tau$ in~(\ref{eq:thm3}).

To establish this bound on $\lambda$, we need the following simple lemma.
\begin{lemma}
Suppose we remove any leaves from a $\kappa$-doku tiling that has $\tau$ tiles and $\rho$ runs. The fragment that remains has perimeter $2 \rho$.
\label{lem:perim1}
\end{lemma}
\begin{proof}
In a $\kappa$-doku tiling with $\tau$ tiles there are a total of $4\tau$ tile edges. However, within every run there are $\kappa-1$ edges that are shared by two tiles. Therefore, the number of tiling edges is $4\tau - (\kappa-1) \rho$, of which $(\kappa-1) \rho$ join tiles within runs. The boundary of the tiling therefore consists of
$  \pi =  4 \tau - 2(\kappa-1)\rho$ edges. Since the tiling has $\lambda = 2 \tau - \kappa \rho$ leaves, the tiling fragment that results from removing the leaves has perimeter $\pi - 2 \lambda = 2 \rho$, independently of $\kappa$.
\end{proof}
\begin{proposition}
A $\kappa$-doku tiling with $\rho$ runs has at most $\rho+3$ leaves.
\label{lem:lambound1}
\end{proposition}
\begin{proof}
In view of Lemma~\ref{lem:perim1}, we need to determine the maximum number of leaves, $\lambda_{\max}$, that can be attached to the boundary a fragment of perimeter $2 \rho$.
More specifically, let $\lambda_{\max}$ denote the maximum number of leaves that can be attached to the boundary in such a way that no two leaves overlap, except perhaps by sharing
a vertex that lies on the boundary of the fragment. In particular, two leaves cannot share an edge and cannot share a vertex that does not lie on the boundary of the fragment,
since this would create a hole in the tiling. Whether a $\kappa$-doku tiling exists with $\lambda_{\max}$ leaves depends on whether it is possible to place the leaves in such a
way that all $\rho$ runs have length $\kappa$.

As an example, the leftmost panel of Figure~\ref{fig:lammax2} shows the verdant tredoku tiling dap11.5 with one additional leaf. The basic tiling has seven leaves attached to the
boundary of a fragment of four tiles that has perimeter 10. However, as shown in the Figure, it is possible to attach a further tile to the boundary, although the resulting tiling
is no longer a tredoku tiling. So for this fragment, $\lambda_{\max} = 8$.

In general, as in this example, the fragment that we are considering will not be convex. However, we can obtain a modified fragment that \emph{is} convex by repeatedly adding a tile
at any reflex vertex\footnote{A reflex vertex of a polygon is a vertex for which the angle between the two edges incident at that vertex toward the interior of the polygon exceeds $180^{\circ}$.} of the fragment, until no reflex vertices remain. Figure~\ref{fig:lammax2} shows this process for the tredoku tilings dap11.5 and dap15.8e. For dap11.5, the fragment expands to a
$3 \times 2$ rhombus, whereas for dap15.8e, the fragment expands to a $3 \times 3 \times 2$ hexagon. Expanding a fragment that has one or more reflex vertices always leads either to
a rhombus or a hexagon.
\begin{figure}[h!]
   \centering
   \makebox[\textwidth][c]{\includegraphics[trim=0cm 18.05cm 0cm 0.01cm, clip, width=0.999\textwidth]{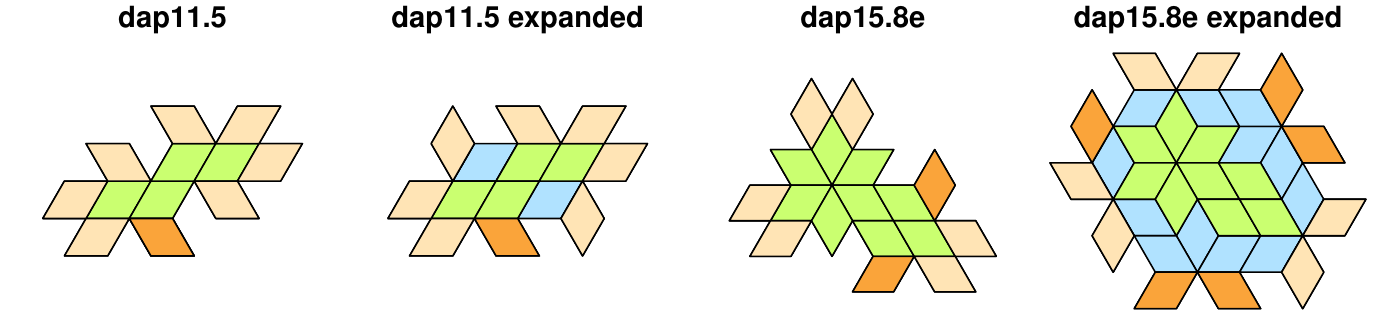}}
   \caption{The first panel shows the verdant tredoku tiling dap11.5 with one additional leaf shown in dark orange. The fragment that remains if the leaves are removed is shown in green. In the
   second panel, this fragment has been expanded to be convex, by adding the blue tiles. The third and fourth panels are similar, but for the tiling dap15.8e.}
   \label{fig:lammax2}
 \end{figure}

Expanding the fragment in this way does not change its perimeter. Moreover, we can ensure, as in Figure~\ref{fig:lammax2}, that leaves are edge-connected, possibly via tiles that have been added,  to the same tiles as for the original fragment. Therefore, the value of $\lambda_{\max}$ for the expanded fragment cannot be less than the value for the original fragment. But it can be greater, as illustrated by the
dap15.8e tiling in Figure~\ref{fig:lammax2}. It is only possible to add two further leaves to dap15.8e, but we can add five leaves to the expanded convex fragment.

So the problem reduces to determining $\lambda_{\max}$ for a convex fragment of perimeter $2 \rho$, either a rhombus or a hexagon.

The maximum number of leaves that we can place along one side of a rhombus or hexagon of length $a$ depends on the parity of $a$ and whether or not we allow the tiles at each end of the side to
protrude beyond the end. The four possibilities are shown in Figure~\ref{fig:lammax4}.
\begin{figure}[h!]
   \centering
   \makebox[\textwidth][c]{\includegraphics[trim=0cm 18.5cm 0cm 0.01cm, clip, width=0.999\textwidth]{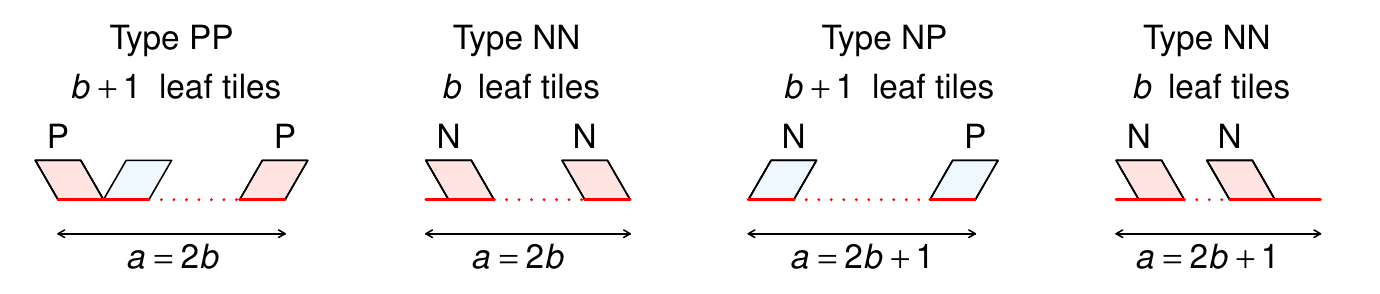}}
   \caption{The maximum number of leaf tiles that can be placed along a side of length $a$, depending on the parity of $a$ and whether tiles at each end are allowed to protrude beyond the
   end of the side (P = protruding allowed, N=protruding not allowed).}
   \label{fig:lammax4}
 \end{figure}

Sides of a hexagon that meet at a corner cannot both have a protruding leaf tile at that corner, because the leaves would share a common edge. For a rhombus, the situation is
a little more complicated. At the two corners where the interior angle is 60$^\circ$ (which, without loss of generality, we take to be the corners joining sides 1 and 2 and sides 3 and 4), both sides can have a protruding leaf tile, but at the other two corners they cannot.

Table~\ref{tab:lammax1} then shows the maximum number of leaves that can be placed around an $a \times c$ rhombus or an $a \times c \times e$ hexagon of perimeter $2 \rho$. We need to consider the parity of the side lengths and the optimal sequence of sides to use, where sides are classified as PP, PN, NP or PP as in Figure~\ref{fig:lammax4}.

\begin{table}[h!]
\begin{center}
\caption{The maximum number of leaves, $\lambda_{\max}$, that can be placed around an $a \times c$ rhombus or an $a \times c  \times e$ hexagon of perimeter $2\rho$, depending on the parity
of $a$, $c$ and $e$. The optimal
sequence refers to the arrangement of leaves along each side of the rhombus or hexagon.}
\label{tab:lammax1}
\begin{tabular}{cccccc}
\hline
\multicolumn{3}{c}{Dimensions} & & & \\
$a$ & $c$ & $e$ &  Optimal sequence & $\rho$ & $\lambda_{\max}$ \\
\hline\\[-1.5ex]
\multicolumn{3}{l}{Rhombus $a \times c$\ \ \ \ \ \ \ }  & & & \\
$2b$ & $2d$ &  &  PP--NN--PP--NN & $2(b+d) \phantom{+ 2} \,\, $& $\rho+2$ \\
$2b$ & $2d+1$ &  &  PP--PN--PP--PN & $2(b+d) + 1 $ & $\rho+3$ \\
$2b+1$ & $2d+1$ &  &  NP--NP--NP--NP & $2(b+d) + 2 $ & $\rho+2$ \\[0.75ex]
\multicolumn{3}{l}{Hexagon $a \times c \times e$\ \ \ \ } & & & \\
$2b$ & $2d$ & $2f$ &  PP--NN--PP--NN--PP--NN & $2(b+d+f) \phantom{+ 2} \,\, $ & $\rho+3$ \\
$2b$ & $2d$ & $2f+1$ & PP--NN--PN--PP--NN--PN   & $2(b+d+f)+1 $ & $\rho+3$ \\
$2b$ & $2d+1$ & $2f+1$ & PP--NP--NP--NN--PN--PN   & $2(b+d+f) +2 $ & $\rho+3$ \\
$2b+1$ & $2d+1$ & $2f+1$ & PN--PN--PN--PN--PN--PN   & $2(b+d+f) + 3 $ & $\rho+3$ \\
\hline
\end{tabular}
\end{center}
\end{table}

For example, consider the penultimate row of Table~\ref{tab:lammax1}, where four sides of the hexagon have odd length. For these sides we can select an NP (or PN) side or an NN side, the latter having one less tile. Similarly, for the two sides of even length, we have a choice between PP or NN, where again the latter has one less tile. We want to assemble the
sides to maximise the
total number of leaves, subject to the constraint that we cannot join two sides if the joined ends both have a protruding tile. This is achieved by minimizing the number of NN sides.
The solution in Table~\ref{tab:lammax1}, which has $\lambda_{\max}=\rho+3$,  has a single NN side.

From Table~\ref{tab:lammax1},
the number of leaves of any $\kappa$-doku tiling satisfies
\begin{equation}
  0 \le \lambda \le \rho+3.\\[-2ex]
\label{eq:boundlam1}
\end{equation}

\end{proof}

In fact, by comparing the parities of $\lambda = 2 \tau - \kappa \rho$ and $\rho+3$, we can give a slightly stronger result.
\begin{corollary}
For $\kappa \ge 2$, the number of leaves, $\lambda$, of a $\kappa$-doku tiling with $\rho$ runs satisfies
\[
   0 \le \lambda \le \left \lbrace \begin{array}{ll} \rho + 3 \quad  & \mathrm{if\ } \kappa \mathrm{\ is\ even\  and\ } \rho \mathrm{\ is \ odd}, \\
   \rho + 2 \quad & \mathrm{otherwise.} \end{array} \right.
\]
\end{corollary}
Moreover, we conjecture that the quadridoku tiling dap9.3 shown in Figure~\ref{fig:verdquad}, which has six leaves, is the only $\kappa$-doku tiling for which $\lambda = \rho+3$.

\subsection{The minimum number of tiles of a $\kappa$-doku tiling} \label{sect:mintaukappa}

For a $\kappa$-doku tiling with $\rho$ runs, the requirement $\lambda \ge 0$ leads to the simple bound
\begin{equation}
    \tau \ge \left \lceil \frac{\kappa \rho^{\vphantom{2}}}{2} \right \rceil.
\label{eq:taueasy}
\end{equation}
However, as we have seen for both tredoku and quadridoku tilings, this bound may not be tight for small values of $\rho$. Proposition~\ref{lem:taubound2}
below gives an alternative bound for $\tau$. And it follows from Propositions~\ref{prop:propturantriple} and~\ref{prop:turan2} that, for $\kappa \ge 5$,  any tiling that achieves
the alternative bound has
\begin{equation}
\tau \ge  \left \lceil \frac{3 (\kappa^2-1)}{4} \right \rceil
\label{eq:tauhard}
\end{equation}
Combining (\ref{eq:taueasy}) and (\ref{eq:tauhard}) gives the lower bound for $\tau$ in~(\ref{eq:thm3}).

These arguments generalise the earlier proofs of the non-existence of 5.3 and 6.4 tredoku tilings (Theorem~\ref{thm:thm1}) and 9.4, 10.4, 10.5 and 11.5 quadridoku tilings (Conjecture~\ref{conj:quadexist}).

\begin{proposition}
For $\kappa \ge 2$, the number of tiles, $\tau$, of a $\kappa$-doku tiling with $\rho$ runs satisfies
\begin{equation}
   \tau \ge \kappa \rho - \left \lfloor \frac{\rho^2}{3} \right \rfloor.
\label{eq:taubound}
\end{equation}
\label{lem:taubound2}
\end{proposition}

\begin{proof}
Consider the run graph of a $\tau.\rho_{\kappa}$
tiling. Following the argument of Lemma~\ref{lem:rungraph1}, the number of tiles is $\tau = \kappa \rho - \epsilon$, where $\epsilon$ is
the number of edges in the run graph. Thus, for a given value of $\rho$, $\tau$ is minimised by maximising $\epsilon$.

In addition, because the run graph has chromatic number at most 3, it cannot contain the complete graph $K_4$ as a subgraph. Tur\'{a}n's theorem (e.g., Aigner, 1995) then gives
\[
   \epsilon \le \left \lfloor \frac{\rho^2}{3} \right \rfloor,
\]
which in turn leads to (\ref{eq:taubound}).
\end{proof}

Suppose a $\tau.\rho_{\kappa}$ tiling achieves the bound~(\ref{eq:taubound}).
Then the number of leaves must satisfy~(\ref{eq:boundlam1}), which leads to the following necessary condition.

\begin{corollary}
If a $\tau.\rho_{\kappa}$ tiling exists with $\tau = \kappa \rho - \left \lfloor \frac{\rho^2}{3} \right \rfloor$, then
\begin{equation}
   2 \left \lfloor \frac{\rho^2}{3} \right \rfloor \le \kappa \rho \le \rho +  2 \left \lfloor \frac{\rho^2}{3} \right \rfloor + 3.
\label{eq:boundkr}
\end{equation}
\label{bound841}
\end{corollary}

We call a triple $(\tau, \rho, \kappa)$ that satisfies the conditions of Corollary~\ref{bound841} a \define{Tur\'{a}n triple}.
Straightforward algebra establishes the following explicit characterisation of Tur\'{a}n triples.
\begin{proposition}
Let $\mathcal{S} = (\tau, \rho, \kappa)$ be a Tur\'{a}n triple with $\kappa \ge 2$.
\begin{itemize}
\item[(i)] If $\kappa$ is odd then $\mathcal{T} = (5, 2, 3)$ or
\begin{equation}
  \rho = \frac{3 \kappa-3}{2}, \ \ \tau = \frac{3 \kappa^2-3}{4}, \qquad {or} \qquad \rho = \frac{3 \kappa-1}{2}, \ \ \tau = \frac{3 (\kappa^2+1)}{4}.
\label{eq:kapodd}
\end{equation}
\item[(ii)] If $\kappa$ is even then $\mathcal{T} = (9, 3, 4), \, (11, 4, 4), \,  (26, 7, 6) $ or
\begin{equation}
  \rho = \frac{3 \kappa-2}{2}, \ \ \tau = \frac{3 \kappa^2}{4}, \qquad {or} \qquad \rho = \frac{3 \kappa}{2}, \ \ \tau = \frac{3 \kappa^2}{4},
\label{eq:kapeven}
\end{equation}
\end{itemize}
\label{prop:propturantriple}
\end{proposition}
Given a Tur\'{a}n triple $\mathcal{T}=(\tau, \rho, \kappa)$,
no $\kappa$-doku tiling with $\rho$ runs exists with fewer than $\tau$ tiles.
Proposition~\ref{prop:turan2} below shows that, with one exception, $\tau.\rho_{\kappa}$ tilings exist for all Tur\'{a}n triples. The proof uses the following lemmas.
\begin{lemma}
The minimum perimeter of a lozenge tiling that consists of $\tau$ tiles is
\begin{equation}
 \pi_{\min}(\tau) = 2 \left \lceil \sqrt{3 \tau} \right \rceil. \\[-6ex]
\label{eq:harary}
\end{equation}
\label{lem:minkappa0}
\end{lemma}
\begin{proof}
Divide each lozenge tile into two equilateral triangles. The resulting tiling is called a \define{polyiamond} (e.g.\ Yang \& Meyer, 2002). Theorem 4 of Harary \& Harborth (1976) gives a formula for the
minimum perimeter of a polyiamond\footnote{Harary \& Harborth use the term \define{triangular animal} rather than polyiamond. Their formula is given in the proof of
Proposition~\ref{eq:inhole1} below.}
 consisting of $n$ tiles. Setting $n = 2 \tau$ and simplifying gives the result.
\end{proof}
The function $\pi_{\min}(\tau)$ is non-decreasing, but not strictly increasing.
Let $\tau_{\max}(\pi)$ denote
the largest value of $\tau$ for which a tiling exists with perimeter $\pi$.

\begin{lemma}
Let $\mathcal{T}$ be a $\tau.\rho_{\kappa}$ tiling and let $\mathcal{F}$ denote the fragment that results from removing any leaf tiles. $\mathcal{F}$ fills a semi-regular  $a \times b \times c$ hexagon if and
only if $(\tau, \rho, \kappa)$ is a Tur\'{a}n triple.
\label{lem:hexshape}
\end{lemma}
\begin{proof}
The proof is based on Theorem 13 of Yang \& Meyer (2002) which relates to polyiamonds. Re-expressed in terms of lozenge tilings, it indicates that a lozenge tiling with $\tau$ tiles and
perimeter $\pi$ is a semi-regular hexagon with sides $ a \times b \times c$ if and only if $\tau = \tau_{\max}(\pi)$. From Lemma~\ref{lem:perim1}, $\mathcal{F}$ has perimeter $\pi = 2\rho$ and
it may be shown from Theorem 12 of Yang \& Meyer that $\tau_{\max}(2 \rho) =  \left \lfloor \frac{\rho^2}{3} \right \rfloor$.
From the  proof of Lemma~\ref{lem:taubound2}, this is the number of edges in the run graph of $\mathcal{T}$, and hence the number of tiles in $\mathcal{F}$,  if and only if $(\tau, \rho, \kappa)$ is a Tur\'{a}n triple.
\end{proof}

The values of $a$, $b$ and $c$ may be deduced from Theorem 13 of Yang \& Meyer (2002); we give some specific values in the proof of the following proposition.

\begin{proposition}
If $\mathcal{S} = (\tau, \rho, \kappa)$ is a Tur\'{a}n triple then a $\tau.\rho_{\kappa}$ tiling exists except when $\mathcal{S}=(26,7,6)$.
\label{prop:turan2}
\end{proposition}
\begin{proof}
Donald gave examples of tilings for $\mathcal{S} = (5,2,3), \, (9,3,4)$ and $(11,4,4)$.  Other tilings are constructed as follows:
\begin{itemize}
\item[(i)] For odd $\kappa \ge 3$ and $\rho = (3\kappa-3) \slash 2$, add $3(\kappa-1) \slash 2$ leaves to the perimeter of a \mbox{$(\kappa-1)\slash 2 \times (\kappa-1)\slash 2 \times
(\kappa-1)\slash 2$} hexagon.

\item[(ii)] For odd $\kappa \ge 3$ and $\rho = (3\kappa-1) \slash 2$, add $(\kappa+1) \slash 2$ leaves to the perimeter of a
\mbox{$(\kappa-1)\slash 2 \times (\kappa-1)\slash 2 \times
(\kappa-1)\slash 2$} hexagon.

\item[(iii)] For even $\kappa \ge 2$ and $\rho = (3\kappa-2) \slash 2$, add $\kappa$ leaves to the perimeter of a \mbox{$\kappa\slash 2 \times \kappa\slash 2 \times
(\kappa-2) \slash 2$} hexagon.\footnote{For $\mathcal{S}=(3,2,2)$ this gives a $1 \times 1 \times 0 $ `hexagon', which should be interpreted as a $1 \times 1$ rhombus, i.e.\ a single tile.}

\item[(iv)] For even $\kappa \ge 2$  and $\rho = 3\kappa \slash 2$, any lozenge tiling of a \mbox{$\kappa \slash 2 \times \kappa\slash 2 \times
\kappa\slash 2$} hexagon is a $\tau.\rho_{\kappa}$ tiling.
\end{itemize}

Finally, to show that no $26.7_6$ tiling exists,
consider the fragment that results from removing the leaves. This has 16 tiles and, from Lemma~\ref{lem:perim1} its perimeter is 14. From Lemma~\ref{lem:minkappa0}
this is the minimum possible perimeter for a fragment of 16 tiles. Any fragment of more than 16 tiles has perimeter greater than 14 and the shape of the 16-tile fragment is
therefore uniquely determined as a $3 \times 2 \times 2$ hexagon.
The three runs that start and end
on the two sides of the hexagon of length 3 consist of only four tiles and so a leaf tile must be added at each end of each run. But the maximum number of leaf tiles
that can be attached to a side of length 3 is two (see Figure~\ref{fig:lammax4}) and hence no $26.7_6$ tiling exists.

\end{proof}

Figure~\ref{fig:kdoku_keven} shows an example of each of the constructions (i)--(iv). Constructions (i) and (iv) are the constructions that Donald gave for odd and even values of $\kappa$.
It is less clear whether he appreciated the generality of (ii) and (iii), but he certainly constructed specific examples -- the quindoku tilings
dap19.7a and dap19.7b (Appendix D) and the quadridoku tilings dap12.5a and dap12.5b (Appendix C).

\begin{figure}[h!]
% Source code in dap_quadridoku_tilings.R
   \centering
   \makebox[\textwidth][c]{\includegraphics[trim=0cm 17.2cm 0cm 0.01cm, clip, width=0.9\textwidth]{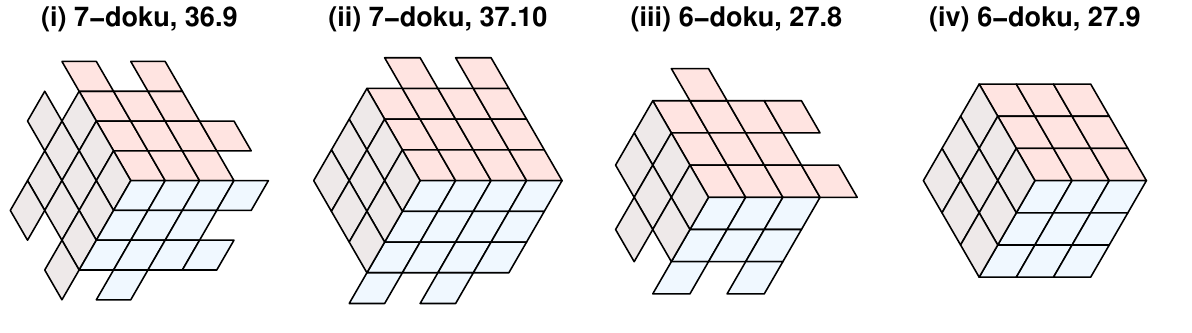}}
   \caption{Examples of the constructions (i)--(iv) given in the proof of Proposition~\ref{prop:turan2} for $\kappa=6$ and $\kappa=7$.}
   \label{fig:kdoku_keven}
 \end{figure}

\subsection{Classification of the smallest $\kappa$-doku tilings} \label{sect:kappastruct}

We have seen that the $\kappa$-doku tilings with the fewest tiles consist of a hexagonal shape, possibly with additional leaf tiles, as illustrated by constructions (i), (iii) and (iv) above.
It is of interest to determine the number of non-isomorphic tilings of each type and in this Section we summarise a few known results.

A regular hexagon of side $a$ can be filled with tredoku tiles in more interesting ways than that shown in Figure~\ref{fig:kdoku}(a), as for example in Figure~\ref{fig:kdoku}(d). Because any such tilings are flip-equivalent, they will always have equal numbers of the three tile types, B, D and F. David \& Tomei (1989) give an elegant visual proof of this result, formulating the
problem as one of fitting calissons\footnote{A calisson is a French sweet that has the shape of a lozenge tile.} into a hexagonal tin.
They show that lozenge tilings of a hexagon with opposite sides of lengths $a$, $b$ and $c$ are in bijection with plane partitions contained in an $a \times b \times c$ box,\footnote{Donald's notes
indicate that he was aware of this connection.} the number of which is given by a famous formula of MacMahon (see e.g.\ Gorin, 2021, Theorem 1.1), namely
\[
  \prod_{i=1}^a \prod_{j=1}^b \prod_{k=1}^c \frac{i+j+k-1}{i+j+k-2}.
\]
When $a = b = c$ this reduces to
\[
       \prod_{i = 0}^{a-1} \frac{i! (i+2n)!}{(i+n)!^2},
\]
and generates the sequence that begins $2, 20, 980, 232848, \ldots$ and appears in the Online Encyclopedia of Integer Sequences (OEIS) as sequence A008793. So this is the number of $\kappa$-doku tilings with $\kappa = 2a$ that fill a regular hexagon of side $a = 1, 2, 3, 4, \ldots$. The number of isomorphism classes is also known (Taylor, 2015) and appears as sequence A066931 in
OEIS. The sequence begins $1, 6, 113, 20174, \ldots$. Note that for tilings that fill a regular hexagon, two tilings are isomorphic if and only if they are congruent. Figure~\ref{fig:quadiso6}
shows a representative of each of the 6 isomorphism classes of 12.6 quadridoku tilings that lie within a regular hexagon of side $a=2$.

\begin{figure}[h!]
   \centering
   \makebox[\textwidth][c]{\includegraphics[trim=0cm 19.5cm 0cm 0.1cm, clip, width=0.999\textwidth]{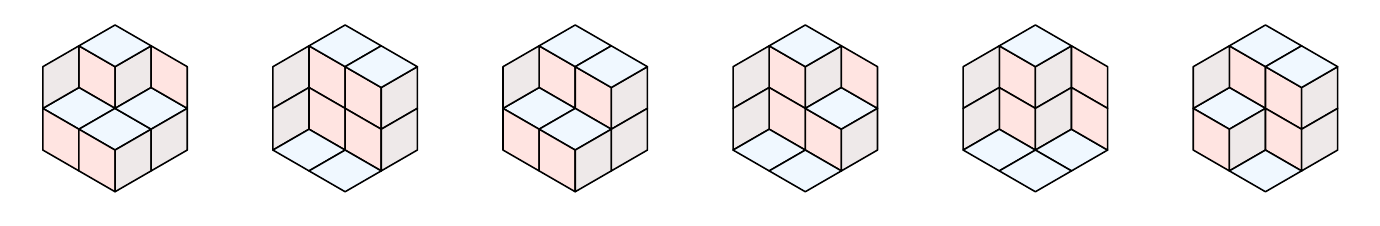}}
   \caption{Representatives of the 6 isomorphism classes of 12.6 quadridoku tilings that lie in a regular hexagon of side 2. The isomorphism classes have size 1, 2, 2, 3, 6, 6.}
   \label{fig:quadiso6}
 \end{figure}

As noted above, we can add $a$ leaves to a hexagonal $2a$-doku tiling to obtain a \mbox{$(2a+1)$-doku} tiling.
For $a=1$, this results in the two non-isomorphic tredoku tilings dap6.3a and dap6.3b.

For $a=2$, the number of non-isomorphic tilings that are obtained depends on the isomorphism
class of the quadridoku tiling. A computer search indicates that the numbers of non-isomorphic quindoku tilings generated by the
six isomorphism classes of quadridoku tilings shown in Figure~\ref{fig:quadiso6}
are, from left to right, 7, 9, 9, 15, 23 and 23, giving a total of 86 isomorphism classes of 18.6 quindoku tilings of this form. Two of these appear in Donald's small collection of quindoku tilings
(Appendix D).

\section{Holes} \label{sect:holes}

Although Donald focused almost entirely on tilings without holes, from the outset tredoku puzzles have sometimes been based on tilings with holes. Figure~\ref{fig:hole2} shows some examples.
This section provides a brief introduction to tilings with holes. Donald's notes contain just six tilings with holes, four tredoku tilings and two quadridoku tilings, and these are shown at the end of this section.

\begin{figure}[h!]
   \centering
   \makebox[\textwidth][c]{\includegraphics[trim=0in 6.5in 0in 0.3in, clip, width=0.94\textwidth]{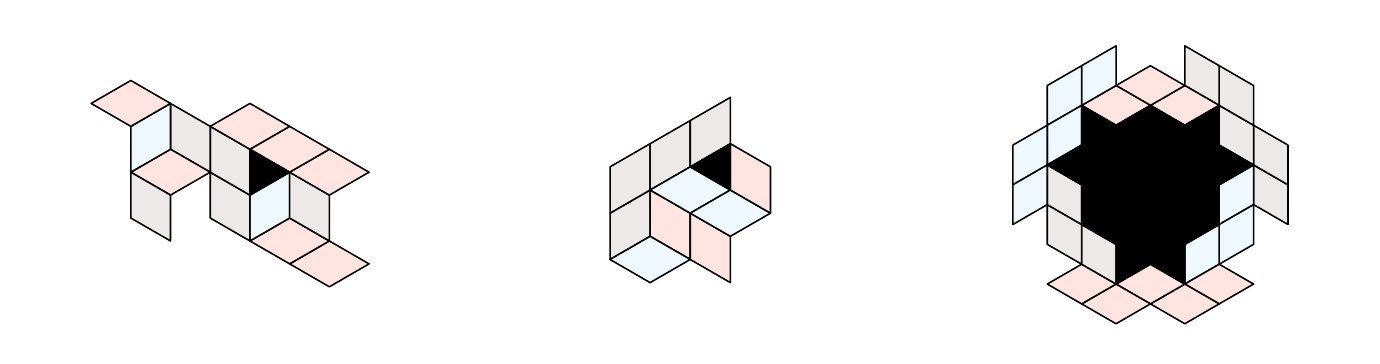}}
   \caption{Examples of tilings with holes that underlie published tredoku puzzles. From left to right, a 14.8 tiling from launch publicity for tredoku (no longer available on the web), a 10.6 tiling from the cover of a book of tredoku puzzles (Mindome 2010), and a 24.24 2-doku tiling from the back cover of a book of puzzles for children (Mindome 2013).}
   % Source: figs2.R
   \label{fig:hole2}
\end{figure}

We extend the notation for tilings to $\tau.\rho.\eta$, where $\eta$ is the number of holes and may be omitted
if $\eta=0$, for consistency with the earlier notation.
Several of the results that we have described earlier no longer hold for tilings with holes, or at least require modification, as discussed below.

\subsection{Existence}

Theorem~\ref{thm:thm1} requires some modification when holes are allowed. The inequalities~(\ref{eq:inequality1}) continue to hold, but only if we adopt Blackburn's (2024) definition of a leaf as a tile that appears in only a single run of length 3, rather than Donald's definition as a tile that shares an edge with only one other tile.

The 11.5.3 tiling shown in Figure~\ref{fig:singlerun} provides an example of the distinction. According to Donald's definition, there is a single leaf, tile 11. But according Blackburn's definition, tiles 2, 4, 8, 9 and 10 are
also leaves, with one edge bordering a
hole and the opposite edge on the outer boundary of the tiling, as is tile 3, which has two parallel edges bordering holes. This gives a total of $\lambda=7$ leaves, so that
$\lambda = 2\tau - 3\rho$, in agreement with~(\ref{eq:leaves}).

\begin{figure}[h!]
   \centering
   \makebox[\textwidth][c]{\includegraphics[trim=0cm 16cm 0cm 0.1cm, clip, width=0.95\textwidth]{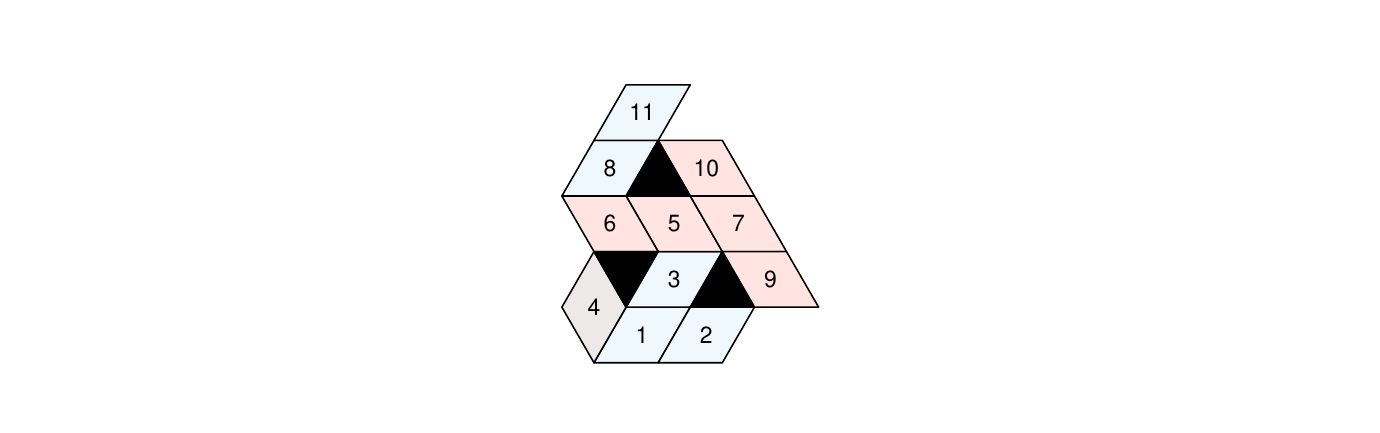}}
   \caption{An 11.5.3 tredoku tiling. Seven of the 11 tiles occur in only a single run of length 3, the other four occur in two runs}
   % figure is eleven_from7[149,]
   % Source: figs2.R
   \label{fig:singlerun}
 \end{figure}

A computer search suggests that it remains the case that no 5.3, 6.4 or 12.8 tiling exists, even if holes are allowed.
However, part (ii) of Theorem~\ref{thm:thm1} is no longer true when holes are allowed. Verdant tredoku tilings, with $\tau=2 \rho + 1$, exist for all values of $\rho \ge 2$.
We have already seen examples  without holes for $\rho = 2, 3, 4$ (Figure~\ref{fig:exist1}) and
Figure~\ref{fig:holes4} shows simple constructions for $\rho \ge 5$. The left panel shows an 11.5.3 tiling that has been extended by adding the four highlighted tiles to the topmost tile to give a 15.7.4 tiling. This addition can be repeated as often as
required to give the sequence 11.5.3, 15.7.4, 19.9.5, $\ldots$. The right panel is similar, but starts with a 13.6.3 tiling and gives the sequence 13.6.3, 17.8.4, 21.10.5, $\ldots$.

\begin{figure}[h!]
   \centering
   \makebox[\textwidth][c]{\includegraphics[trim=0cm 14cm 0cm 2cm, clip, width=0.85\textwidth]{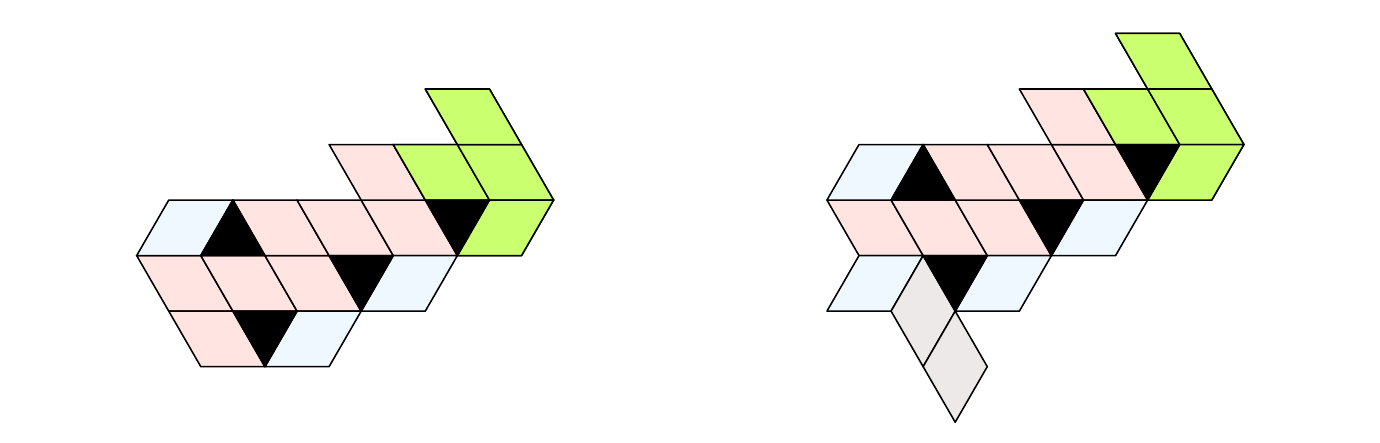}}
   \caption{Constructions used to generate a verdant tredoku tiling with $\tau=2 \rho + 1$ for all $\rho \ge 5 $ when holes are allowed.}
   % figure is eleven_from7[149,]
   % Source: figs2.R
   \label{fig:holes4}
 \end{figure}

What if we insist on tilings having at least one hole?
For example, for which (necessarily even) values of $\rho$ does a leafless tiling exist, with $\tau=3\rho/2$ tiles, if the tiling is required to have at least one hole? Figure~\ref{fig:holes5} shows that such tilings exist
for $\rho=18$ and $\rho=30$, for example, and by mixing the `arms' of these tilings one can also create 33.22.1 and 39.26.1 tilings. The tiling dap54.36.1 shown later in Figure~\ref{fig:hole1} provides another example of a leafless tiling, and Figure~\ref{fig:holes10} includes some further examples.
But it remains to fully characterise the values of $\tau$ and $\rho$ for which a leafless tiling
exists, and more generally to develop a version of the existence theorem that applies to tilings that have at least one hole.

\begin{figure}[h!]
   \centering
   \makebox[\textwidth][c]{\includegraphics[trim=0cm 13cm 0cm 1cm, clip, width=0.85\textwidth]{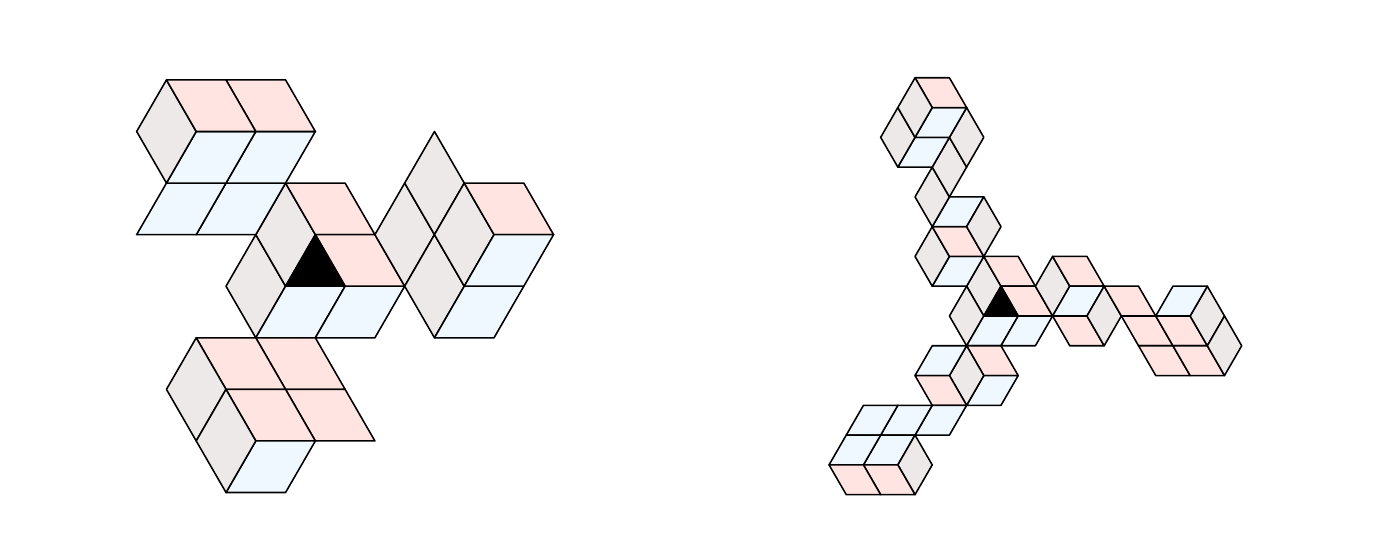}}
   \caption{Examples of leafless tilings with a single small hole, a 27.18.1 tiling on the left and a 45.30.1 tiling on the right.}
   % Source: figs2.R
   \label{fig:holes5}
 \end{figure}

\subsection{Flip-equivalence}

We saw in Section~\ref{sect:flip} that for tilings without holes, the number of alternative tilings that cover the same region of the plane is given by the determinant of the Kastelyn
matrix. When the tiling has holes, one instead needs to calculate the permanent of this matrix.

We also saw that any tiling that is flip-equivalent to a tredoku tiling is itself a tredoku tiling. This is no longer guaranteed if the tiling has holes.
The smallest counterexamples involve 10 tiles and two such examples are shown in
Figure~\ref{fig:holes8}. The left panel shows a 10.5.1 tiling with a single triangular hole. The adjacent panel shows the only other way of tiling the same region of the plane, but this is not a tredoku tiling because it has a run
of length 5. The right two panels provide a similar example with two holes.

\begin{figure}[h!]
   \centering
   \makebox[\textwidth][c]{\includegraphics[trim=0in 7.2in 0in 0.2in, clip, width=1\textwidth]{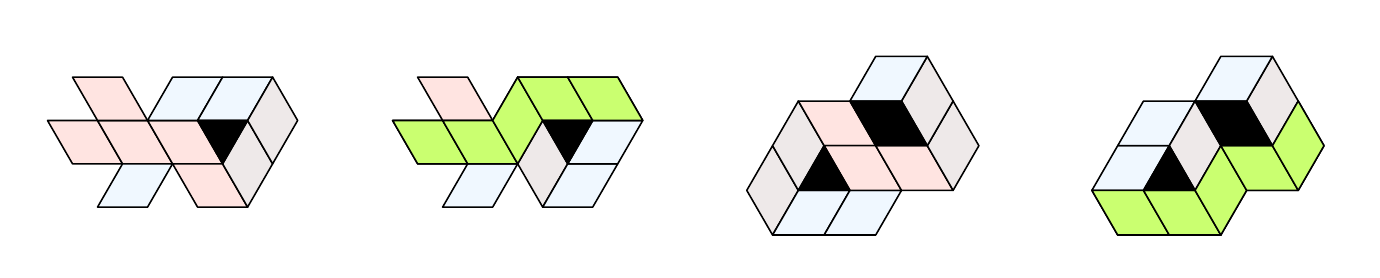}}
   % Previously figure was called kast3
   %\captionsetup{format=hang}
   \caption{Examples of tredoku tilings for which the only alternative tiling of the same region of the plane is not a tredoku tiling. Runs of length greater than
   3 are shown in green.}
   \label{fig:holes8}
\end{figure}

\subsection{Examples of tilings with holes}

Table~\ref{tab:count4} gives a computer-based enumeration of isomorphism classes of tredoku tilings with up to 12 tiles, classified by the number of holes that they contain.
\begin{table}[h!]
\begin{center}
\caption{Number of isomorphism classes of tredoku tilings with up to 12 tiles, classified by the number of holes that they contain.}
\label{tab:count4}
\begin{tabular}{crrrrr}
\hline
  & \multicolumn{4}{c}{Number of holes} \\
 Tiling $\qquad$ & 0 $\qquad$ & 1 $\qquad$ &  2 $\qquad$ & 3 $\qquad$  & Total $\qquad$ \\
\hline
 5.2 $\qquad$ &  1 $\qquad$   &  $\qquad$    &  $\qquad$  &  \phantom{158}  $\qquad$ & 1 $\qquad$ \\[0.75ex]
 6.3 $\qquad$ &  2  $\qquad$ &  1  $\qquad$&  $\qquad$ &    $\qquad$ & 3  $\qquad$ \\[0.75ex]
 7.3 $\qquad$ &  1   $\qquad$ &  7  $\qquad$ &  1 $\qquad$ &    $\qquad$ & 9 $\qquad$ \\
 7.4 $\qquad$ &  4   $\qquad$ &     $\qquad$ &   $\qquad$ &    $\qquad$ & 4 $\qquad$ \\[0.75ex]
 8.4 $\qquad$ &  7   $\qquad$ &  4  $\qquad$ &   $\qquad$ &    $\qquad$ & 11 $\qquad$ \\
 8.5 $\qquad$ &  3   $\qquad$ &   $\qquad$ &   $\qquad$ &    $\qquad$ & 3 $\qquad$ \\[0.75ex]
 9.4 $\qquad$ &  1   $\qquad$ &  34 $\qquad$ & 3  $\qquad$ &    $\qquad$ & 38 $\qquad$ \\
 9.5 $\qquad$ &  14   $\qquad$ &  2 $\qquad$ &   $\qquad$ &    $\qquad$ & 16 $\qquad$ \\
 9.6 $\qquad$ &  1   $\qquad$ &   $\qquad$ &   $\qquad$ &    $\qquad$ & 1 $\qquad$ \\[0.75ex]
 10.5 $\qquad$ &  14   $\qquad$ & 57  $\qquad$ &  4 $\qquad$ &    $\qquad$ & 75 $\qquad$ \\
 10.6 $\qquad$ &  23   $\qquad$ &  1  $\qquad$ &   $\qquad$ &    $\qquad$ & 24 $\qquad$ \\[0.75ex]
 11.5 $\qquad$ &  1   $\qquad$ &  143 $\qquad$ & 139  $\qquad$ &  5  $\qquad$ & 288 $\qquad$ \\
 11.6 $\qquad$ &  44   $\qquad$ & 58  $\qquad$ &  2 $\qquad$ &    $\qquad$ & 104 $\qquad$ \\
 11.7 $\qquad$ &  1   $\qquad$ &   $\qquad$ &   $\qquad$ &    $\qquad$ & 1 $\qquad$ \\[0.75ex]
 12.6 $\qquad$ &  37   $\qquad$ & 458  $\qquad$ & 158  $\qquad$ &  5  $\qquad$  & 658 $\qquad$ \\
 12.7 $\qquad$ &  61   $\qquad$ & 45  $\qquad$ &   $\qquad$ &    $\qquad$  & 106 $\qquad$ \\
\hline
\end{tabular}
\end{center}
\end{table}
\ \\[-8ex]

For given values of the parameters $\tau$, $\rho$ and $\eta$, tilings may exist with holes of different areas, as illustrated by the top row of Figure~\ref{fig:holes7}, which shows three 7.3.1 tilings, each with a different sized hole. Indeed, hole area is not necessarily preserved, even in isomorphic tilings, as illustrated in the bottom row of Figure~\ref{fig:holes7}, which shows two 9.4.1 tilings which are clearly isomorphic but have holes of different area.

\begin{figure}[h!]
   \centering
   \makebox[\textwidth][c]{\includegraphics[trim=0cm 14cm 0cm 0.01cm, clip, width=0.85\textwidth]{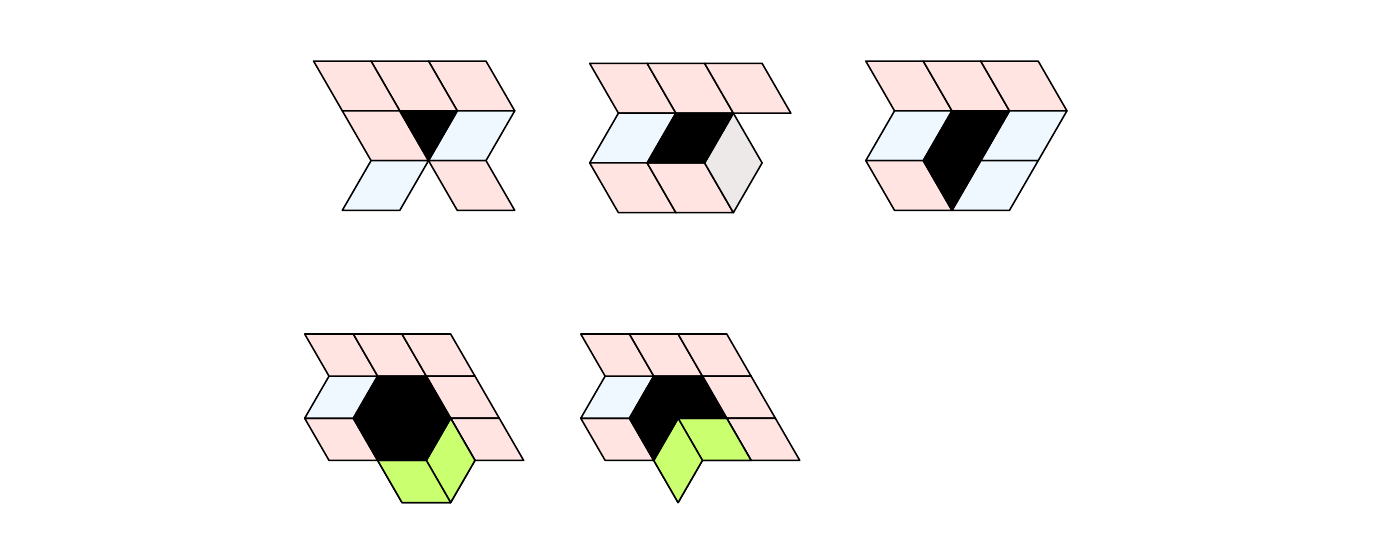}}
   \caption{Examples of tilings with holes of differing area. Top row, three 7.3.1 tilings. Bottom row, two 9.4.1 tilings that are clearly isomorphic, because they differ only in the tiles
   highlighted in green.}
   % Source: figs2.R
   \label{fig:holes7}
 \end{figure}

However, the total perimeter, $\pi$, of a $\tau.\rho.\eta$ tiling, i.e.\ the sum of the outer perimeter and the perimeters of any holes, is fixed and is given by the formula
$  \pi =  4 (\tau - \rho)$ that appears in the proof of Proposition~\ref{lem:perim1}.

Figures~\ref{fig:holes6}
shows a single example of a $\tau.\rho.\eta$ tiling for all combinations of parameter values with $\tau \le 12$ and $\eta \ge 1$ for which a tiling exists. Starting with the smallest tredoku tiling with a hole, the 6.3.1 tiling, which brings to mind a Penrose triangle, several of these tilings can be viewed as representations of `impossible' objects.

\begin{figure}[h!]
   \centering
   \makebox[\textwidth][c]{\includegraphics[trim=0cm 0.01cm 0cm 0.01cm, clip, width=0.9\textwidth]{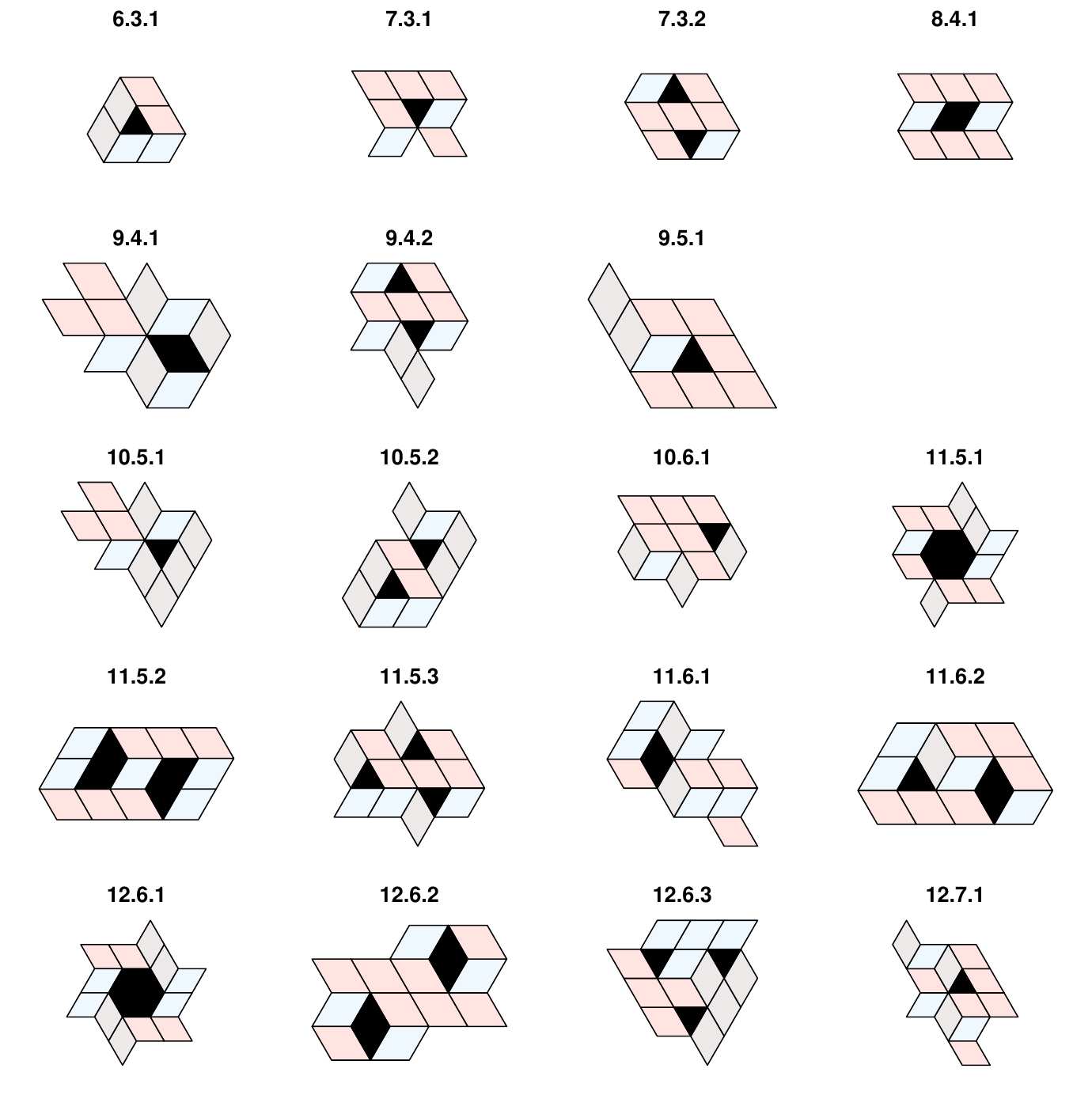}}
   \caption{Examples of tilings with holes.}
   % Source: figs2.R
   \label{fig:holes6}
 \end{figure}

\subsection{The smallest tredoku tiling with a specified number of holes}

What is the smallest number of tiles that a tredoku tiling with $\eta$ holes can have? Malen \& Rold\'an (2020a,\,b) have solved the corresponding problem for polyiamonds. A \define{polyiamond} may be thought of as a finite monohedral tiling that satisfies the properties P1--P3, but where the prototile is an equilateral triangle. In particular, if we partition each tile of a tredoku tiling into its two constituent equilateral triangles, the resulting tiling is a polyiamond.

It seems plausible that the smallest tredoku tiling with a given number of holes arises when each hole is of the smallest possible size, that is an equilateral triangle consisting of half
a tredoku tile. If this is the case, the following result leads to a lower bound on $\tau$.

\begin{proposition}
If $\mathcal{T}$ is a tredoku tiling with $\tau$ tiles and $\eta$ holes that are equilateral triangles of minimal size, then
\begin{equation}
2 \left \lfloor \frac{\tau+1}{2} \right \rfloor \ge  \left \lceil \frac{1}{2}\left (\eta  + \sqrt{6 (2\tau+ \eta)} \right) \right \rceil + \eta.
\label{eq:inhole1}
\end{equation}
\end{proposition}

\begin{proof}
The exterior perimeter of the tiling is
 $\pi_{\mathrm{ext}} = 4(\tau-\rho) - 3 \eta$. Since $\rho \ge \left \lceil (\tau-1)\slash 2  \right \rceil$, from (\ref{eq:inequality1}),
this implies
\[
\pi_{\mathrm{ext}} \le 4 \left \lfloor \frac{\tau+1}{2} \right \rfloor - 3\eta.
\]
Now take the polyiamond derived from $\mathcal{T}$ and fill in the holes. This gives a polyiamond consisting of $n = 2 \tau + \eta$ triangular tiles that has the same external perimeter
as $\mathcal{T}$. Harary \& Harborth (1976) show that the minimum perimeter of a polyiamond with $n$ tiles is
\[
   p_{\mathrm{min}}(n) = 2 \left \lceil \frac{1}{2}\left (n + \sqrt{6n} \right) \right \rceil - n.
\]
Thus, the external perimeter must satisfy the bounds
\begin{equation}
p_{\mathrm{min}}(2 \tau + \eta) \le \pi_{\mathrm{ext}} \le  4 \left \lfloor \frac{\tau+1}{2} \right \rfloor - 3\eta .
\label{eq:inhole2}
\end{equation}
After some simplification, this leads to the inequality~(\ref{eq:inhole1}).
\end{proof}
Let $\tau^{*}(\eta)$ denote the smallest value of $\tau$ satisfying inequality~(\ref{eq:inhole1}), and $\tau_{\mathrm{min}}(\eta)$ the minimum value of $\tau$ for which a tredoku tiling can
be constructed with $\tau$ tiles and $\eta$ holes.
The best estimates of $\tau_{\min}(\eta)$ that I have been able to obtain for $\eta = 1, \ldots, 6$ are as follows:
\begin{center}
\begin{tabular}{rcccccc}
\hline\\[-2.8ex]
$\eta$                       & 1  & 2  & 3  & 4  & 5  & 6  \\
\hline\\[-2.5ex]
$\tau^{*}(\eta)$  & 5  & 7  & 11 & 13 & 15 & 17 \\
\hline\\[-2.5ex]
$\tau_{\mathrm{min}}(\eta)$  & 6  & 7  & 11 & 13 & 17 & 18 \\
\hline
\end{tabular}
\end{center}

These are based on the tilings shown in Figure~\ref{fig:bestholes}.
We have already seen that the smallest tredoku tiling with a hole has six tiles.
For $\eta = 2, 3, 4$, tilings exists with $\tau_{\mathrm{min}}(\eta) = \tau^{*}(\eta)$, but it may be possible to obtain tilings with fewer tiles
for $\eta = 5, 6$.

\begin{figure}[h!]
   \centering
   \makebox[\textwidth][c]{\includegraphics[trim=0in 7.6in 0in 0in, clip, width=0.99\textwidth]{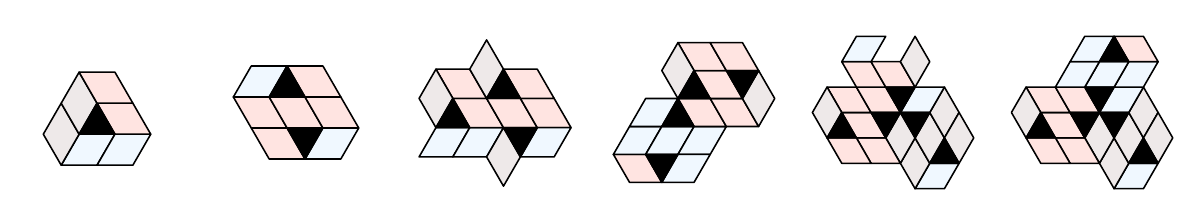}}
   \caption{Tilings with $\eta$ triangular holes ($\eta = 1, \ldots, 6$) that have few tiles. Tilings with fewer tiles do not exist for $\eta \le 4$, but it may be possible to find solutions
   for $\eta = 5, 6$ that use fewer tiles.}
   \label{fig:bestholes}
\end{figure}

\subsection{Donald's tilings with holes}\label{dapholes}

Figure~\ref{fig:hole1} shows the four tredoku tilings with holes that Donald produced. The first three of these share a common hexagonal core, as shown in the top row of the Figure. The tiling dap60.36.1
is particularly interesting as it can be used as a prototile for tiling the entire plane (with holes); Figure~\ref{fig:hole228_144}, shows a 228.144.7 tiling that comprises seven of these prototiles.
A tredoku puzzle based on dap60.36.1 exists online.\footnote{\url{https://pbs.twimg.com/profile_images/507609945232007168/hx3KigKB_400x400.jpeg}}
I don't know whether this image, which dates from September 2014, has any connection with Donald's work.

\begin{figure}[h!]
   \centering
   \makebox[\textwidth][c]{\includegraphics[trim=0in 2.6in 0in 0in, clip, width=0.99\textwidth]{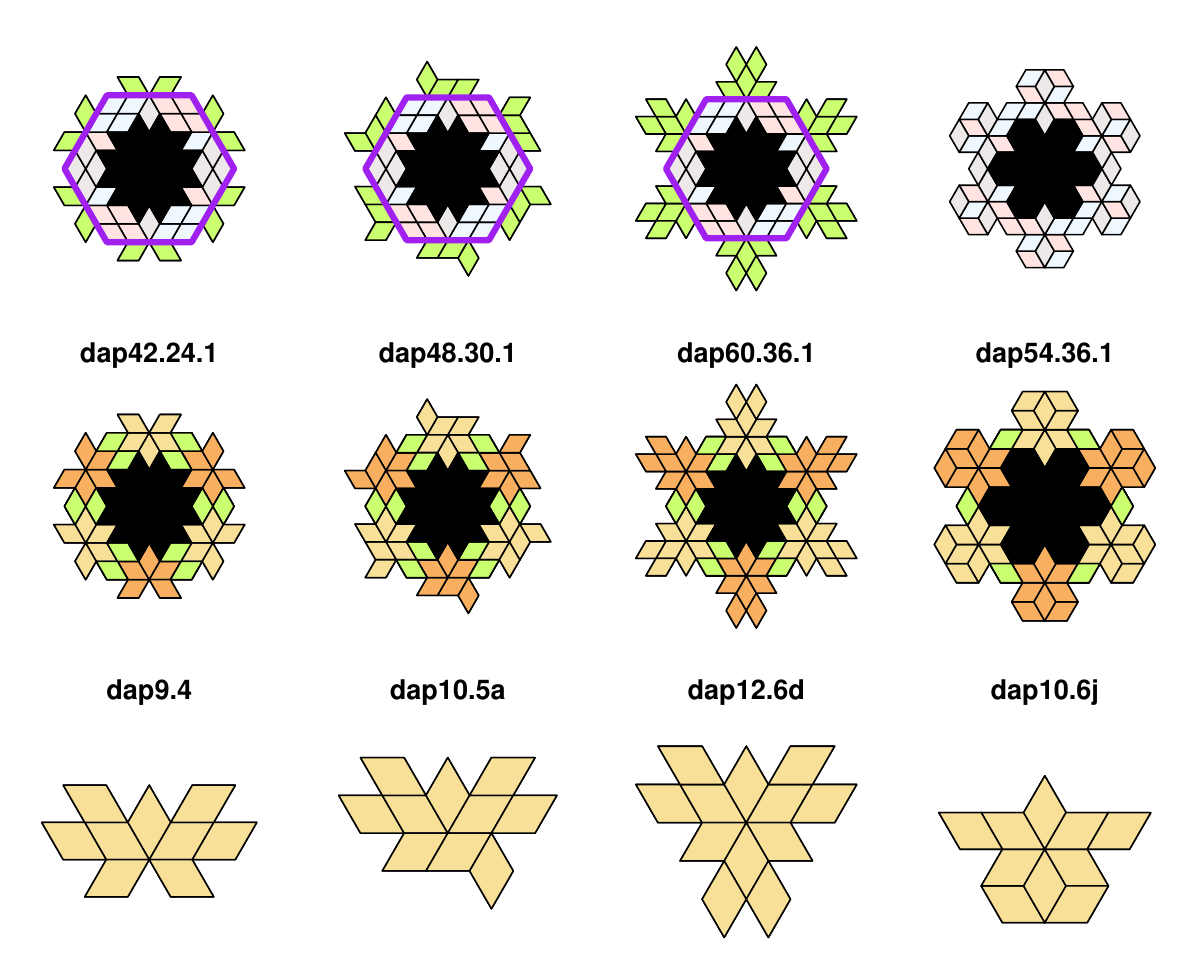}}
   \caption{The four tredoku tilings with holes that appear in Donald's notes. The top row shows the tilings, emphasising the common hexagonal core of the first three tilings.
   The second row shows how the tilings are composed of six copies of a simpler tiling (displayed in the third row), shown in alternating light and dark orange, with overlapping
   tiles shown in green.}
   \label{fig:hole1}
\end{figure}

The tilings in Figure~\ref{fig:hole1} consist of six copies of a simpler tredoku tiling, shown in the bottom row, that are single or double merged, and this is probably how Donald created
the tilings. Many other tilings with holes can be created in a similar way and Figure~\ref{fig:holes10} shows some examples.

\begin{figure}[h!]
   \centering
   \makebox[\textwidth][c]{\includegraphics[trim=0in 0.9in 0in 0in, clip, width=0.45\textwidth]{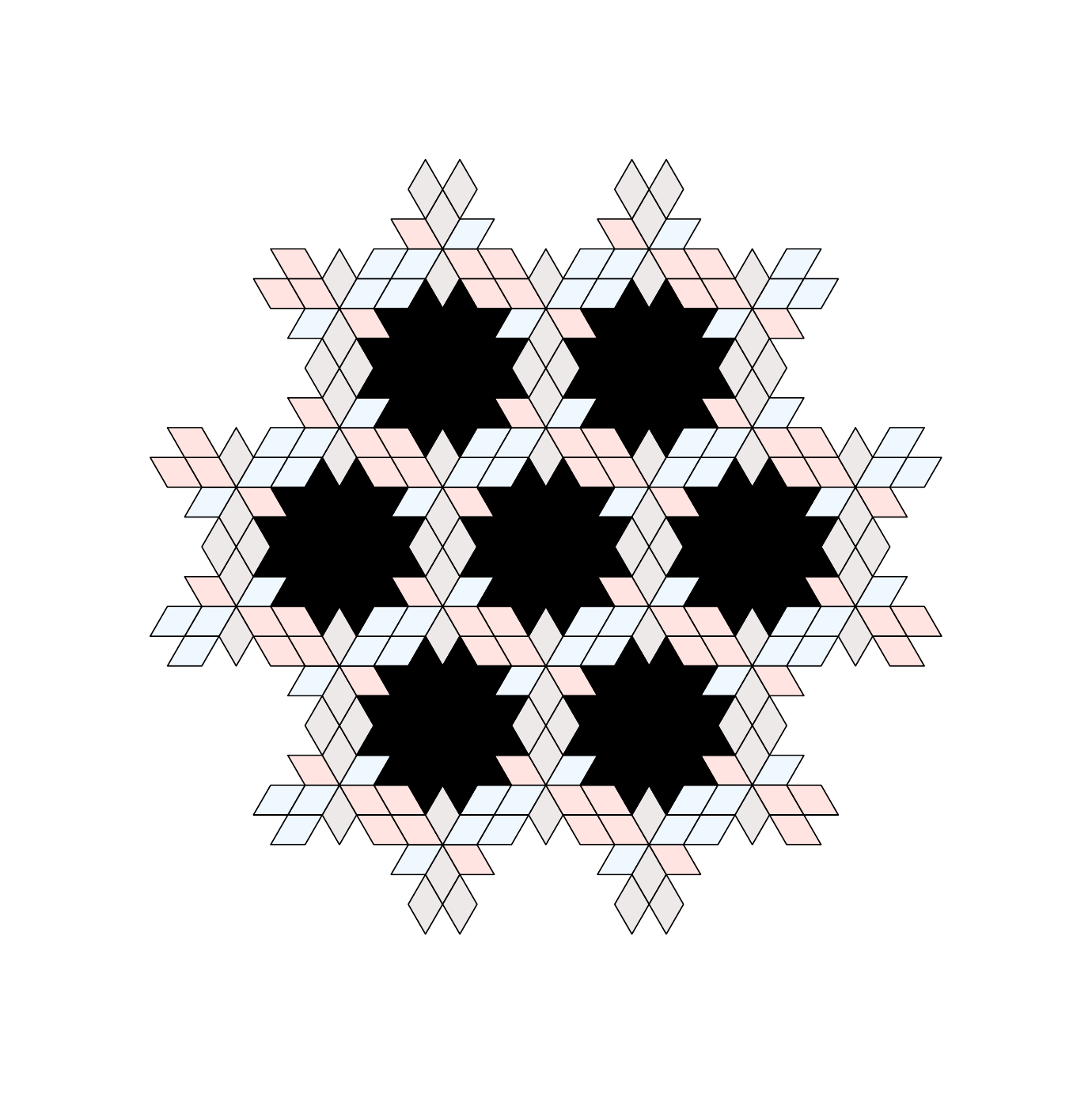}}
   \caption{A 228.144.7 tiling made from seven copies of the tiling dap60.36.1.\ \\}
   \label{fig:hole228_144}
\end{figure}

\begin{figure}[h!]
   \centering
   \makebox[\textwidth][c]{\includegraphics[trim=0in 4.7in 0in 0.01in, clip, width=0.85\textwidth]{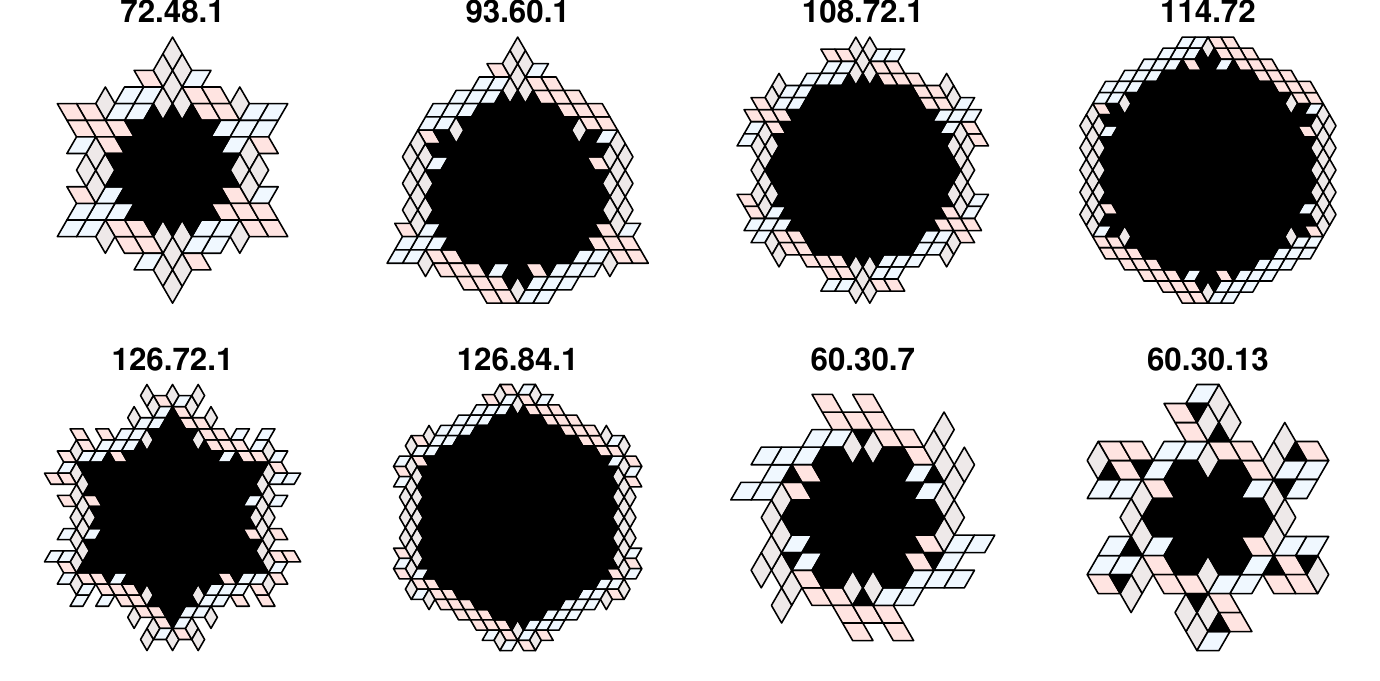}}
   % Previously figure was called kast3
   \caption{Some further examples of tredoku tilings with holes that are created by merging smaller tilings.\ \\}
   \label{fig:holes10}
\end{figure}

Finally, Figure~\ref{fig:quadholes} shows Donald's two quadridoku tilings with holes. The tiling dap66.30.1 is created by single merging six copies of dap12.5a and dap114.54.1
by double merging six copies of dap21.9h.
\begin{figure}[h!]
   \centering
   \makebox[\textwidth][c]{\includegraphics[trim=0in 4.7in 0in 0.01in, clip, width=0.6\textwidth]{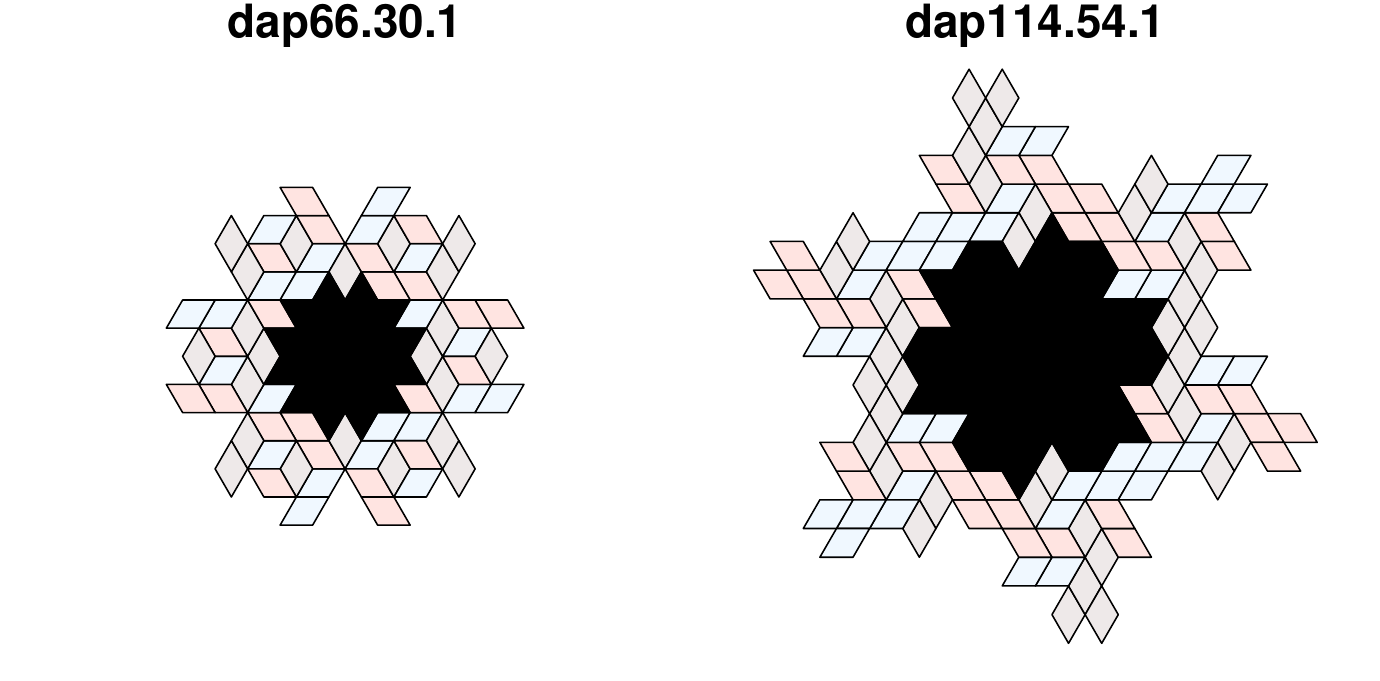}}
   \caption{The two quadridoku tilings with holes that appear in Donald's notes.}
   \label{fig:quadholes}
\end{figure}

\section{Conclusion}\label{sect:discuss}

Donald would undoubtedly have been delighted to see Simon Blackburn's proof of his conjecture about the existence of tredoku tilings. Although this completed a key aspect of
 Donald's work, there remain many avenues for further investigation. Here are some specific possibilities:
\begin{enumerate}
\item  The enumerations in Section~\ref{sect:enum} would benefit from independent verification. Can better counting techniques be developed that would extend to
larger tilings? This would seem to be essential for enumerating $\kappa$-doku tilings with $\kappa > 3$.

\item  What general constructions, analogous to those in Appendix A or in Blackburn (2024), are available for $\kappa$-doku tilings with $\kappa > 3$. For quadridoku tilings, we saw
that 4-merging is a new way to combine existing tilings. What other new possibilities emerge as $\kappa$ increases?

\item Can any of the proofs of the non-existence of a 12.8 tredoku tiling be adapted to show the non-existence of 18.9 and 22.11 quadridoku tilings?
More generally, is Donald's conjecture about the existence of quadridoku tilings (Conjecture~\ref{conj:quadexist}) correct?

\item Theorem \ref{eq:thm3} gives bounds on the number of tiles, $\tau$, of a $\kappa$-doku tiling with $\rho$ runs. Is there a simpler proof? For tredoku tilings and also quadridoku tilings, if Donald's conjecture is correct, tilings exist for most values of $\tau$, except at the upper limit. Is this true also of general $\kappa$-doku tilings?

\item Given the strong constraints on $2$-doku tilings, what can be said about their structure and properties?

\item What else can run graphs reveal about the structure of tredoku tilings? What about run graphs of general $\kappa$-doku tilings?

\item We noted in Section~\ref{dapholes} that the tiling dap60.36.1 can serve as a prototile for a tiling of the plane, with holes. For what other tilings with holes is
this possible?

\end{enumerate}

\ \\[3ex]
\textbf{\Large References}
\ \\

%Bicyclic Ma, Shi, Wang, Yue On Weiner polarity index
%Tricyclic Fang, Cai, Li Maximum detour index of tricyclic graphs

Aigner, M. (1995) Tur\'an's graph theorem. \emph{The American Mathematical Monthly}, \textbf{102}, 808--816.

Bailey, R.A. (2014). Donald Arthur Preece: A life in statistics, mathematics and music. \textit{arXiv} 1402.2220v1, 10 Februrary 2014.

Blackburn, S.R. (2024). Tredoku patterns. \textit{arXiv} 2407.10752v3, 6 December 2024.

Brinkmann, G., Van Cleemput, N. \& Pisanski, T. (2013).  Generation of various classes of trivalent graphs. \textit{Theoretical Computer Science},
\textbf{502}, 16--29.

David, G. \& Tomei, C. (1989). The problem of the calissons. \textit{The American Mathematical Monthly}, \textbf{96}, 429--431.

Frettl\"{o}h \& D., Harriss, E. (2013). Parallelogram tilings, worms, and finite orientations. \emph{Discrete \& Computational  Geometry}, \textbf{49}, 531–-539.

Gorin, V. (2021). \emph{Lectures on Random Lozenge Tilings}. Cambridge University Press.

Gr\"{u}nbaum, B. \& Shephard, G.C. (2016) \emph{Tilings \& Patterns}, 2nd edition. Dover, New York.

Harary, F. \& Harborth, H. (1976). Extremal animals. \textit{Journal of Combinatorics, Information \& System Sciences}, \textbf{1}, 1--8.

Hormann, K. \& Agathos, A. (2001). The point in polygon problem for arbitrary polygons. \textit{Computational Geometry}, \textbf{20}, 131–-144.

Junttila, T. \& Kaski, P. (2007). Engineering an efficient canonical labeling tool for large and sparse graphs. In \emph{Proceedings of the Ninth Workshop on Algorithm Engineering
                and Experiments and the Fourth Workshop on Analytic Algorithms and Combinatorics} New Orleans, USA. D. Applegate, G. Brodal, D. Panario \& R. Sedgewick (eds.),
                pp. 135--149. Society for Industrial \& Applied Mathematics.

Junttila, T. \& Kaski, P. (2011). Conflict propagation and component recursion for canonical labeling. In \emph{Theory and Practice of Algorithms in (Computer) Systems --
               First International {ICST} Conference, {TAPAS}} Rome, Italy. A. Marchetti{-}Spaccamela \& M. Segal (eds.), pp. 151--162. Springer.

Kenyon, R. (1993). Tiling a polygon with parallelograms. \textit{Algorithmica}, \textbf{9}, 382--397.

Malen, G. \& Rold\'an, \'E. (2020a). Polyiamonds attaining extremal topological properties, i. \textit{Geombinatorics}, \textbf{30}, 14--25.

Malen, G. \& Rold\'an, \'E. (2020b). Polyiamonds attaining extremal topological properties, ii. \textit{Geombinatorics}, \textbf{30}, 63--76.

Mindome Games (2010). \emph{Tredoku Medium-hard, Book 1}.  Mindome, Ltd.

Mindome Games (2013). \emph{Tredoku Kids, Book 1}. Mindome, Ltd.

Redelmeier, D.J. (1981). Counting polyominoes: yet another attack. \textit{Discrete Mathematics}, \textbf{36}, 191--203.

Taylor, P.J. (2015) Counting distinct dimer hex tilings. Available at \url{https://cheddarmonk.org/papers/distinct-dimer-hex-tilings.pdf}. Accessed 26/08/25.

Yang, R. \& Meyer, W. (2002). Maximal \& minimal polyiamonds. Technical report, University of Wisconsin-Madison. Available at
\url{https://minds.wisconsin.edu/handle/1793/64366}. Accessed 26/08/25.

\begin{center}
\emph{Appendices A--D are provided as separate ancillary files.}
\end{center}

\end{document}